\documentclass[oneside,english]{amsart}
\usepackage[T1]{fontenc}
\usepackage[latin9]{inputenc}
\usepackage{babel}
\usepackage{amsbsy}
\usepackage{amstext}
\usepackage{amsthm}
\usepackage[unicode=true,pdfusetitle,
 bookmarks=true,bookmarksnumbered=false,bookmarksopen=false,
 breaklinks=false,pdfborder={0 0 1},backref=false,colorlinks=false]
 {hyperref}

\makeatletter
\numberwithin{equation}{section}
\theoremstyle{plain}
\newtheorem{thm}{\protect\theoremname}[section]

\makeatother

\providecommand{\theoremname}{Theorem}

\begin{document}
\title{On an expansion formula of Liu }
\author{Bing He}
\address{School of Mathematics and Statistics, HNP-LAMA, Central South University
\\
Changsha 410083, Hunan, People's Republic of China}
\email{yuhelingyun@foxmail.com; yuhe001@foxmail.com}
\begin{abstract}
In this paper, we extend an expansion formula of Liu to multiple
basic hypergeometric series over the root system $A_{n}.$ The usefulness
of Liu's expansion formula in special functions and number theory
has been shown by Liu and many others. We first establish a very general
multiple expansion formula over the root system $A_{n}$ and then
deduce several $A_{n}$ extensions of Liu's expansion formula. From
these multiple formulas, we derive two groups of multiple expansion
formulas for infinite products. As applications, we deduce an $A_{n}$
Rogers' $\text{}_{6}\phi_{5}$ summation, an $A_{n}$ extension of
Sylvester's identity, some multiple expansion formulas for $(q)_{\infty}^{m},\text{\ensuremath{\pi_{q}}}$
and $1/\pi_{q}$, two $A_{n}$ extensions of the Rogers-Fine identity,
an $A_{n}$ extension of Liu's extension of Rogers' non-terminating
$_{6}\phi_{5}$ summation, an $A_{n}$ extension of a generalization
of Fang's identity and an $A_{n}$ extension of Andrews' expansion
formula.
\end{abstract}

\maketitle

\section{Introduction}

In a battery of papers, Liu \cite{L2,L1,L3} extended many summation
formulas and transformation formulas in the theory of basic hypergeometric
series and proved the following well-known expansion formula \cite[Theorem 1.1]{L2}\cite[Theorem 1.2]{L1}:
\begin{thm}
Suppose that $f(x)$ is an analytic function near $x=0,$ then, under
suitable convergence conditions,
\begin{equation}
\begin{aligned} & \frac{(\alpha q,\alpha ab/q)_{\infty}}{(\alpha a,\alpha b)_{\infty}}f(\alpha a)\\
 & =\sum_{n=0}^{\infty}\frac{(1-\alpha q^{2n})(\alpha,q/a)_{n}(a/q)^{n}}{(1-\alpha)(q,\alpha a)_{n}}\sum_{k=0}^{n}\frac{(q^{-n},\alpha q^{n})_{k}q^{k}}{(q,\alpha b)_{k}}f(\alpha q^{k+1}),
\end{aligned}
\label{eq:10-1}
\end{equation}
where the $q$-shifted factorial $(z)_{\infty}$ is defined by 
\[
(z)_{\infty}:=(z;q)_{\infty}=\prod_{k=0}^{\infty}(1-zq^{k})
\]
 and 
\[
(z)_{n}=\frac{(z)_{\infty}}{(zq^{n})_{\infty}}.
\]
\end{thm}
Here and in what follows, we always assume that $0<|q|<1$ and we
also use the compact notation: 
\[
(a_{1},a_{2},\cdots,a_{n})_{k}=(a_{1})_{k}(a_{2})_{k}\cdots(a_{n})_{k},\;k\;\mathrm{is}\;\mathrm{a}\;\mathrm{nonnegative}\;\mathrm{integer}\;\mathrm{and}\;\infty.
\]
The expansion formula \eqref{eq:10-1} is very useful and was used
by Liu in \cite{L2} to extend Rogers' non-terminating $_{6}\phi_{5}$
summation formula \cite[eq.(2.7.1)]{GR}, to give a new proof of the
orthogonality relation \cite[eq.(7.5.15)]{GR} for the Askey-Wilson
polynomials and to derive an interesting double summation formula
which is equivalent to an identity of Andrews \cite[Theorem 5]{A86}.

Our work belongs to the context of multiple basic hypergeometric series
over the root system $A_{n}.$ Series of this type have been investigated
systematically over the last many years and many important results
in basic hypergeometric series have been extended. See Schlosser \cite{S}
for a comprehensive survey of this type of series. It should be mentioned
that S.C. Milne made important contributions to multiple basic hypergeometric
series over the root system $A_{n}$. See \cite{C1}, \cite{C2},
\cite{F} and \cite{W09} for certain generalizations of Milne's $A_{n}$
$q$-summation formulas. Recently, G. Bhatnager and S. Rai \cite{BR}
extended some expansion formulas to the context of multiple series
over root systems. $A_{n}$ Extensions of Liu's results in \cite{L02,L1}
have not been fully studied yet, even though his resuls are very useful
in special functions and number theory \cite{CL,L2,L1,L3}. Inspired
by their work, we will extend \eqref{eq:10-1} to multiple basic hypergeometric
series over the root system $A_{n}.$

Before stating our results, we need to mention some notations on multiple
basic hypergeometric series over the root system $A_{n}.$ The series
we will discuss are of the form:
\[
\sum_{{k_{r}\geq0\atop r=1,\cdots,n}}S(\boldsymbol{k}),
\]
where $\boldsymbol{k}=(k_{1},\cdots,k_{n})$ and $k_{1},\cdots,k_{n}$
are nonnegative integers. Similarly, $\boldsymbol{j}$ and $\boldsymbol{m}$
are multi-indices of nonnegative integers. This type of series reduces
to classical basic hypergeometric series when $n=1.$ We also use
the notation $|\boldsymbol{k}|:=k_{1}+\cdots+k_{n}.$ Series of this
type are called $A_{n}$ basic hypergeometric series if they only
have 
\[
\prod_{1\leq r<s\leq n}\frac{1-q^{k_{r}-k_{s}}x_{r}/x_{s}}{1-x_{r}/x_{s}}
\]
as a factor of $S(\boldsymbol{k}).$

The main result of this paper is to extend \eqref{eq:10-1} to the
following multiple transformation formula over the root system $A_{n}.$ 
\begin{thm}
\label{t1} Let $K(y)$ and $\beta(j)$ be chosen so that the series
on both sides are analytic in $y_{1},y_{2},\cdots,y_{n}$ and converge
absolutely for each $y_{r}$ in a disk around the origin; or, $y_{r}=q^{N_{r}},$
for $r=1,2,\cdots,n,$ and so the series on both sides terminate.
Then
\begin{equation}
\begin{aligned} & K(\boldsymbol{y})\sum_{{j_{r}\geq0\atop r=1,2,\cdots,n}}\prod_{1\leq r<s\leq n}\frac{1-q^{j_{r}-j_{s}}x_{r}/x_{s}}{1-x_{r}/x_{s}}\prod_{r,s=1}^{n}(x_{r}/x_{s}y_{s})_{j_{r}}\\
 & \quad\times\frac{1-bq^{2|\boldsymbol{j}|}}{(1-b)(bqy_{1}y_{2}\cdots y_{n})_{|\boldsymbol{j}|}}(y_{1}y_{2}\cdots y_{n})^{|\boldsymbol{j}|}q^{\sum_{r=1}^{n}(r-1)j_{r}}\beta(\boldsymbol{j})\\
 & =\sum_{{k_{r}\geq0\atop r=1,2,\cdots,n}}\prod_{1\leq r<s\leq n}\frac{1-q^{k_{r}-k_{s}}x_{r}/x_{s}}{1-x_{r}/x_{s}}\prod_{r,s=1}^{n}\frac{(x_{r}/x_{s}y_{s})_{k_{r}}}{(qx_{r}/x_{s})_{k_{r}}}\\
 & \times\frac{(1-aq^{2|\boldsymbol{k}|})(a)_{|\boldsymbol{k}|}}{(1-a)(aqy_{1}y_{2}\cdots y_{n})_{|\boldsymbol{k}|}}(y_{1}y_{2}\cdots y_{n})^{|\boldsymbol{k}|}q^{\sum_{r=1}^{n}(r-1)k_{r}}
\end{aligned}
\label{eq:t1-1}
\end{equation}
\[
\begin{aligned} & \times\sum_{{0\leq j_{r}\leq k_{r}\atop r=1,2,\cdots,n}}\prod_{1\leq r<s\leq n}\frac{1-q^{j_{r}-j_{s}}x_{r}/x_{s}}{1-x_{r}/x_{s}}\prod_{r,s=1}^{n}(q^{-k_{s}}x_{r}/x_{s})_{j_{r}}\\
 & \;\times\frac{(aq^{|\boldsymbol{k}|},bq)_{|\boldsymbol{j}|}}{(b)_{2|\boldsymbol{j}|}}(-1)^{|\boldsymbol{j}|}q^{\sum_{r=1}^{n}rj_{r}+{|\boldsymbol{j}| \choose 2}}\beta(\boldsymbol{j})\\
 & \;\times\sum_{{0\leq m_{r}\leq k_{r}-j_{r}\atop r=1,2,\cdots,n}}\prod_{1\leq r<s\leq n}\frac{1-q^{m_{r}-m_{s}+j_{r}-j_{s}}x_{r}/x_{s}}{1-q^{j_{r}-j_{s}}x_{r}/x_{s}}q^{\sum_{r=1}^{n}rm_{r}}K(q^{\boldsymbol{m+j}})\\
 & \quad\times\prod_{r,s=1}^{n}\frac{(q^{-k_{s}+j_{r}}x_{r}/x_{s})_{m_{r}}}{(q^{1+j_{r}-j_{s}}x_{r}/x_{s})_{m_{r}}}\frac{(aq^{|\boldsymbol{k}|+|\boldsymbol{j}|},bq^{1+|\boldsymbol{j}|})_{|\boldsymbol{m}|}}{(aq)_{|\boldsymbol{m}|+|\boldsymbol{j}|}(bq^{2|\boldsymbol{j}|+1})_{|\boldsymbol{m}|}}.
\end{aligned}
\]
\end{thm}
The identity \eqref{eq:t1-1} in Theorem \ref{t1} is an $A_{n}$
extension of the following expansion formula:
\begin{equation}
\begin{aligned} & K(y)\sum_{j=0}^{\infty}\frac{(1-bq^{2j})(1/y)_{j}}{(1-b)(bqy)_{j}}y^{j}\beta_{j}\\
 & =\sum_{k=0}^{\infty}\frac{(1-aq^{2k})(1/y,a)_{k}}{(1-a)(q,aqy)_{k}}y^{k}\sum_{j=0}^{k}\frac{(q^{-k},aq^{k},bq)_{j}}{(b)_{2j}}(-1)^{j}q^{{j+1 \choose 2}}\beta_{j}\\
 & \times\sum_{m=0}^{k-j}\frac{(q^{-k+j},aq^{j+k},bq^{j+1})_{m}}{(q,bq^{1+2j})_{m}(aq)_{j+m}}q^{m}K(q^{j+m}).
\end{aligned}
\label{eq:a1-21}
\end{equation}
Take $y_{1}\mapsto y$ in the $n=1$ case of \eqref{eq:t1-1} in Theorem
\ref{t1} to obtain \eqref{eq:a1-21} easily.  Another $A_{n}$ extension
of \eqref{eq:a1-21} can be found in \cite[Theorem 2.2]{BR}.

Theorem \ref{t1} plays an important role in multiple $q$-series
over the root system $A_{n},$ from which many important results can
be derived. In this paper, we will use it to deduce several summation
formulas and transformation formulas for $q$-series over the root
system $A_{n}.$

We first apply the identity in Theorem \ref{t1} to obtain four $A_{n}$
extensions of \eqref{eq:10-1}.
\begin{thm}
\textup{\label{th-1} (First $A_{n}$ extension of \eqref{eq:10-1})
}Let $f(y)$ be such that the series on both sides are analytic in
$a_{1},a_{2},\cdots,a_{n}$ and converge absolutely for each $a_{r}$
in a disk around the origin; or, $a_{r}=q^{N_{r}},$ for $r=1,2,\cdots,n,$
and so the series terminate. Then
\begin{equation}
\begin{aligned} & f(\alpha a_{1}\cdots a_{n}q^{1-n})\frac{(\alpha q,\alpha a_{1}\cdots a_{n}b/q^{n})_{\infty}}{(\alpha a_{1}\cdots a_{n}q^{1-n},\alpha b)_{\infty}}\\
 & =\sum_{{k_{r}\geq0\atop r=1,2,\cdots,n}}\prod_{1\leq r<s\leq n}\frac{1-q^{k_{r}-k_{s}}x_{r}/x_{s}}{1-x_{r}/x_{s}}\prod_{r,s=1}^{n}\frac{(qx_{r}/x_{s}a_{s})_{k_{r}}}{(qx_{r}/x_{s})_{k_{r}}}\\
 & \times\frac{(1-\alpha q^{2|\boldsymbol{k}|})(\alpha)_{|\boldsymbol{k}|}}{(1-\alpha)(\alpha a_{1}\cdots a_{n}q^{1-n})_{|\boldsymbol{k}|}}(a_{1}\cdots a_{n})^{|\boldsymbol{k}|}q^{\sum_{r=1}^{n}(r-1)k_{r}-n|\boldsymbol{k}|}\\
 & \;\times\sum_{{0\leq m_{r}\leq k_{r}\atop r=1,2,\cdots,n}}\prod_{1\leq r<s\leq n}\frac{1-q^{m_{r}-m_{s}}x_{r}/x_{s}}{1-x_{r}/x_{s}}\prod_{r,s=1}^{n}\frac{(q^{-k_{s}}x_{r}/x_{s})_{m_{r}}}{(qx_{r}/x_{s})_{m_{r}}}\\
 & \quad\times\frac{(\alpha q^{|\boldsymbol{k}|})_{|\boldsymbol{m}|}}{(\alpha b)_{|\boldsymbol{m}|}}q^{\sum_{r=1}^{n}rm_{r}}f(\alpha q^{1+|\boldsymbol{m}|}).
\end{aligned}
\label{eq:a1-17}
\end{equation}
\end{thm}
\begin{thm}
\label{th-2} \textup{(Second $A_{n}$ extension of \eqref{eq:10-1})
}Let $f(y_{1},y_{2},\cdots,y_{n})$ be such that the series on both
sides are analytic in $a_{1},a_{2},\cdots,a_{n}$ and converge absolutely
for each $a_{r}$ in a disk around the origin; or, $a_{r}=q^{N_{r}},$
for $r=1,2,\cdots,n,$ and so the series terminate. Then
\begin{equation}
\begin{aligned} & f(\alpha a_{1},\cdots,\alpha a_{n})\frac{(\alpha q,\alpha a_{1}\cdots a_{n}b/q^{n})_{\infty}}{(\alpha b,\alpha a_{1}\cdots a_{n}q^{1-n})_{\infty}}\\
 & =\sum_{{k_{r}\geq0\atop r=1,2,\cdots,n}}\prod_{1\leq r<s\leq n}\frac{1-q^{k_{r}-k_{s}}x_{r}/x_{s}}{1-x_{r}/x_{s}}\prod_{r,s=1}^{n}\frac{(qx_{r}/x_{s}a_{s})_{k_{r}}}{(qx_{r}/x_{s})_{k_{r}}}\\
 & \times\frac{(1-\alpha q^{2|\boldsymbol{k}|})(\alpha)_{|\boldsymbol{k}|}}{(1-\alpha)(\alpha a_{1}\cdots a_{n}q^{1-n})_{|\boldsymbol{k}|}}(a_{1}\cdots a_{n})^{|\boldsymbol{k}|}q^{\sum_{r=1}^{n}(r-1)k_{r}-n|\boldsymbol{k}|}\\
 & \;\times\sum_{{0\leq m_{r}\leq k_{r}\atop r=1,2,\cdots,n}}\prod_{1\leq r<s\leq n}\frac{1-q^{m_{r}-m_{s}}x_{r}/x_{s}}{1-x_{r}/x_{s}}\prod_{r,s=1}^{n}\frac{(q^{-k_{s}}x_{r}/x_{s})_{m_{r}}}{(qx_{r}/x_{s})_{m_{r}}}\\
 & \quad\times\frac{(\alpha q^{|\boldsymbol{k}|})_{|\boldsymbol{m}|}}{(\alpha b)_{|\boldsymbol{m}|}}q^{\sum_{r=1}^{n}rm_{r}}f(\alpha q^{1+m_{1}},\cdots,\alpha q^{1+m_{n}}).
\end{aligned}
\label{eq:a1-18}
\end{equation}
\end{thm}
\begin{thm}
\label{th-3} \textup{(Third $A_{n}$ extension of \eqref{eq:10-1})
}Let $f(y)$ be such that the series on both sides are analytic in
$a_{1},a_{2},\cdots,a_{n}$ and converge absolutely for each $a_{r}$
in a disk around the origin; or, $a_{r}=q^{N_{r}},$ for $r=1,2,\cdots,n,$
and so the series terminate. Then
\begin{equation}
\begin{aligned} & f(\alpha a_{1}\cdots a_{n}q^{1-n})\frac{(\alpha q)_{\infty}}{(\alpha a_{1}\cdots a_{n}q^{1-n})_{\infty}}\prod_{r=1}^{n}\frac{(\alpha a_{r}bx_{r}/q)_{\infty}}{(\alpha bx_{r})_{\infty}}\\
 & =\sum_{{k_{r}\geq0\atop r=1,2,\cdots,n}}\prod_{1\leq r<s\leq n}\frac{1-q^{k_{r}-k_{s}}x_{r}/x_{s}}{1-x_{r}/x_{s}}\prod_{r,s=1}^{n}\frac{(qx_{r}/x_{s}a_{s})_{k_{r}}}{(qx_{r}/x_{s})_{k_{r}}}\\
 & \times\frac{(1-\alpha q^{2|\boldsymbol{k}|})(\alpha)_{|\boldsymbol{k}|}}{(1-\alpha)(\alpha a_{1}\cdots a_{n}q^{1-n})_{|\boldsymbol{k}|}}(a_{1}\cdots a_{n})^{|\boldsymbol{k}|}q^{\sum_{r=1}^{n}(r-1)k_{r}-n|\boldsymbol{k}|}\\
 & \;\times\sum_{{0\leq m_{r}\leq k_{r}\atop r=1,2,\cdots,n}}\prod_{1\leq r<s\leq n}\frac{1-q^{m_{r}-m_{s}}x_{r}/x_{s}}{1-x_{r}/x_{s}}\prod_{r,s=1}^{n}\frac{(q^{-k_{s}}x_{r}/x_{s})_{m_{r}}}{(qx_{r}/x_{s})_{m_{r}}}\\
 & \quad\times\frac{(\alpha q^{|\boldsymbol{k}|})_{|\boldsymbol{m}|}}{\prod_{r=1}^{n}(\alpha bx_{r})_{m_{r}}}q^{\sum_{r=1}^{n}rm_{r}}f(\alpha q^{1+|\boldsymbol{m}|}).
\end{aligned}
\label{eq:a1-19}
\end{equation}
\end{thm}
\begin{thm}
\label{th-4} \textup{(Fourth $A_{n}$ extension of \eqref{eq:10-1})
}Let $f(y_{1},y_{2},\cdots,y_{n})$ be such that the series on both
sides are analytic in $a_{1},a_{2},\cdots,a_{n}$ and converge absolutely
for each $a_{r}$ in a disk around the origin; or, $a_{r}=q^{N_{r}},$
for $r=1,2,\cdots,n,$ and so the series terminate. Then
\begin{equation}
\begin{aligned} & f(\alpha a_{1},\cdots,\alpha a_{n})\frac{(\alpha q)_{\infty}}{(\alpha a_{1}a_{2}\cdots a_{n}q^{1-n})_{\infty}}\prod_{r=1}^{n}\frac{(\alpha a_{r}bx_{r}/q)_{\infty}}{(\alpha bx_{r})_{\infty}}\\
 & =\sum_{{k_{r}\geq0\atop r=1,2,\cdots,n}}\prod_{1\leq r<s\leq n}\frac{1-q^{k_{r}-k_{s}}x_{r}/x_{s}}{1-x_{r}/x_{s}}\prod_{r,s=1}^{n}\frac{(qx_{r}/x_{s}a_{s})_{k_{r}}}{(qx_{r}/x_{s})_{k_{r}}}\\
 & \times\frac{(1-\alpha q^{2|\boldsymbol{k}|})(\alpha)_{|\boldsymbol{k}|}}{(1-\alpha)(\alpha a_{1}a_{2}\cdots a_{n}q^{1-n})_{|\boldsymbol{k}|}}(a_{1}a_{2}\cdots a_{n})^{|\boldsymbol{k}|}q^{\sum_{r=1}^{n}(r-1)k_{r}-n|\boldsymbol{k}|}
\end{aligned}
\label{eq:a1-20}
\end{equation}
\begin{align*}
 & \;\times\sum_{{0\leq m_{r}\leq k_{r}\atop r=1,2,\cdots,n}}\prod_{1\leq r<s\leq n}\frac{1-q^{m_{r}-m_{s}}x_{r}/x_{s}}{1-x_{r}/x_{s}}\prod_{r,s=1}^{n}\frac{(q^{-k_{s}}x_{r}/x_{s})_{m_{r}}}{(qx_{r}/x_{s})_{m_{r}}}\\
 & \quad\times\frac{(\alpha q^{|\boldsymbol{k}|})_{|\boldsymbol{m}|}}{\prod_{r=1}^{n}(\alpha bx_{r})_{m_{r}}}q^{\sum_{r=1}^{n}rm_{r}}f(\alpha q^{1+m_{1}},\cdots,\alpha q^{1+m_{n}}).
\end{align*}
\end{thm}
Two other $A_{n}$ extensions of \eqref{eq:10-1} are given in the
following theorems.
\begin{thm}
\textup{\label{t11-1} (Fifth $A_{n}$ extension of \eqref{eq:10-1})}
Let $f(y)$ be such that the series on both sides are analytic in
$a_{1},a_{2},\cdots,a_{n}$ and converge absolutely for each $a_{r}$
in a disk around the origin; or, $a_{r}=q^{N_{r}},$ for $r=1,2,\cdots,n,$
and so the series terminate. Then
\begin{equation}
\begin{aligned} & f(\alpha a_{1}\cdots a_{n}q^{1-n})\prod_{r=1}^{n}\frac{(\alpha qx_{r},\alpha a_{r}bx_{r}/q)_{\infty}}{(\alpha bx_{r},\alpha a_{r}x_{r})_{\infty}}\\
 & =\sum_{{k_{r}\geq0\atop r=1,\cdots,n}}\prod_{1\leq r<s\leq n}\frac{1-q^{k_{r}-k_{s}}x_{r}/x_{s}}{1-x_{r}/x_{s}}\prod_{r,s=1}^{n}\frac{\left(qx_{r}/x_{s}a_{s}\right)_{k_{r}}}{\left(qx_{r}/x_{s}\right)_{k_{r}}}\\
 & \times\prod_{r=1}^{n}\frac{1-\alpha x_{r}q^{k_{r}+|\boldsymbol{k}|}}{1-\alpha x_{r}}\prod_{r=1}^{n}\frac{\left(\alpha x_{r}\right)_{|\boldsymbol{k}|}}{\left(\alpha x_{r}a_{r}\right)_{|\boldsymbol{k}|}}\left(a_{1}\cdots a_{n}\right)^{|\boldsymbol{k}|}q^{\sum_{r=1}^{n}(r-1)k_{r}-n|\boldsymbol{k}|}\\
 & \times\sum_{{0\leq m_{r}\leq k_{r}\atop r=1,\cdots,n}}\prod_{1\leq r<s\leq n}\frac{1-q^{m_{r}-m_{s}}x_{r}/x_{s}}{1-x_{r}/x_{s}}\prod_{r,s=1}^{n}\frac{\left(q^{-k_{s}}x_{r}/x_{s}\right)_{m_{r}}}{\left(qx_{r}/x_{s}\right)_{m_{r}}}\\
 & \;\times\prod_{r=1}^{n}\frac{\left(\alpha x_{r}q^{|\boldsymbol{k}|}\right)_{m_{r}}}{(\alpha bx_{r})_{m_{r}}}q^{\sum_{r=1}^{n}rm_{r}}f(\alpha q^{1+|\boldsymbol{m}|}).
\end{aligned}
\label{eq:12-1}
\end{equation}
\end{thm}
\begin{thm}
\textup{\label{t11-2} (Sixth $A_{n}$ extension of \eqref{eq:10-1})}
Let $f(y_{1},y_{2},\cdots,y_{n})$ be a function of $n$ variables
such that the series on both sides are analytic in $a_{1},a_{2},\cdots,a_{n}$
and converge absolutely for each $a_{r}$ in a disk around the origin;
or, $a_{r}=q^{N_{r}},$ for $r=1,2,\cdots,n,$ and so the series terminate.
Then
\begin{equation}
\begin{aligned} & f(\alpha a_{1},\cdots,\alpha a_{n})\prod_{r=1}^{n}\frac{(\alpha qx_{r},\alpha a_{r}bx_{r}/q)_{\infty}}{(\alpha bx_{r},\alpha x_{r}a_{r})_{\infty}}\\
 & =\sum_{{k_{r}\geq0\atop r=1,\cdots,n}}\prod_{1\leq r<s\leq n}\frac{1-q^{k_{r}-k_{s}}x_{r}/x_{s}}{1-x_{r}/x_{s}}\prod_{r,s=1}^{n}\frac{\left(qx_{r}/x_{s}a_{s}\right)_{k_{r}}}{\left(qx_{r}/x_{s}\right)_{k_{r}}}\\
 & \times\prod_{r=1}^{n}\frac{1-\alpha x_{r}q^{k_{r}+|\boldsymbol{k}|}}{1-\alpha x_{r}}\prod_{r=1}^{n}\frac{\left(\alpha x_{r}\right)_{|\boldsymbol{k}|}}{\left(\alpha x_{r}a_{r}\right)_{|\boldsymbol{k}|}}\left(a_{1}\cdots a_{n}/q^{n}\right)^{|\boldsymbol{k}|}q^{\sum_{r=1}^{n}(r-1)k_{r}}\\
 & \times\sum_{{0\leq m_{r}\leq k_{r}\atop r=1,\cdots,n}}\prod_{1\leq r<s\leq n}\frac{1-q^{m_{r}-m_{s}}x_{r}/x_{s}}{1-x_{r}/x_{s}}\prod_{r,s=1}^{n}\frac{\left(q^{-k_{s}}x_{r}/x_{s}\right)_{m_{r}}}{\left(qx_{r}/x_{s}\right)_{m_{r}}}\\
 & \;\times\prod_{r=1}^{n}\frac{\left(\alpha x_{r}q^{|\boldsymbol{k}|}\right)_{m_{r}}}{(\alpha bx_{r})_{m_{r}}}q^{\sum_{r=1}^{n}rm_{r}}f(\alpha q^{1+m_{1}},\cdots,\alpha q^{1+m_{n}}).
\end{aligned}
\label{eq:12-2}
\end{equation}
\end{thm}
The identities in Theorems \ref{th-1}\textendash \ref{t11-2} are
all $A_{n}$ extensions of \eqref{eq:10-1}; we set $x_{r}\mapsto x_{r}/x_{n},a_{1}\mapsto a$
in the $n=1$ case of the identities in Theorems \ref{th-1}\textendash \ref{t11-2}
to obtain \eqref{eq:10-1} readily.

Two groups of multiple expansion formulas for infinite products can
be derived from the identities in Theorems \ref{th-1}\textendash \ref{t11-2}.
We set 
\[
f(x)=\frac{(cx/q,dx/q)_{\infty}}{(\beta x/q,\gamma x/q)_{\infty}},
\]
in \eqref{eq:a1-17}, \eqref{eq:a1-19} and \eqref{eq:12-1}, and
then simplify to obtain the following three multiple expansion formulas
for infinite products.
\begin{thm}
\label{t21-1} Let $\alpha,\beta,\gamma,a_{1},\cdots,a_{n},b,c,d$
be such that $\max\{|\alpha a_{1}\cdots a_{n}q^{1-n}|,|\alpha b|,$\\
$|\alpha c|,|\alpha d|,|\alpha\beta a_{1}\cdots a_{n}/q^{n}|,|\alpha\gamma a_{1}\cdots a_{n}/q^{n}|\}<1$;
or, $a_{r}=q^{N_{r}},$ for $r=1,2,\cdots,n,$ and so the series terminate.
Suppose that none of the denominators in the following identity vanishes,
then 
\[
\begin{aligned} & \frac{(\alpha q,\alpha a_{1}\cdots a_{n}b/q^{n},\alpha a_{1}\cdots a_{n}c/q^{n},\alpha a_{1}\cdots a_{n}d/q^{n},\alpha\beta,\alpha\gamma)_{\infty}}{(\alpha a_{1}\cdots a_{n}q^{1-n},\alpha b,\alpha c,\alpha d,\alpha\beta a_{1}\cdots a_{n}/q^{n},\alpha\gamma a_{1}\cdots a_{n}/q^{n})_{\infty}}\\
 & =\sum_{{k_{r}\geq0\atop r=1,2,\cdots,n}}\prod_{1\leq r<s\leq n}\frac{1-q^{k_{r}-k_{s}}x_{r}/x_{s}}{1-x_{r}/x_{s}}\prod_{r,s=1}^{n}\frac{(qx_{r}/x_{s}a_{s})_{k_{r}}}{(qx_{r}/x_{s})_{k_{r}}}\\
 & \times\frac{(1-\alpha q^{2|\boldsymbol{k}|})(\alpha)_{|\boldsymbol{k}|}}{(1-\alpha)(\alpha a_{1}\cdots a_{n}q^{1-n})_{|\boldsymbol{k}|}}(a_{1}\cdots a_{n})^{|\boldsymbol{k}|}q^{\sum_{r=1}^{n}(r-1)k_{r}-n|\boldsymbol{k}|}\\
 & \;\times\sum_{{0\leq m_{r}\leq k_{r}\atop r=1,2,\cdots,n}}\prod_{1\leq r<s\leq n}\frac{1-q^{m_{r}-m_{s}}x_{r}/x_{s}}{1-x_{r}/x_{s}}\prod_{r,s=1}^{n}\frac{(q^{-k_{s}}x_{r}/x_{s})_{m_{r}}}{(qx_{r}/x_{s})_{m_{r}}}\\
 & \quad\times\frac{(\alpha q^{|\boldsymbol{k}|},\alpha\beta,\alpha\gamma)_{|\boldsymbol{m}|}}{(\alpha b,\alpha c,\alpha d)_{|\boldsymbol{m}|}}q^{\sum_{r=1}^{n}rm_{r}}.
\end{aligned}
\]
\end{thm}
\begin{thm}
Let $\alpha,\beta,\gamma,a_{1},\cdots,a_{n},b,c,d,x_{1},\cdots,x_{n}$
be such that $\max\{|\alpha c|,|\alpha d|,$\\
$|\alpha a_{1}\cdots a_{n}q^{1-n}|,|\alpha\beta a_{1}\cdots a_{n}/q^{n}|,|\alpha\gamma a_{1}\cdots a_{n}/q^{n}|\}<1$
and $\max\{|\alpha bx_{r}|:1\leq r\leq n\}<1$; or, $a_{r}=q^{N_{r}},$
for $r=1,2,\cdots,n,$ and so the series terminate. Suppose that none
of the denominators in the following identity vanishes, then
\[
\begin{aligned} & \frac{(\alpha q,\alpha a_{1}\cdots a_{n}c/q^{n},\alpha a_{1}\cdots a_{n}d/q^{n},\alpha\beta,\alpha\gamma)_{\infty}}{(\alpha a_{1}\cdots a_{n}q^{1-n},\alpha c,\alpha d,\alpha a_{1}\cdots a_{n}\beta/q^{n},\alpha a_{1}\cdots a_{n}\gamma/q^{n})_{\infty}}\prod_{r=1}^{n}\frac{(\alpha a_{r}bx_{r}/q)_{\infty}}{(\alpha bx_{r})_{\infty}}\\
 & =\sum_{{k_{r}\geq0\atop r=1,2,\cdots,n}}\prod_{1\leq r<s\leq n}\frac{1-q^{k_{r}-k_{s}}x_{r}/x_{s}}{1-x_{r}/x_{s}}\prod_{r,s=1}^{n}\frac{(qx_{r}/x_{s}a_{s})_{k_{r}}}{(qx_{r}/x_{s})_{k_{r}}}\\
 & \times\frac{(1-\alpha q^{2|\boldsymbol{k}|})(\alpha)_{|\boldsymbol{k}|}}{(1-\alpha)(\alpha a_{1}\cdots a_{n}q^{1-n})_{|\boldsymbol{k}|}}(a_{1}\cdots a_{n})^{|\boldsymbol{k}|}q^{\sum_{r=1}^{n}(r-1)k_{r}-n|\boldsymbol{k}|}\\
 & \;\times\sum_{{0\leq m_{r}\leq k_{r}\atop r=1,2,\cdots,n}}\prod_{1\leq r<s\leq n}\frac{1-q^{m_{r}-m_{s}}x_{r}/x_{s}}{1-x_{r}/x_{s}}\prod_{r,s=1}^{n}\frac{(q^{-k_{s}}x_{r}/x_{s})_{m_{r}}}{(qx_{r}/x_{s})_{m_{r}}}\\
 & \quad\times\frac{(\alpha q^{|\boldsymbol{k}|},\alpha\beta,\alpha\gamma)_{|\boldsymbol{m}|}}{(\alpha c,\alpha d)_{|\boldsymbol{m}|}\prod_{r=1}^{n}(\alpha bx_{r})_{m_{r}}}q^{\sum_{r=1}^{n}rm_{r}}.
\end{aligned}
\]
\end{thm}
\begin{thm}
\label{t11-3} Let $\alpha,\beta,\gamma,a_{1},\cdots,a_{n},b,c,d,x_{1},\cdots,x_{n}$
be such that $\max\{|\alpha c|,|\alpha d|,$\\
$|\alpha\beta a_{1}\cdots a_{n}/q^{n}|,|\alpha\gamma a_{1}\cdots a_{n}/q^{n}|\}<1$
and $\max\{|\alpha ax_{r}|,|\alpha bx_{r}|:1\leq r\leq n\}<1$; or,
$a_{r}=q^{N_{r}},$ for $r=1,2,\cdots,n,$ and so the series terminate.
Suppose that none of the denominators in the following identity vanishes,
then
\begin{equation}
\begin{aligned} & \frac{(\alpha\beta,\alpha\gamma,\alpha a_{1}\cdots a_{n}cq^{-n},\alpha a_{1}\cdots a_{n}dq^{-n})_{\infty}}{(\alpha c,\alpha d,\alpha\beta a_{1}\cdots a_{n}q^{-n},\alpha\gamma a_{1}\cdots a_{n}q^{-n})_{\infty}}\prod_{r=1}^{n}\frac{(\alpha qx_{r},\alpha a_{r}bx_{r}/q)_{\infty}}{(\alpha bx_{r},\alpha a_{r}x_{r})_{\infty}}\\
 & =\sum_{{k_{r}\geq0\atop r=1,\cdots,n}}\prod_{1\leq r<s\leq n}\frac{1-q^{k_{r}-k_{s}}x_{r}/x_{s}}{1-x_{r}/x_{s}}\prod_{r,s=1}^{n}\frac{\left(qx_{r}/x_{s}a_{s}\right)_{k_{r}}}{\left(qx_{r}/x_{s}\right)_{k_{r}}}\\
 & \times\prod_{r=1}^{n}\frac{1-\alpha x_{r}q^{k_{r}+|\boldsymbol{k}|}}{1-\alpha x_{r}}\prod_{r=1}^{n}\frac{\left(\alpha x_{r}\right)_{|\boldsymbol{k}|}}{\left(\alpha x_{r}a_{r}\right)_{|\boldsymbol{k}|}}\left(a_{1}\cdots a_{n}\right)^{|\boldsymbol{k}|}q^{\sum_{r=1}^{n}(r-1)k_{r}-n|\boldsymbol{k}|}\\
 & \times\sum_{{0\leq m_{r}\leq k_{r}\atop r=1,\cdots,n}}\prod_{1\leq r<s\leq n}\frac{1-q^{m_{r}-m_{s}}x_{r}/x_{s}}{1-x_{r}/x_{s}}\prod_{r,s=1}^{n}\frac{\left(q^{-k_{s}}x_{r}/x_{s}\right)_{m_{r}}}{\left(qx_{r}/x_{s}\right)_{m_{r}}}\\
 & \;\times\prod_{r=1}^{n}\frac{\left(\alpha x_{r}q^{|\boldsymbol{k}|}\right)_{m_{r}}}{(\alpha bx_{r})_{m_{r}}}\frac{(\alpha\beta,\alpha\gamma)_{|\boldsymbol{m}|}}{(\alpha c,\alpha d)_{|\boldsymbol{m}|}}q^{\sum_{r=1}^{n}rm_{r}}.
\end{aligned}
\label{eq:12-3}
\end{equation}
\end{thm}
The second group can be obtained as follows. We now take

\[
f(y_{1},y_{2},\cdots,y_{n})=\prod_{r=1}^{n}\frac{(cx_{r}y_{r}/q,dx_{r}y_{r}/q)_{\infty}}{(\beta x_{r}y_{r}/q,\gamma x_{r}y_{r}/q)_{\infty}}
\]
in \eqref{eq:a1-18}, \eqref{eq:a1-20} and \eqref{eq:12-2}, and
then simplify to give three other multiple expansion formulas for
infinite products.
\begin{thm}
Let $\alpha,\beta,\gamma,a_{1},a_{2},\cdots,a_{n},b,c,d,x_{1},\cdots,x_{n}$
be such that $\max\{|\alpha b|,$\\
$|\alpha a_{1}\cdots a_{n}q^{1-n}|\}<1$ and $\max\{|\alpha cx_{r}|,|\alpha dx_{r}|,|\alpha\beta a_{r}x_{r}/q|,|\alpha\gamma a_{r}x_{r}/q|:1\leq r\leq n\}<1$;
or, $a_{r}=q^{N_{r}},$ for $r=1,2,\cdots,n,$ and so the series terminate.
Suppose that none of the denominators in the following identity vanishes,
then
\end{thm}
\[
\begin{aligned} & \frac{(\alpha q,\alpha a_{1}\cdots a_{n}b/q^{n})_{\infty}}{(\alpha b,\alpha a_{1}\cdots a_{n}q^{1-n})_{\infty}}\prod_{r=1}^{n}\frac{(\alpha a_{r}cx_{r}/q,\alpha a_{r}dx_{r}/q,\alpha\beta x_{r},\alpha\gamma x_{r})_{\infty}}{(\alpha cx_{r},\alpha dx_{r},\alpha\beta a_{r}x_{r}/q,\alpha\gamma a_{r}x_{r}/q)_{\infty}}\\
 & =\sum_{{k_{r}\geq0\atop r=1,2,\cdots,n}}\prod_{1\leq r<s\leq n}\frac{1-q^{k_{r}-k_{s}}x_{r}/x_{s}}{1-x_{r}/x_{s}}\prod_{r,s=1}^{n}\frac{(qx_{r}/x_{s}a_{s})_{k_{r}}}{(qx_{r}/x_{s})_{k_{r}}}\\
 & \times\frac{(1-\alpha q^{2|\boldsymbol{k}|})(\alpha)_{|\boldsymbol{k}|}}{(1-\alpha)(\alpha a_{1}\cdots a_{n}q^{1-n})_{|\boldsymbol{k}|}}(a_{1}\cdots a_{n})^{|\boldsymbol{k}|}q^{\sum_{r=1}^{n}(r-1)k_{r}-n|\boldsymbol{k}|}\\
 & \;\times\sum_{{0\leq m_{r}\leq k_{r}\atop r=1,2,\cdots,n}}\prod_{1\leq r<s\leq n}\frac{1-q^{m_{r}-m_{s}}x_{r}/x_{s}}{1-x_{r}/x_{s}}\prod_{r,s=1}^{n}\frac{(q^{-k_{s}}x_{r}/x_{s})_{m_{r}}}{(qx_{r}/x_{s})_{m_{r}}}\\
 & \quad\times\frac{(\alpha q^{|\boldsymbol{k}|})_{|\boldsymbol{m}|}}{(\alpha b)_{|\boldsymbol{m}|}}\prod_{r=1}^{n}\frac{(\alpha\beta x_{r},\alpha\gamma x_{r})_{m_{r}}}{(\alpha cx_{r},\alpha dx_{r})_{m_{r}}}q^{\sum_{r=1}^{n}rm_{r}}.
\end{aligned}
\]

\begin{thm}
Let $\alpha,\beta,\gamma,a_{1},\cdots,a_{n},b,c,d,x_{1},\cdots,x_{n}$
be such that $\max\{|\alpha a_{1}\cdots a_{n}q^{1-n}|\}$\\
$<1$ and $\max\{|\alpha bx_{r}|,|\alpha cx_{r}|,|\alpha dx_{r}|,|\alpha\beta a_{r}x_{r}/q|,|\alpha\gamma a_{r}x_{r}/q|:1\leq r\leq n\}<1$;
or, $a_{r}=q^{N_{r}},$ for $r=1,2,\cdots,n,$ and so the series terminate.
Suppose that none of the denominators in the following identity vanishes,
then\\
\[
\begin{aligned} & \frac{(\alpha q)_{\infty}}{(\alpha a_{1}a_{2}\cdots a_{n}q^{1-n})_{\infty}}\prod_{r=1}^{n}\frac{(\alpha a_{r}bx_{r}/q,\alpha a_{r}cx_{r}/q,\alpha a_{r}dx_{r}/q,\alpha\beta x_{r},\alpha\gamma x_{r})_{\infty}}{(\alpha bx_{r},\alpha cx_{r},\alpha dx_{r},\alpha\beta a_{r}x_{r}/q,\alpha\gamma a_{r}x_{r}/q)_{\infty}}\\
 & =\sum_{{k_{r}\geq0\atop r=1,2,\cdots,n}}\prod_{1\leq r<s\leq n}\frac{1-q^{k_{r}-k_{s}}x_{r}/x_{s}}{1-x_{r}/x_{s}}\prod_{r,s=1}^{n}\frac{(qx_{r}/x_{s}a_{s})_{k_{r}}}{(qx_{r}/x_{s})_{k_{r}}}\\
 & \times\frac{(1-\alpha q^{2|\boldsymbol{k}|})(\alpha)_{|\boldsymbol{k}|}}{(1-\alpha)(\alpha a_{1}a_{2}\cdots a_{n}q^{1-n})_{|\boldsymbol{k}|}}(a_{1}a_{2}\cdots a_{n})^{|\boldsymbol{k}|}q^{\sum_{r=1}^{n}(r-1)k_{r}-n|\boldsymbol{k}|}\\
 & \;\times\sum_{{0\leq m_{r}\leq k_{r}\atop r=1,2,\cdots,n}}\prod_{1\leq r<s\leq n}\frac{1-q^{m_{r}-m_{s}}x_{r}/x_{s}}{1-x_{r}/x_{s}}\prod_{r,s=1}^{n}\frac{(q^{-k_{s}}x_{r}/x_{s})_{m_{r}}}{(qx_{r}/x_{s})_{m_{r}}}\\
 & \quad\times(\alpha q^{|\boldsymbol{k}|})_{|\boldsymbol{m}|}\prod_{r=1}^{n}\frac{(\alpha\beta x_{r},\alpha\gamma x_{r})_{m_{r}}}{(\alpha bx_{r},\alpha cx_{r},\alpha dx_{r})_{m_{r}}}q^{\sum_{r=1}^{n}rm_{r}}.
\end{aligned}
\]
\end{thm}
\begin{thm}
\textup{\label{t11-4}} Let $\alpha,\beta,\gamma,a_{1},\cdots,a_{n},b,c,d,x_{1},\cdots,x_{n}$
be such that $\max\{|\alpha ax_{r}|,$\\
$|\alpha bx_{r}|,|\alpha cx_{r}|,|\alpha dx_{r}|,|\alpha\beta a_{r}x_{r}/q|,|\alpha\gamma a_{r}x_{r}/q|:1\leq r\leq n\}<1$;
or, $a_{r}=q^{N_{r}},$ for $r=1,2,\cdots,n,$ and so the series terminate.
Suppose that none of the denominators in the following identity vanishes,
then

\begin{equation}
\begin{aligned} & \prod_{r=1}^{n}\frac{(\alpha qx_{r},\alpha a_{r}bx_{r}/q,\alpha a_{r}cx_{r}/q,\alpha a_{r}dx_{r}/q,\alpha\beta x_{r},\alpha\gamma x_{r})_{\infty}}{(\alpha a_{r}x_{r},\alpha bx_{r},\alpha cx_{r},\alpha dx_{r},\alpha\beta a_{r}x_{r}/q,\alpha\gamma a_{r}x_{r}/q)_{\infty}}\\
 & =\sum_{{k_{r}\geq0\atop r=1,\cdots,n}}\prod_{1\leq r<s\leq n}\frac{1-q^{k_{r}-k_{s}}x_{r}/x_{s}}{1-x_{r}/x_{s}}\prod_{r,s=1}^{n}\frac{\left(qx_{r}/x_{s}a_{s}\right)_{k_{r}}}{\left(qx_{r}/x_{s}\right)_{k_{r}}}\\
 & \times\prod_{r=1}^{n}\frac{1-\alpha x_{r}q^{k_{r}+|\boldsymbol{k}|}}{1-\alpha x_{r}}\prod_{r=1}^{n}\frac{\left(\alpha x_{r}\right)_{|\boldsymbol{k}|}}{\left(\alpha x_{r}a_{r}\right)_{|\boldsymbol{k}|}}\left(a_{1}\cdots a_{n}/q^{n}\right)^{|\boldsymbol{k}|}q^{\sum_{r=1}^{n}(r-1)k_{r}}\\
 & \times\sum_{{0\leq m_{r}\leq k_{r}\atop r=1,\cdots,n}}\prod_{1\leq r<s\leq n}\frac{1-q^{m_{r}-m_{s}}x_{r}/x_{s}}{1-x_{r}/x_{s}}\prod_{r,s=1}^{n}\frac{\left(q^{-k_{s}}x_{r}/x_{s}\right)_{m_{r}}}{\left(qx_{r}/x_{s}\right)_{m_{r}}}\\
 & \;\times\prod_{r=1}^{n}\frac{\left(\alpha x_{r}q^{|\boldsymbol{k}|},\alpha\beta x_{r},\alpha\gamma x_{r}\right)_{m_{r}}}{(\alpha bx_{r},\alpha cx_{r},\alpha dx_{r})_{m_{r}}}q^{\sum_{r=1}^{n}rm_{r}}.
\end{aligned}
\label{eq:12-4}
\end{equation}
\end{thm}
Take $x_{r}\mapsto x_{r}/x_{n},a_{1}\mapsto a$ in the $n=1$ case
of the identities in Theorems \ref{t21-1}\textendash \ref{t11-4}.
We can obtain the following expansion formula \cite[Theorem 1.2]{L2}:
for $\max\{|\alpha a|,$ $|\alpha b|,|\alpha c|,|\alpha d|,|\alpha\beta a/q|,|\alpha\gamma a/q|\}<1,$
we have
\begin{equation}
\begin{aligned} & \frac{(\alpha q,\alpha ab/q,\alpha ac/q,\alpha ad/q,\alpha\beta,\alpha\gamma)_{\infty}}{(\alpha a,\alpha b,\alpha c,\alpha d,\alpha\beta a/q,\alpha\gamma a/q)_{\infty}}\\
 & =\sum_{n=0}^{\infty}\frac{(1-\alpha q^{2n})(\alpha,q/a)_{n}(a/q)^{n}}{(1-\alpha)(q,\alpha a)_{n}}\sum_{k=0}^{n}\frac{(q^{-n},\alpha q^{n},\alpha\beta,\alpha\gamma)_{k}q^{k}}{(q,\alpha b,\alpha c,\alpha d)_{k}}.
\end{aligned}
\label{eq:10-7}
\end{equation}
Therefore, the identities in Theorems \ref{t21-1}\textendash \ref{t11-4}
are all $A_{n}$ extensions of \eqref{eq:10-7}. The identity \eqref{eq:10-7}
is very useful in \cite{L2}. It was used by Liu \cite{L2} to give
a new proof of the orthogonality relation for the Askey-Wilson polynomials
and to derive a useful identity which is equivalent to the well-known
identity of Andrews in \cite[Theorem 5]{A86}.

Several multiple expansion formulas for $(q)_{\infty}^{3}$ and $(q)_{\infty}^{-2}$
can be easily derived from the identities in Theorems \ref{t21-1}\textendash \ref{t11-4}.
For example, we set $\alpha=q,\beta=\gamma=1,a_{1}=\cdots=a_{n}=b=c=d=0$
or $\alpha=1,\beta=\gamma=q,a_{1}=\cdots=a_{n}=b=c=d=0$ in the identity
of Theorem \ref{t21-1} to get two multiple expansion formulas for
$(q)_{\infty}^{3}:$

\begin{equation}
\begin{aligned}(q)_{\infty}^{3} & =\sum_{{k_{r}\geq0\atop r=1,2,\cdots,n}}\prod_{1\leq r<s\leq n}\frac{1-q^{k_{r}-k_{s}}x_{r}/x_{s}}{1-x_{r}/x_{s}}\cdot\frac{(1-q^{2|\boldsymbol{k}|+1})(q)_{|\boldsymbol{k}|}}{\prod_{r,s=1}^{n}(qx_{r}/x_{s})_{k_{r}}}\\
 & \times(-1)^{n|\boldsymbol{k}|}q^{\sum_{r=1}^{n}(r-1)k_{r}+n\sum_{r=1}^{n}{k_{r} \choose 2}}\prod_{r=1}^{n}x_{r}^{nk_{r}-|\boldsymbol{k}|}\\
 & \;\times\sum_{{0\leq m_{r}\leq k_{r}\atop r=1,2,\cdots,n}}\prod_{1\leq r<s\leq n}\frac{1-q^{m_{r}-m_{s}}x_{r}/x_{s}}{1-x_{r}/x_{s}}\\
 & \quad\times\prod_{r,s=1}^{n}\frac{(q^{-k_{s}}x_{r}/x_{s})_{m_{r}}}{(qx_{r}/x_{s})_{m_{r}}}\cdot(q^{|\boldsymbol{k}|+1},q,q)_{|\boldsymbol{m}|}q^{\sum_{r=1}^{n}rm_{r}}
\end{aligned}
\label{eq:1-20}
\end{equation}
and 
\begin{equation}
\begin{aligned}(q)_{\infty}^{3} & =1+\sum_{{k_{1}\geq0,\cdots,k_{n}\geq0\atop |\boldsymbol{k}|\geq1}}\prod_{1\leq r<s\leq n}\frac{1-q^{k_{r}-k_{s}}x_{r}/x_{s}}{1-x_{r}/x_{s}}\cdot\frac{(1-q^{2|\boldsymbol{k}|})(q)_{|\boldsymbol{k}|-1}}{\prod_{r,s=1}^{n}(qx_{r}/x_{s})_{k_{r}}}\\
 & \times(-1)^{n|\boldsymbol{k}|}q^{\sum_{r=1}^{n}(r-1)k_{r}+n\sum_{r=1}^{n}{k_{r} \choose 2}}\prod_{r=1}^{n}x_{r}^{nk_{r}-|\boldsymbol{k}|}\\
 & \;\times\sum_{{0\leq m_{r}\leq k_{r}\atop r=1,2,\cdots,n}}\prod_{1\leq r<s\leq n}\frac{1-q^{m_{r}-m_{s}}x_{r}/x_{s}}{1-x_{r}/x_{s}}\\
 & \quad\times\prod_{r,s=1}^{n}\frac{(q^{-k_{s}}x_{r}/x_{s})_{m_{r}}}{(qx_{r}/x_{s})_{m_{r}}}\cdot(q^{|\boldsymbol{k}|},q,q)_{|\boldsymbol{m}|}q^{\sum_{r=1}^{n}rm_{r}}.
\end{aligned}
\label{eq:1-21}
\end{equation}
Taking $\alpha=q,\beta=\gamma=a_{1}=\cdots=a_{n}=0,b=q,c=d=1$ or
$\alpha=1,\beta=\gamma=a_{1}=\cdots=a_{n}=0,b=c=d=q$ in the identity
of Theorem \ref{t21-1} we obtain two multiple expansion formulas
for $(q)_{\infty}^{-2}:$
\begin{equation}
\begin{aligned}\frac{1}{(q)_{\infty}^{2}} & =\sum_{{k_{r}\geq0\atop r=1,2,\cdots,n}}\prod_{1\leq r<s\leq n}\frac{1-q^{k_{r}-k_{s}}x_{r}/x_{s}}{1-x_{r}/x_{s}}\cdot\frac{(1-q^{2|\boldsymbol{k}|+1})(q)_{|\boldsymbol{k}|}}{(1-q)\prod_{r,s=1}^{n}(qx_{r}/x_{s})_{k_{r}}}\\
 & \times(-1)^{n|\boldsymbol{k}|}q^{\sum_{r=1}^{n}(r-1)k_{r}+n\sum_{r=1}^{n}{k_{r} \choose 2}}\prod_{r=1}^{n}x_{r}^{nk_{r}-|\boldsymbol{k}|}\\
 & \;\times\sum_{{0\leq m_{r}\leq k_{r}\atop r=1,2,\cdots,n}}\prod_{1\leq r<s\leq n}\frac{1-q^{m_{r}-m_{s}}x_{r}/x_{s}}{1-x_{r}/x_{s}}\\
 & \quad\times\prod_{r,s=1}^{n}\frac{(q^{-k_{s}}x_{r}/x_{s})_{m_{r}}}{(qx_{r}/x_{s})_{m_{r}}}\cdot\frac{(q^{|\boldsymbol{k}|+1})_{|\boldsymbol{m}|}}{(q^{2},q,q)_{|\boldsymbol{m}|}}q^{\sum_{r=1}^{n}rm_{r}}
\end{aligned}
\label{eq:1-22}
\end{equation}
and 
\begin{equation}
\begin{aligned}\frac{1}{(q)_{\infty}^{2}} & =1+\sum_{{k_{1}\geq0,\cdots,k_{n}\geq0\atop |\boldsymbol{k}|\geq1}}\prod_{1\leq r<s\leq n}\frac{1-q^{k_{r}-k_{s}}x_{r}/x_{s}}{1-x_{r}/x_{s}}\cdot\frac{(1-q^{2|\boldsymbol{k}|})(q)_{|\boldsymbol{k}|-1}}{\prod_{r,s=1}^{n}(qx_{r}/x_{s})_{k_{r}}}\\
 & \times(-1)^{n|\boldsymbol{k}|}q^{\sum_{r=1}^{n}(r-1)k_{r}+n\sum_{r=1}^{n}{k_{r} \choose 2}}\prod_{r=1}^{n}x_{r}^{nk_{r}-|\boldsymbol{k}|}\\
 & \;\times\sum_{{0\leq m_{r}\leq k_{r}\atop r=1,2,\cdots,n}}\prod_{1\leq r<s\leq n}\frac{1-q^{m_{r}-m_{s}}x_{r}/x_{s}}{1-x_{r}/x_{s}}\\
 & \quad\times\prod_{r,s=1}^{n}\frac{(q^{-k_{s}}x_{r}/x_{s})_{m_{r}}}{(qx_{r}/x_{s})_{m_{r}}}\cdot\frac{(q^{|\boldsymbol{k}|})_{|\boldsymbol{m}|}}{(q,q,q)_{|\boldsymbol{m}|}}q^{\sum_{r=1}^{n}rm_{r}}.
\end{aligned}
\label{eq:1-23}
\end{equation}
For any interger $m,$ some general multiple expansion formulas for
$(q)_{\infty}^{m}$ will be showed in Section \ref{sec:2}.

The rest of this paper is organized as follows. In the next section,
we first construct two inverse matrices equivalent to the matrices
in \cite[Theorem 8.26]{M} and then employ a inverse relation satisfied
by these two matrices to prove Theorem \ref{t1}. The identity in
Theorem \ref{t1} is very useful. In Section \ref{sec:3}, we will
apply the identity \eqref{eq:t1-1} in Theorem \ref{t1} and a multiple
series expansion formula in \cite[Theorem 2.2]{BR} to show the identities
in Theorems \ref{th-1}\textendash \ref{t11-2}.  As applications,
 in Section \ref{sec:4}, we will apply the identity in Theorem \ref{t11-3}
to deduce an $A_{n}$ Rogers' $\text{}_{6}\phi_{5}$ summation formula
which is equivalent to the identity in \cite[Theorem A.3]{MN12}.
From this formula, we will derive an $A_{n}$ extension of Sylvester's
identity. This $A_{n}$ extension is equivalent to the identity in
\cite[Theorem 5.19]{M94}; in Section \ref{sec:2}, we will employ
the identities in Theorems \ref{th-1} and \ref{t21-1} to deduce
some multiple expansion formulas for $(q)_{\infty}^{m},\text{\ensuremath{\pi_{q}}}$
and $1/\pi_{q}$; in Section \ref{sec:7}, we will apply the identity
\eqref{eq:t1-1} of Theorem \ref{t1} to present two $A_{n}$ extensions
of the Rogers-Fine identity \cite[eq.(14.1)]{Fi}; in Section \ref{sec:8},
we will employ the identity \eqref{eq:12-1} in Theorem \ref{t11-2}
to give an $A_{n}$ extension of Liu's extension \cite[Theorem 3]{L11}
of Rogers' non-terminating $_{6}\phi_{5}$ summation \cite[eq.(II.20)]{GR};
in Section \ref{sec:9}, we will use the identity \eqref{eq:t1-1}
of Theorem \ref{t1} to prove an $A_{n}$ extension of a generalization
of Fang's identity \cite[eq.(9)]{F07}; and in the last section, we
will apply the identity in Theorem \ref{t21-1} to deduce an $A_{n}$
extension of Andrews' expansion formula \cite[Theorem 5]{A86}. 

\section{Proof of Theorem \ref{t1}}

We first construct two inverse, infinite matrices, which are equivalent
to the matrices in \cite[Theorem 8.26]{M}. Let 
\[
H_{\boldsymbol{km}}(a)=\prod_{r,s=1}^{n}\left(q^{1+m_{r}-m_{s}}\frac{x_{r}}{x_{s}}\right)_{k_{r}-m_{r}}^{-1}(aq)_{|\boldsymbol{k}|+|\boldsymbol{m}|}^{-1}
\]
and 
\begin{align*}
I_{\boldsymbol{km}} & (a)=(1-aq^{2|\boldsymbol{k}|})(aq)_{|\boldsymbol{k}|+|\boldsymbol{m}|-1}\prod_{r,s=1}^{n}\left(q^{1+m_{r}-m_{s}}\frac{x_{r}}{x_{s}}\right)_{k_{r}-m_{r}}^{-1}\\
 & \quad\times(-1)^{|\boldsymbol{k}|+|\boldsymbol{m}|}q^{{|\boldsymbol{k}|-|\boldsymbol{m}| \choose 2}}
\end{align*}
with rows and columns ordered lexicographically. Then $\boldsymbol{H}=(H_{\boldsymbol{km}}(a))$
and $\boldsymbol{I}=(I_{\boldsymbol{km}}(a))$ are inverse, infinite
matrices.

Recall the following important result \cite[Lemma 4.3]{M}:
\begin{equation}
\begin{aligned} & \prod_{r,s=1}^{n}\left(q^{1+m_{r}-m_{s}}\frac{x_{r}}{x_{s}}\right)_{k_{r}-m_{r}}^{-1}\\
 & =(-1)^{|\boldsymbol{m}|}q^{|\boldsymbol{k}||\boldsymbol{m}|-{|\boldsymbol{m}| \choose 2}+\sum_{r=1}^{n}(r-1)m_{r}}\\
 & \quad\times\prod_{1\leq r<s\leq n}\frac{1-q^{m_{r}-m_{s}}x_{r}/x_{s}}{1-x_{r}/x_{s}}\prod_{r,s=1}^{n}\frac{(q^{-k_{s}}x_{r}/x_{s})_{m_{r}}}{(qx_{r}/x_{s})_{k_{r}}}.
\end{aligned}
\label{eq:a1-1}
\end{equation}
Let $\boldsymbol{D}$ and $\boldsymbol{E}$ be two diagonal matrices
with diagonal entries given by 
\[
D_{kk}=(aq)_{|\boldsymbol{k}|}\prod_{r,s=1}^{n}(qx_{r}/x_{s})_{k_{r}}
\]
and 
\[
E_{mm}=(-1)^{|\boldsymbol{m}|}q^{{|\boldsymbol{m}| \choose 2}}(1-aq^{2|\boldsymbol{m}|}).
\]
Set $\boldsymbol{F}=\boldsymbol{D}\boldsymbol{H}\boldsymbol{E}$ and
$\boldsymbol{G}=\boldsymbol{E}^{-1}\boldsymbol{I}\boldsymbol{D}^{-1}.$
Then, by \eqref{eq:a1-1},
\begin{equation}
\begin{aligned}F_{\boldsymbol{km}} & (a)=\prod_{1\leq r<s\leq n}\frac{1-q^{m_{r}-m_{s}}x_{r}/x_{s}}{1-x_{r}/x_{s}}q^{|\boldsymbol{k}||\boldsymbol{m}|+\sum_{r=1}^{n}(r-1)m_{r}}\\
 & \quad\times\frac{1-aq^{2|\boldsymbol{m}|}}{(aq^{|\boldsymbol{k}|+1})_{|\boldsymbol{m}|}}\prod_{r,s=1}^{n}(q^{-k_{s}}x_{r}/x_{s})_{m_{r}}
\end{aligned}
\label{eq:a1-7}
\end{equation}
and 
\begin{equation}
\begin{aligned}G_{\boldsymbol{km}} & (a)=\prod_{1\leq r<s\leq n}\frac{1-q^{m_{r}-m_{s}}x_{r}/x_{s}}{1-x_{r}/x_{s}}q^{\sum_{r=1}^{n}rm_{r}}\\
 & \quad\times\frac{(aq^{|\boldsymbol{m}|})_{|\boldsymbol{k}|}}{1-aq^{|\boldsymbol{m}|}}\prod_{r,s=1}^{n}\frac{(q^{-k_{s}}x_{r}/x_{s})_{m_{r}}}{(qx_{r}/x_{s})_{k_{r}}(qx_{r}/x_{s})_{m_{r}}}.
\end{aligned}
\label{eq:a1-8}
\end{equation}
Since $\boldsymbol{H}$ and $\boldsymbol{I}$ are inverses, the matrices
$\boldsymbol{F}$ and $\boldsymbol{G}$ are also inverses. Note that
$F_{\boldsymbol{km}}(a)=G_{\boldsymbol{km}}(a)=0$ unless $m_{r}\leq k_{r}$
for $r=1,2,\cdots,n.$ We have

\begin{equation}
\sum_{{m_{r}\leq k_{r}\leq N_{r}\atop r=1,2,\cdots,n}}F_{\boldsymbol{Nk}}(a)G_{\boldsymbol{km}}(a)=\prod_{r=1}^{n}\delta_{N_{r}m_{r}},\label{eq:a1-2}
\end{equation}
where
\[
\delta_{ij}=\begin{cases}
1, & \mathrm{if}\:i=j,\\
0, & \mathrm{otherwise}.
\end{cases}
\]

We now prove \eqref{eq:t1-1}. Let $f_{\boldsymbol{k}}(\boldsymbol{y};a)$
be a function of $\boldsymbol{y}=(y_{1},y_{2},\cdots,y_{n})$ given
by 
\begin{align*}
f_{\boldsymbol{k}}(\boldsymbol{y};a) & =\prod_{1\leq r<s\leq n}\frac{1-q^{k_{r}-k_{s}}x_{r}/x_{s}}{1-x_{r}/x_{s}}q^{\sum_{r=1}^{n}(r-1)k_{r}}\\
 & \quad\times\frac{1-aq^{2|\boldsymbol{k}|}}{(aqy_{1}y_{2}\cdots y_{n})_{|\boldsymbol{k}|}}(y_{1}y_{2}\cdots y_{n})^{|\boldsymbol{k}|}\prod_{r,s=1}^{n}(x_{r}/x_{s}y_{s})_{k_{r}}.
\end{align*}
Then $f_{\boldsymbol{k}}(\boldsymbol{y};a)$ is analytic in $\boldsymbol{y}$
in a disk around the origin and 
\begin{equation}
f_{\boldsymbol{k}}(q^{\boldsymbol{m}};a)=F_{\boldsymbol{mk}}(a).\label{eq:a1-9}
\end{equation}

Let
\begin{equation}
g(\boldsymbol{y})=K(\boldsymbol{y})\sum_{{j_{r}\geq0\atop r=1,2,\cdots,n}}f_{\boldsymbol{j}}(\boldsymbol{y};b)\beta(\boldsymbol{j}).\label{eq:a1-3}
\end{equation}
It follows from \eqref{eq:a1-2} that 
\begin{equation}
\begin{aligned}g(q^{\boldsymbol{N}}) & =\sum_{{0\leq m_{r}\leq N_{r}\atop r=1,2,\cdots,n}}\sum_{{m_{r}\leq k_{r}\leq N_{r}\atop r=1,2,\cdots,n}}F_{\boldsymbol{Nk}}(a)G_{\boldsymbol{km}}(a)g(q^{\boldsymbol{m}})\\
 & =\sum_{{0\leq k_{r}\leq N_{r}\atop r=1,2,\cdots,n}}F_{\boldsymbol{Nk}}(a)\sum_{{0\leq m_{r}\leq k_{r}\atop r=1,2,\cdots,n}}G_{\boldsymbol{km}}(a)g(q^{\boldsymbol{m}}).
\end{aligned}
\label{eq:a1-4}
\end{equation}
Substituting \eqref{eq:a1-3} into both sides of \eqref{eq:a1-4}
we have
\begin{equation}
\begin{aligned} & K(q^{\boldsymbol{N}})\sum_{{j_{r}\geq0\atop r=1,2,\cdots,n}}f_{\boldsymbol{j}}(q^{\boldsymbol{N}};b)\beta(\boldsymbol{j})\\
 & =\sum_{{0\leq k_{r}\leq N_{r}\atop r=1,2,\cdots,n}}F_{\boldsymbol{Nk}}(a)\sum_{{0\leq m_{r}\leq k_{r}\atop r=1,2,\cdots,n}}G_{\boldsymbol{km}}(a)K(q^{\boldsymbol{m}})\sum_{{0\leq j_{r}\leq m_{r}\atop r=1,2,\cdots,n}}f_{\boldsymbol{j}}(q^{\boldsymbol{m}};b)\beta(\boldsymbol{j})\\
 & =\sum_{{0\leq k_{r}\leq N_{r}\atop r=1,2,\cdots,n}}F_{\boldsymbol{Nk}}(a)\sum_{{0\leq j_{r}\leq k_{r}\atop r=1,2,\cdots,n}}\beta(\boldsymbol{j})\sum_{{0\leq m_{r}\leq k_{r}-j_{r}\atop r=1,2,\cdots,n}}G_{\boldsymbol{k,m+j}}(a)K(q^{\boldsymbol{m+j}})f_{\boldsymbol{j}}(q^{\boldsymbol{m+j}};b),
\end{aligned}
\label{eq:a1-6}
\end{equation}
where we interchanged the inner sums using the identity
\[
\sum_{{0\leq m_{r}\leq k_{r}\atop r=1,2,\cdots,n}}\sum_{{0\leq j_{r}\leq m_{r}\atop r=1,2,\cdots,n}}S(\boldsymbol{m},\boldsymbol{j})=\sum_{{0\leq j_{r}\leq k_{r}\atop r=1,2,\cdots,n}}\sum_{{0\leq m_{r}\leq k_{r}-j_{r}\atop r=1,2,\cdots,n}}S(\boldsymbol{m}+\boldsymbol{j},\boldsymbol{j}).
\]
Replacing $N_{r}$ by $m_{r}+j_{r},$ $y_{r}$ by $j_{r}$ in the
identity of \cite[Lemma 4.3]{M} we get
\begin{equation}
\begin{aligned} & \prod_{r,s=1}^{n}\frac{(q^{-m_{s}-j_{s}}x_{r}/x_{s})_{j_{r}}}{(qx_{r}/x_{s})_{m_{r}+j_{r}}}\\
 & =\prod_{r,s=1}^{n}(q^{1+j_{r}-j_{s}}x_{r}/x_{s})_{m_{r}}^{-1}\prod_{1\leq r<s\leq n}\frac{1-x_{r}/x_{s}}{1-q^{j_{r}-j_{s}}x_{r}/x_{s}}\\
 & \quad\times(-1)^{|\boldsymbol{j}|}q^{{|\boldsymbol{j}| \choose 2}-\sum_{r=1}^{n}(r-1)j_{r}-(|\boldsymbol{m}|+|\boldsymbol{j}|)|\boldsymbol{j}|}.
\end{aligned}
\label{eq:a1-5}
\end{equation}
Dividing both sides of \eqref{eq:a1-6} by $1-b,$ then substituting
\eqref{eq:a1-7}, \eqref{eq:a1-8}, \eqref{eq:a1-9} and \eqref{eq:a1-5}
into \eqref{eq:a1-6} and finally simplifying using the identities
\begin{align*}
 & \frac{(aq^{m+j})_{k}}{1-aq^{m+j}}=\frac{(a)_{k}(aq^{k})_{j}(aq^{k+j})_{m}}{(1-a)(aq)_{m+j}},\\
 & \frac{1-bq^{2j}}{(1-b)(bq^{m+j+1})_{j}}=\frac{(bq)_{j}(bq^{1+j})_{m}}{(b)_{2j}(bq^{2j+1})_{m}}
\end{align*}
we obtain
\begin{align*}
 & K(q^{\boldsymbol{N}})\sum_{{0\leq j_{r}\leq N_{r}\atop r=1,2,\cdots,n}}\prod_{1\leq r<s\leq n}\frac{1-q^{j_{r}-j_{s}}x_{r}/x_{s}}{1-x_{r}/x_{s}}\prod_{r,s=1}^{n}(q^{-N_{s}}x_{r}/x_{s})_{j_{r}}\\
 & \quad\times\frac{1-bq^{2|\boldsymbol{j}|}}{(1-b)(bq^{|\boldsymbol{N}|+1})_{|\boldsymbol{j}|}}q^{|\boldsymbol{N}||\boldsymbol{j}|+\sum_{r=1}^{n}(r-1)j_{r}}\beta(\boldsymbol{j})\\
 & =\sum_{{0\leq k_{r}\leq N_{r}\atop r=1,2,\cdots,n}}\prod_{1\leq r<s\leq n}\frac{1-q^{k_{r}-k_{s}}x_{r}/x_{s}}{1-x_{r}/x_{s}}\prod_{r,s=1}^{n}\frac{(q^{-N_{s}}x_{r}/x_{s})_{k_{r}}}{(qx_{r}/x_{s})_{k_{r}}}\\
 & \times\frac{(1-aq^{2|\boldsymbol{k}|})(a)_{|\boldsymbol{k}|}}{(1-a)(aq^{|\boldsymbol{N}|+1})_{|\boldsymbol{k}|}}q^{|\boldsymbol{N}||\boldsymbol{k}|+\sum_{r=1}^{n}(r-1)k_{r}}\\
 & \times\sum_{{0\leq j_{r}\leq k_{r}\atop r=1,2,\cdots,n}}\prod_{1\leq r<s\leq n}\frac{1-q^{j_{r}-j_{s}}x_{r}/x_{s}}{1-x_{r}/x_{s}}\prod_{r,s=1}^{n}(q^{-k_{s}}x_{r}/x_{s})_{j_{r}}\\
 & \;\times\frac{(aq^{|\boldsymbol{k}|},bq)_{|\boldsymbol{j}|}}{(b)_{2|\boldsymbol{j}|}}(-1)^{|\boldsymbol{j}|}q^{\sum_{r=1}^{n}rj_{r}+{|\boldsymbol{j}| \choose 2}}\beta(\boldsymbol{j})\\
 & \;\times\sum_{{0\leq m_{r}\leq k_{r}-j_{r}\atop r=1,2,\cdots,n}}\prod_{1\leq r<s\leq n}\frac{1-q^{m_{r}-m_{s}+j_{r}-j_{s}}x_{r}/x_{s}}{1-q^{j_{r}-j_{s}}x_{r}/x_{s}}q^{\sum_{r=1}^{n}rm_{r}}K(q^{\boldsymbol{m+j}})\\
 & \quad\times\prod_{r,s=1}^{n}\frac{(q^{-k_{s}+j_{r}}x_{r}/x_{s})_{m_{r}}}{(q^{1+j_{r}-j_{s}}x_{r}/x_{s})_{m_{r}}}\frac{(aq^{|\boldsymbol{k}|+|\boldsymbol{j}|},bq^{1+|\boldsymbol{j}|})_{|\boldsymbol{m}|}}{(aq)_{|\boldsymbol{m}|+|\boldsymbol{j}|}(bq^{2|\boldsymbol{j}|+1})_{|\boldsymbol{m}|}}.
\end{align*}
This indicates that the identity \eqref{eq:t1-1} holds when $y_{r}=q^{N_{r}}$
for $r=1,2,\cdots,n.$ Assume that both sides of \eqref{eq:t1-1}
are analytic in $y_{1},y_{2},...,y_{n}$ in a disk around the origin.
Note that $N_{r}=0,1,\cdots$ for $r=1,2,\cdots,n.$ Thus the identity
\eqref{eq:t1-1} is also true in the same disk around the origin by
analytic continuation. This completes the proof of Theorem \ref{t1}.
\qed

\section{\label{sec:3} Proofs of Theorems \ref{th-1}\textendash \ref{t11-2}}

\noindent{\it Proofs of Theorems \ref{th-1}--\ref{th-4}.} Taking
$\beta(\boldsymbol{j})=\prod_{r=1}^{n}\delta_{j_{r},0}$ in \eqref{eq:t1-1}
we have
\begin{equation}
\begin{aligned}K(\boldsymbol{y}) & =\sum_{{k_{r}\geq0\atop r=1,2,\cdots,n}}\prod_{1\leq r<s\leq n}\frac{1-q^{k_{r}-k_{s}}x_{r}/x_{s}}{1-x_{r}/x_{s}}\prod_{r,s=1}^{n}\frac{(x_{r}/x_{s}y_{s})_{k_{r}}}{(qx_{r}/x_{s})_{k_{r}}}\\
 & \times\frac{(1-aq^{2|\boldsymbol{k}|})(a)_{|\boldsymbol{k}|}}{(1-a)(aqy_{1}y_{2}\cdots y_{n})_{|\boldsymbol{k}|}}(y_{1}y_{2}\cdots y_{n})^{|\boldsymbol{k}|}q^{\sum_{r=1}^{n}(r-1)k_{r}}\\
 & \;\times\sum_{{0\leq m_{r}\leq k_{r}\atop r=1,2,\cdots,n}}\prod_{1\leq r<s\leq n}\frac{1-q^{m_{r}-m_{s}}x_{r}/x_{s}}{1-x_{r}/x_{s}}q^{\sum_{r=1}^{n}rm_{r}}K(q^{\boldsymbol{m}})\\
 & \quad\times\prod_{r,s=1}^{n}\frac{(q^{-k_{s}}x_{r}/x_{s})_{m_{r}}}{(qx_{r}/x_{s})_{m_{r}}}\frac{(aq^{|\boldsymbol{k}|})_{|\boldsymbol{m}|}}{(aq)_{|\boldsymbol{m}|}}.
\end{aligned}
\label{eq:a1-16}
\end{equation}
Then \eqref{eq:a1-17} follows readily by setting 
\[
K(\boldsymbol{y})=f(\alpha qy_{1}y_{2}\cdots y_{n})\frac{(\alpha q,\alpha by_{1}y_{2}\cdots y_{n})_{\infty}}{(\alpha b,\alpha qy_{1}y_{2}\cdots y_{n})_{\infty}}
\]
in \eqref{eq:a1-16}, and then replacing $a$ by $\alpha,$ $y_{r}$
by $a_{r}/q$ in the resulting identity.

Putting
\[
K(\boldsymbol{y})=f(\alpha qy_{1},\cdots,\alpha qy_{n})\frac{(\alpha q,\alpha by_{1}y_{2}\cdots y_{n})_{\infty}}{(\alpha b,\alpha qy_{1}y_{2}\cdots y_{n})_{\infty}}
\]
in \eqref{eq:a1-16}, and then replacing $a$ by $\alpha,$ $y_{r}$
by $a_{r}/q$ we easily obtain \eqref{eq:a1-18}.

Take 
\[
K(\boldsymbol{y})=f(\alpha qy_{1}y_{2}\cdots y_{n})\frac{(\alpha q)_{\infty}}{(\alpha qy_{1}y_{2}\cdots y_{n})_{\infty}}\prod_{r=1}^{n}\frac{(\alpha bx_{r}y_{r})_{\infty}}{(\alpha bx_{r})_{\infty}}
\]
in \eqref{eq:a1-16}, and then replace $a$ by $\alpha,$ $y_{r}$
by $a_{r}/q$ we can get \eqref{eq:a1-19}.

The identity \eqref{eq:a1-20} follows easily by substituting
\[
K(\boldsymbol{y})=f(\alpha qy_{1},\cdots,\alpha qy_{n})\frac{(\alpha q)_{\infty}}{(\alpha qy_{1}y_{2}\cdots y_{n})_{\infty}}\prod_{r=1}^{n}\frac{(\alpha bx_{r}y_{r})_{\infty}}{(\alpha bx_{r})_{\infty}}
\]
into \eqref{eq:a1-16}, and then replacing $a$ by $\alpha,$ $y_{r}$
by $a_{r}/q.$ \qed

\noindent{\it Proofs of Theorems \ref{t11-1}--\ref{t11-2}.} Recall
a multiple series expansion formula in \cite[Theorem 2.2]{BR}:
\begin{equation}
\begin{aligned} & K(\boldsymbol{y})\sum_{{j_{r}\geq0\atop r=1,\ldots,n}}\prod_{1\leq r<s\leq n}\frac{1-q^{j_{r}-j_{s}}x_{r}/x_{s}}{1-x_{r}/x_{s}}\prod_{r,s=1}^{n}\left(x_{r}/x_{s}y_{s}\right)_{j_{r}}\\
 & \times\prod_{r=1}^{n}\frac{1-bx_{r}q^{j_{r}+|\boldsymbol{j}|}}{(1-bx_{r})(bqx_{r}y_{r})_{|\boldsymbol{j}|}}\left(y_{1}\cdots y_{n}\right)^{|\boldsymbol{j}|}q^{\sum_{r=1}^{n}(r-1)j_{r}}\beta(\boldsymbol{j})\\
 & =\sum_{{k_{r}\geq0\atop r=1,\cdots,n}}\prod_{1\leq r<s\leq n}\frac{1-q^{k_{r}-k_{s}}x_{r}/x_{s}}{1-x_{r}/x_{s}}\prod_{r,s=1}^{n}\frac{\left(x_{r}/x_{s}y_{s}\right)_{k_{r}}}{\left(qx_{r}/x_{s}\right)_{k_{r}}}\\
 & \times\prod_{r=1}^{n}\frac{1-ax_{r}q^{k_{r}+|\boldsymbol{k}|}}{1-ax_{r}}\prod_{r=1}^{n}\frac{\left(ax_{r}\right)_{|\boldsymbol{k}|}}{\left(aqx_{r}y_{r}\right)_{|\boldsymbol{k}|}}\left(y_{1}\cdots y_{n}\right)^{|\boldsymbol{k}|}q^{\sum_{r=1}^{n}(r-1)k_{r}}\\
 & \times\sum_{{0\leq j_{r}\leq k_{r}\atop r=1,\cdots,n}}\prod_{1\leq r<s\leq n}\frac{1-q^{j_{r}-j_{s}}x_{r}/x_{s}}{1-x_{r}/x_{s}}\prod_{r,s=1}^{n}\left(q^{-k_{s}}x_{r}/x_{s}\right)_{j_{r}}
\end{aligned}
\label{eq:2-1}
\end{equation}

\[
\begin{aligned} & \;\times\prod_{r=1}^{n}\frac{\left(ax_{r}q^{|\boldsymbol{k}|},bqx_{r}\right)_{j_{r}}}{\left(bx_{r}\right)_{j_{r}+|\boldsymbol{j}|}}(-1)^{|\boldsymbol{j}|}q^{\sum_{r=1}^{n}rj_{r}+\left(_{2}^{|\boldsymbol{j}|}\right)}\beta(\boldsymbol{j})\\
 & \;\times\sum_{{0\leq m_{r}\leq k_{r}-j_{r}\atop r=1,\cdots,n}}\prod_{1\leq r<s\leq n}\frac{1-q^{j_{r}-j_{s}+m_{r}-m_{s}}x_{r}/x_{s}}{1-q^{j_{r}-j_{s}}x_{r}/x_{s}}q^{\sum_{r=1}^{n}rm_{r}}K\left(q^{\boldsymbol{j}+\boldsymbol{m}}\right)\\
 & \quad\times\prod_{r,s=1}^{n}\frac{\left(q^{j_{r}-k_{s}}x_{r}/x_{s}\right)_{m_{r}}}{\left(q^{1+j_{r}-j_{s}}x_{r}/x_{s}\right)_{m_{r}}}\prod_{r=1}^{n}\frac{\left(ax_{r}q^{j_{r}+|\boldsymbol{k}|},bx_{r}q^{1+j_{r}}\right)_{m_{r}}}{\left(aqx_{r}\right)_{j_{r}+m_{r}}\left(bx_{r}q^{1+j_{r}+|\boldsymbol{j}|}\right)_{m_{r}}}.
\end{aligned}
\]
Then the identity \eqref{eq:12-1} follows readily by taking 
\[
K(\boldsymbol{y})=f(\alpha qy_{1}y_{2}\cdots y_{n})\prod_{r=1}^{n}\frac{(\alpha qx_{r},\alpha bx_{r}y_{r})_{\infty}}{(\alpha bx_{r},\alpha qx_{r}y_{r})_{\infty}},
\]
and $\beta(\boldsymbol{j})=\prod_{r=1}^{n}\delta_{j_{r},0}$ in \eqref{eq:2-1}
and then replacing $a$ by $\alpha,$ $y_{r}$ by $a_{r}/q$ in the
resulting formula.

Setting 
\[
K(\boldsymbol{y})=f(\alpha qy_{1},\cdots,\alpha qy_{n})\prod_{r=1}^{n}\frac{(\alpha qx_{r},\alpha bx_{r}y_{r})_{\infty}}{(\alpha bx_{r},\alpha qx_{r}y_{r})_{\infty}},
\]
and $\beta(\boldsymbol{j})=\prod_{r=1}^{n}\delta_{j_{r},0}$ in \eqref{eq:2-1}
and then substituting $a\mapsto\alpha,y_{r}\mapsto a_{r}/q$ into
the resulting identity we easily obtain \eqref{eq:12-2}. \qed

\section{\label{sec:4} An $A_{n}$ Rogers' $\text{}_{6}\phi_{5}$ summation
and an $A_{n}$ extension of Sylvester's identity}

In this section, we apply the identity \eqref{eq:12-3} in Theorem
\ref{t11-3} deduce an $A_{n}$ Rogers' $\text{}_{6}\phi_{5}$ summation
formula which is equivalent to the identity in \cite[Theorem A.3]{MN12}.
From this summation formula, we derive an $A_{n}$ extension of Sylvester's
identity. This $A_{n}$ extension is equivalent to the identity in
\cite[Theorem 5.19]{M94}.

\subsection{An $A_{n}$ Rogers' $\text{}_{6}\phi_{5}$ summation formula}

Milne proved the following well-known identity \cite[Theorem 4.1]{M}:
\begin{equation}
\begin{aligned} & \frac{(c/a)_{|\boldsymbol{N}|}}{(c/ab)_{|\boldsymbol{N}|}}\prod_{i=1}^{n}\frac{(cx_{i}/b)_{N_{i}}}{(cx_{i})_{N_{i}}}\\
 & =\sum_{{0\leq y_{i}\leq N_{i}\atop i=1,\cdots,n}}\prod_{1\leq r<s\leq n}\frac{1-q^{y_{r}-y_{s}}x_{r}/x_{s}}{1-x_{r}/x_{s}}\\
 & \times\prod_{r,s=1}^{n}\frac{\left(q^{-N_{s}}x_{r}/x_{s}\right)_{y_{r}}}{\left(qx_{r}/x_{s}\right)_{y_{r}}}\prod_{i=1}^{n}\frac{(ax_{i})_{y_{i}}}{(cx_{i})_{y_{i}}}\\
 & \times\frac{(b)_{|\boldsymbol{y}|}}{(abq^{1-|\boldsymbol{N}|}/c)_{|\boldsymbol{y}|}}q^{\sum_{r=1}^{n}ry_{r}},
\end{aligned}
\label{eq:2-2}
\end{equation}
where $N_{1},N_{2},\cdots,N_{n}$ are nonnegative integers. The $N_{s}\mapsto k_{s},a\mapsto\alpha q^{|\boldsymbol{k}|},b\mapsto\alpha bc/q,c\mapsto\alpha b$
case of \eqref{eq:2-2} is
\[
\begin{aligned} & \sum_{{0\leq y_{i}\leq k_{i}\atop i=1,\cdots,n}}\prod_{1\leq r<s\leq n}\frac{1-q^{y_{r}-y_{s}}x_{r}/x_{s}}{1-x_{r}/x_{s}}\\
 & \times\prod_{r,s=1}^{n}\frac{\left(q^{-k_{s}}x_{r}/x_{s}\right)_{y_{r}}}{\left(qx_{r}/x_{s}\right)_{y_{r}}}\prod_{i=1}^{n}\frac{(\alpha q^{|\boldsymbol{k}|}x_{i})_{y_{i}}}{(\alpha bx_{i})_{y_{i}}}\\
 & \times\frac{(\alpha bc/q)_{|\boldsymbol{y}|}}{(\alpha\alpha b)_{|\boldsymbol{y}|}}q^{\sum_{r=1}^{n}ry_{r}}\\
 & =\frac{(bq^{-|\boldsymbol{k}|})_{|\boldsymbol{k}|}}{(q^{1-|\boldsymbol{k}|}/\alpha c)_{|\boldsymbol{k}|}}\prod_{i=1}^{n}\frac{(qx_{i}/c)_{k_{i}}}{(\alpha bx_{i})_{k_{i}}}\\
 & =\frac{(q/b)_{|\boldsymbol{k}|}}{(\alpha c)_{|\boldsymbol{k}|}}\left(\frac{\alpha bc}{q}\right)^{|\boldsymbol{k}|}\prod_{i=1}^{n}\frac{(qx_{i}/c)_{k_{i}}}{(\alpha bx_{i})_{k_{i}}},
\end{aligned}
\]
where, in the last equality we have used \cite[eq.(I.8)]{GR}. Letting
$\beta=bc/q,\gamma=d$ in \eqref{eq:12-3} of Theorem \ref{t11-3}
and then applying the above identity in the inner sum on the right
side of the resulting identity we get

\begin{thm}
\label{t11-5} Let $\alpha,a_{1},\cdots,a_{n},b,c$ be such that $|\alpha a_{1}\cdots a_{n}bc/q^{1+n}|<1$;
or, $a_{r}=q^{N_{r}},$ for $r=1,2,\cdots,n,$ and so the series terminate.
Suppose that none of the denominators in the following identity vanishes,
then 
\begin{equation}
\begin{aligned} & \frac{(\alpha bc/q,\alpha a_{1}\cdots a_{n}cq^{-n})_{\infty}}{(\alpha c,\alpha a_{1}\cdots a_{n}bcq^{-(1+n)})_{\infty}}\prod_{r=1}^{n}\frac{(\alpha qx_{r},\alpha a_{r}bx_{r}/q)_{\infty}}{(\alpha bx_{r},\alpha a_{r}x_{r})_{\infty}}\\
 & =\sum_{{k_{r}\geq0\atop r=1,\cdots,n}}\prod_{1\leq r<s\leq n}\frac{1-q^{k_{r}-k_{s}}x_{r}/x_{s}}{1-x_{r}/x_{s}}\prod_{r,s=1}^{n}\frac{\left(qx_{r}/x_{s}a_{s}\right)_{k_{r}}}{\left(qx_{r}/x_{s}\right)_{k_{r}}}\\
 & \times\prod_{r=1}^{n}\frac{1-\alpha x_{r}q^{k_{r}+|\boldsymbol{k}|}}{1-\alpha x_{r}}\prod_{r=1}^{n}\frac{\left(\alpha x_{r}\right)_{|\boldsymbol{k}|}}{\left(\alpha x_{r}a_{r}\right)_{|\boldsymbol{k}|}}q^{\sum_{r=1}^{n}(r-1)k_{r}}\\
 & \times\frac{(q/b)_{|\boldsymbol{k}|}}{(\alpha c)_{|\boldsymbol{k}|}}\prod_{r=1}^{n}\frac{(qx_{r}/c)_{k_{r}}}{(\alpha bx_{r})_{k_{r}}}\cdot\bigg(\frac{\alpha a_{1}\cdots a_{n}bc}{q^{1+n}}\bigg)^{|\boldsymbol{k}|}.
\end{aligned}
\label{eq:12-5}
\end{equation}
\end{thm}
This identity is an $A_{n}$ extension of Rogers' non-terminating
$_{6}\phi_{5}$ summation formula \cite[eq.(II.20)]{GR}:
\begin{align*}
 & \text{}_{6}\phi_{5}\left(\begin{matrix}\alpha,q\sqrt{\alpha},-q\sqrt{\alpha},q/a,q/b,q/c\\
\sqrt{\alpha},-\sqrt{\alpha},\alpha a,\alpha b,\alpha c
\end{matrix};q,\frac{\alpha abc}{q^{2}}\right)\\
 & =\frac{(\alpha q,\alpha ab/q,\alpha ac/q,\alpha bc/q)_{\infty}}{(\alpha a,\alpha b,\alpha c,\alpha abc/q^{2})_{\infty}},
\end{align*}
where $|\alpha abc/q^{2}|<1$ and 
\[
\text{}_{r+1}\phi_{r}\left(\begin{matrix}a_{1},a_{2},\cdots,a_{r+1}\\
b_{1},b_{2},\cdots,b_{r}
\end{matrix};q,z\right):=\sum_{n=0}^{\infty}\frac{(a_{1},a_{2},\cdots,a_{r+1})_{n}}{(q,b_{1},b_{2},\cdots,b_{r})_{n}}z^{n}.
\]

Replacing $q/a_{r}$ by $a_{r}\:(r=1,2,\cdots,n),\:q/b$ by $b,\:q/c$
by $c$ in the identity \eqref{eq:12-5} of Theorem \ref{t11-5},
we have 
\[
\begin{aligned} & \frac{(\alpha q/bc,\alpha q/a_{1}\cdots a_{n}c)_{\infty}}{(\alpha q/c,\alpha q/a_{1}\cdots a_{n}bc)_{\infty}}\prod_{r=1}^{n}\frac{(\alpha qx_{r},\alpha qx_{r}/a_{r}b)_{\infty}}{(\alpha qx_{r}/b,\alpha qx_{r}/a_{r})_{\infty}}\\
 & =\sum_{{k_{r}\geq0\atop r=1,\cdots,n}}\prod_{1\leq r<s\leq n}\frac{1-q^{k_{r}-k_{s}}x_{r}/x_{s}}{1-x_{r}/x_{s}}\prod_{r,s=1}^{n}\frac{\left(a_{s}x_{r}/x_{s}\right)_{k_{r}}}{\left(qx_{r}/x_{s}\right)_{k_{r}}}\\
 & \times\prod_{r=1}^{n}\frac{1-\alpha x_{r}q^{k_{r}+|\boldsymbol{k}|}}{1-\alpha x_{r}}\prod_{r=1}^{n}\frac{\left(\alpha x_{r}\right)_{|\boldsymbol{k}|}}{\left(\alpha qx_{r}/a_{r}\right)_{|\boldsymbol{k}|}}q^{\sum_{r=1}^{n}(r-1)k_{r}}\\
 & \times\frac{(b)_{|\boldsymbol{k}|}}{(\alpha q/c)_{|\boldsymbol{k}|}}\prod_{r=1}^{n}\frac{(cx_{r})_{k_{r}}}{(\alpha qx_{r}/b)_{k_{r}}}\cdot\bigg(\frac{\alpha q}{a_{1}\cdots a_{n}bc}\bigg)^{|\boldsymbol{k}|}.
\end{aligned}
\]
This identity is equivalent to the $A_{n}$ nonterminating $\text{}_{6}\phi_{5}$
summation formula in \cite[Theorem A.3]{MN12}.

\subsection{An $A_{n}$ extension of Sylvester's identity}

The Euler Pentagonal Number Theorem may be one of the most interesting
formulas in basic hypergeometric series, which states that \cite[eq.(8.10.10)]{GR}
\begin{equation}
(q)_{\infty}=1+\sum_{n=1}^{\infty}(-1)^{n}(q^{n(3n-1)/2}+q^{n(3n+1)/2}).\label{eq:12-6}
\end{equation}
Sylvester gave the following beautiful refinement \cite[eq.(9.2.3)]{A76}
of Euler's Pentagonal Number Theorem.
\begin{thm}
\textup{(Sylvester's identity)} We have
\[
\frac{1}{(\alpha)_{\infty}}\sum_{k=0}^{\infty}(-\alpha)^{k}q^{k^{2}+{k \choose 2}}(1-\alpha q^{2k})\frac{(\alpha)_{k}}{(q)_{k}}=1.
\]
\end{thm}
In this subsection we will derive the following $A_{n}$ extension
of Sylvester's identity.
\begin{thm}
\label{t4-4} Suppose that none of the denominators in the following
identity vanishes, then
\[
\begin{aligned}1= & \frac{1}{\prod_{r=1}^{n}(\alpha x_{r})_{\infty}}\sum_{{k_{r}\geq0\atop r=1,\cdots,n}}\prod_{1\leq r<s\leq n}\frac{1-q^{k_{r}-k_{s}}x_{r}/x_{s}}{1-x_{r}/x_{s}}\prod_{r,s=1}^{n}\left(qx_{r}/x_{s}\right)_{k_{r}}^{-1}\\
 & \times\prod_{r=1}^{n}\{(1-\alpha x_{r}q^{k_{r}+|\boldsymbol{k}|})\left(\alpha x_{r}\right)_{|\boldsymbol{k}|}\}\prod_{r=1}^{n}x_{r}^{(n+1)k_{r}-|\boldsymbol{k}|}\\
 & \times(-1)^{n|\boldsymbol{k}|}\alpha^{|\boldsymbol{k}|}q^{(n+1)\sum_{r=1}^{n}{k_{r} \choose 2}+{|\boldsymbol{k}| \choose 2}+\sum_{r=1}^{n}rk_{r}}.
\end{aligned}
\]
\end{thm}
\begin{proof}
The result follows readily by multiplying both sides of the identity
in Theorem \ref{t11-5} by $\prod_{r=1}^{n}(\alpha qx_{r})_{\infty}^{-1},$
setting $a_{r}\rightarrow0\:(r=1,2,\cdots,n),b\rightarrow0,c\rightarrow0$
in the resulting identity and then using the limiting relation $\lim_{a\rightarrow0}(u/a)_{n}a^{n}=(-u)^{n}q^{{n \choose 2}}.$
\end{proof}
Actually, this $A_{n}$ extension is equivalent to the identity in
\cite[Theorem 5.19]{M94}. Replacing $a$ by $\alpha,$ $z_{i}/z_{n}$
by $x_{i}$ on both sides of \cite[Theorem 5.19]{M94} and then multiplying
both sides of the resulting identity by $\prod_{i=1}^{n}(\alpha qx_{i})_{\infty}^{-1}$
we can easily obtain the identity in Theorem \ref{t4-4}.

Take $\alpha=1$ in the identity of Theorem \ref{t4-4} and then multiply
both sides of the resulting identity by $\prod_{r=1}^{n}(qx_{r})_{\infty}.$
We get
\[
\begin{aligned}\prod_{r=1}^{n}(qx_{r})_{\infty}= & \frac{1}{\prod_{r=1}^{n}(1-x_{r})}\sum_{{k_{r}\geq0\atop r=1,\cdots,n}}\prod_{1\leq r<s\leq n}\frac{1-q^{k_{r}-k_{s}}x_{r}/x_{s}}{1-x_{r}/x_{s}}\prod_{r,s=1}^{n}\left(qx_{r}/x_{s}\right)_{k_{r}}^{-1}\\
 & \times\prod_{r=1}^{n}\{(1-x_{r}q^{k_{r}+|\boldsymbol{k}|})\left(x_{r}\right)_{|\boldsymbol{k}|}\}\prod_{r=1}^{n}x_{r}^{(n+1)k_{r}-|\boldsymbol{k}|}\\
 & \times(-1)^{n|\boldsymbol{k}|}q^{(n+1)\sum_{r=1}^{n}{k_{r} \choose 2}+{|\boldsymbol{k}| \choose 2}+\sum_{r=1}^{n}rk_{r}}.
\end{aligned}
\]
This identity can be regarded as an $A_{n}$ extension of the Euler
Pentagonal Number Theorem \eqref{eq:12-6}.

\section{\label{sec:2} Multiple expansion formulas for $(q)_{\infty}^{m},\text{\ensuremath{\pi_{q}}}$
and $1/\pi_{q}$}

In this section, we first employ the identity \eqref{eq:a1-17} in
Theorem \ref{th-1} to deduce some multiple expansion formulas for
$(q)_{\infty}^{m}$ and then apply the identity in Theorem \ref{t21-1}
to derive several multiple expansion formulas for $\pi_{q}$ and $1/\pi_{q}.$

\subsection{Multiple expansion formulas for $(q)_{\infty}^{m}$ }

For any integer $m,$ we will give some multiple expansion formulas
for $(q)_{\infty}^{m}.$ These formulas can be divided into two groups:
the first one involves $(q)_{\infty}^{m}$ with positive integer $m$
while the second one involves $(q)_{\infty}^{m}$ with non-positive
integer $m.$ All of these expansion formulas are $A_{n}$ extensions
of some well-known identities of Liu \cite{L1}.

The first group of multiple expansion formulas for $(q)_{\infty}^{m}$
is as follows.
\begin{thm}
\label{t5-1} Let $m$ be a nonnegative integer. Then
\begin{equation}
\begin{aligned}(q)_{\infty}^{m+1} & =\sum_{{k_{r}\geq0\atop r=1,2,\cdots,n}}\prod_{1\leq r<s\leq n}\frac{1-q^{k_{r}-k_{s}}x_{r}/x_{s}}{1-x_{r}/x_{s}}\cdot\frac{(1-q^{2|\boldsymbol{k}|+1})(q)_{|\boldsymbol{k}|}}{\prod_{r,s=1}^{n}(qx_{r}/x_{s})_{k_{r}}}\\
 & \times(-1)^{n|\boldsymbol{k}|}q^{\sum_{r=1}^{n}(r-1)k_{r}+n\sum_{i=1}^{n}{k_{i} \choose 2}}\prod_{r=1}^{n}x_{r}^{nk_{r}-|\boldsymbol{k}|}\\
 & \;\times\sum_{{0\leq m_{r}\leq k_{r}\atop r=1,2,\cdots,n}}\prod_{1\leq r<s\leq n}\frac{1-q^{m_{r}-m_{s}}x_{r}/x_{s}}{1-x_{r}/x_{s}}\\
 & \quad\times\prod_{r,s=1}^{n}\frac{(q^{-k_{s}}x_{r}/x_{s})_{m_{r}}}{(qx_{r}/x_{s})_{m_{r}}}\cdot(q^{|\boldsymbol{k}|+1})_{|\boldsymbol{m}|}(q)_{|\boldsymbol{m}|}^{m}q^{\sum_{r=1}^{n}rm_{r}}
\end{aligned}
\label{eq:21-20}
\end{equation}
and 
\begin{equation}
\begin{aligned}(q)_{\infty}^{m+1} & =1+\sum_{{k_{1},k_{2},\cdots,k_{n}\geq0\atop |\boldsymbol{k}|\geq1}}\prod_{1\leq r<s\leq n}\frac{1-q^{k_{r}-k_{s}}x_{r}/x_{s}}{1-x_{r}/x_{s}}\cdot\frac{(1-q^{2|\boldsymbol{k}|})(q)_{|\boldsymbol{k}|-1}}{\prod_{r,s=1}^{n}(qx_{r}/x_{s})_{k_{r}}}\\
 & \times(-1)^{n|\boldsymbol{k}|}q^{\sum_{r=1}^{n}(r-1)k_{r}+n\sum_{i=1}^{n}{k_{i} \choose 2}}\prod_{r=1}^{n}x_{r}^{nk_{r}-|\boldsymbol{k}|}\\
 & \;\times\sum_{{0\leq m_{r}\leq k_{r}\atop r=1,2,\cdots,n}}\prod_{1\leq r<s\leq n}\frac{1-q^{m_{r}-m_{s}}x_{r}/x_{s}}{1-x_{r}/x_{s}}\\
 & \quad\times\prod_{r,s=1}^{n}\frac{(q^{-k_{s}}x_{r}/x_{s})_{m_{r}}}{(qx_{r}/x_{s})_{m_{r}}}\cdot(q^{|\boldsymbol{k}|})_{|\boldsymbol{m}|}(q)_{|\boldsymbol{m}|}^{m}\cdot q^{\sum_{r=1}^{n}rm_{r}}.
\end{aligned}
\label{eq:21-19}
\end{equation}
\end{thm}
\noindent{\it Proof.}  If 
\[
f(y)=\prod_{j=1}^{m}\frac{(b_{j}y/q)_{\infty}}{(c_{j}y/q)_{\infty}},
\]
then
\begin{equation}
f(\alpha a_{1}\cdots a_{n}q^{1-n})=\prod_{j=1}^{m}\frac{(\alpha a_{1}\cdots a_{n}b_{j}/q^{n})_{\infty}}{(\alpha a_{1}\cdots a_{n}c_{j}/q^{n})_{\infty}},\label{eq:21-1}
\end{equation}
and
\begin{equation}
f(\alpha q^{1+|\boldsymbol{m}|})=\prod_{j=1}^{m}\frac{(\alpha b_{j}q^{|\boldsymbol{m}|})_{\infty}}{(\alpha c_{j}q^{|\boldsymbol{m}|})_{\infty}}=\prod_{j=1}^{m}\frac{(\alpha b_{j})_{\infty}}{(\alpha c_{j})_{\infty}}\prod_{j=1}^{m}\frac{(\alpha c_{j})_{|\boldsymbol{m}|}}{(\alpha b_{j})_{|\boldsymbol{m}|}}.\label{eq:21-2}
\end{equation}
Setting 
\[
f(y)=\prod_{j=1}^{m}\frac{(b_{j}y/q)_{\infty}}{(c_{j}y/q)_{\infty}}
\]
in the identity \eqref{eq:a1-17} of Theorem \ref{th-1}, substituting
\eqref{eq:21-1} and \eqref{eq:21-2} into the resulting identity
and then taking $a_{1}=a_{2}=\cdots=a_{n}=0$ we get
\begin{equation}
\begin{aligned} & \frac{(\alpha q)_{\infty}}{(\alpha b)_{\infty}}\prod_{j=1}^{m}\frac{(\alpha c_{j})_{\infty}}{(\alpha b_{j})_{\infty}}\\
 & =\sum_{{k_{r}\geq0\atop r=1,2,\cdots,n}}\prod_{1\leq r<s\leq n}\frac{1-q^{k_{r}-k_{s}}x_{r}/x_{s}}{1-x_{r}/x_{s}}\cdot\frac{(1-\alpha q^{2|\boldsymbol{k}|})(\alpha)_{|\boldsymbol{k}|}}{(1-\alpha)\prod_{r,s=1}^{n}(qx_{r}/x_{s})_{k_{r}}}\\
 & \times(-1)^{n|\boldsymbol{k}|}q^{\sum_{r=1}^{n}(r-1)k_{r}+n\sum_{i=1}^{n}{k_{i} \choose 2}}\prod_{r=1}^{n}x_{r}^{nk_{r}-|\boldsymbol{k}|}\\
 & \;\times\sum_{{0\leq m_{r}\leq k_{r}\atop r=1,2,\cdots,n}}\prod_{1\leq r<s\leq n}\frac{1-q^{m_{r}-m_{s}}x_{r}/x_{s}}{1-x_{r}/x_{s}}\prod_{r,s=1}^{n}\frac{(q^{-k_{s}}x_{r}/x_{s})_{m_{r}}}{(qx_{r}/x_{s})_{m_{r}}}\\
 & \quad\times\frac{(\alpha q^{|\boldsymbol{k}|})_{|\boldsymbol{m}|}}{(\alpha b)_{|\boldsymbol{m}|}}\prod_{j=1}^{m}\frac{(\alpha c_{j})_{|\boldsymbol{m}|}}{(\alpha b_{j})_{|\boldsymbol{m}|}}q^{\sum_{r=1}^{n}rm_{r}}.
\end{aligned}
\label{eq:21-3}
\end{equation}
The $b=b_{1}=\cdots=b_{m}=0$ case of \eqref{eq:21-3} is

\begin{equation}
\begin{aligned} & (\alpha q)_{\infty}\prod_{j=1}^{m}(\alpha c_{j})_{\infty}\\
 & =\sum_{{k_{r}\geq0\atop r=1,2,\cdots,n}}\prod_{1\leq r<s\leq n}\frac{1-q^{k_{r}-k_{s}}x_{r}/x_{s}}{1-x_{r}/x_{s}}\cdot\frac{(1-\alpha q^{2|\boldsymbol{k}|})(\alpha)_{|\boldsymbol{k}|}}{(1-\alpha)\prod_{r,s=1}^{n}(qx_{r}/x_{s})_{k_{r}}}\\
 & \times(-1)^{n|\boldsymbol{k}|}q^{\sum_{r=1}^{n}(r-1)k_{r}+n\sum_{i=1}^{n}{k_{i} \choose 2}}\prod_{r=1}^{n}x_{r}^{nk_{r}-|\boldsymbol{k}|}\\
 & \;\times\sum_{{0\leq m_{r}\leq k_{r}\atop r=1,2,\cdots,n}}\prod_{1\leq r<s\leq n}\frac{1-q^{m_{r}-m_{s}}x_{r}/x_{s}}{1-x_{r}/x_{s}}\prod_{r,s=1}^{n}\frac{(q^{-k_{s}}x_{r}/x_{s})_{m_{r}}}{(qx_{r}/x_{s})_{m_{r}}}\\
 & \quad\times(\alpha q^{|\boldsymbol{k}|})_{|\boldsymbol{m}|}\prod_{j=1}^{m}(\alpha c_{j})_{|\boldsymbol{m}|}\cdot q^{\sum_{r=1}^{n}rm_{r}}.
\end{aligned}
\label{eq:tt1-6-1}
\end{equation}
Then the identities \eqref{eq:21-20} and \eqref{eq:21-19} follow
by setting $\alpha=q,c_{1}=c_{2}=\cdots=c_{m}=1$ and $\alpha\rightarrow1,c_{1}=c_{2}=\cdots=c_{m}=q$
in \eqref{eq:tt1-6-1} respectively and then simplifying. \qed

The identity \eqref{eq:21-20} in Theorem \ref{t5-1} is an $A_{n}$
extension of the well-known identity in \cite[Proposition 3.1]{L1}
while \eqref{eq:21-19} can be regarded as an $A_{n}$ extension of
the identity in \cite[Proposition 3.2]{L1}. The identities \eqref{eq:1-20}
and \eqref{eq:1-21} are the special $m=2$ case of \eqref{eq:21-20}
and \eqref{eq:21-19} respectively.

The second group of multiple expansion formulas for $(q)_{\infty}^{m}$
is showed in the following theorem.
\begin{thm}
\label{t5-2} Let $m$ be a nonnegative integer. Then
\begin{equation}
\begin{aligned}\frac{1}{(q)_{\infty}^{m}} & =\sum_{{k_{r}\geq0\atop r=1,2,\cdots,n}}\prod_{1\leq r<s\leq n}\frac{1-q^{k_{r}-k_{s}}x_{r}/x_{s}}{1-x_{r}/x_{s}}\cdot\frac{(1-q^{2|\boldsymbol{k}|+1})(q)_{|\boldsymbol{k}|}}{\prod_{r,s=1}^{n}(qx_{r}/x_{s})_{k_{r}}}\\
 & \times(-1)^{n|\boldsymbol{k}|}q^{\sum_{r=1}^{n}(r-1)k_{r}+n\sum_{i=1}^{n}{k_{i} \choose 2}}\prod_{r=1}^{n}x_{r}^{nk_{r}-|\boldsymbol{k}|}\\
 & \;\times\sum_{{0\leq m_{r}\leq k_{r}\atop r=1,2,\cdots,n}}\prod_{1\leq r<s\leq n}\frac{1-q^{m_{r}-m_{s}}x_{r}/x_{s}}{1-x_{r}/x_{s}}\\
 & \quad\times\prod_{r,s=1}^{n}\frac{(q^{-k_{s}}x_{r}/x_{s})_{m_{r}}}{(qx_{r}/x_{s})_{m_{r}}}\cdot\frac{(q^{|\boldsymbol{k}|+1})_{|\boldsymbol{m}|}}{(q)_{|\boldsymbol{m}|}^{m+1}}q^{\sum_{r=1}^{n}rm_{r}},
\end{aligned}
\label{eq:21-18}
\end{equation}
and
\begin{equation}
\begin{aligned}\frac{1}{(q)_{\infty}^{m}} & =1+\sum_{{k_{1},k_{2},\cdots,k_{n}\geq0\atop |\boldsymbol{k}|\geq1}}\prod_{1\leq r<s\leq n}\frac{1-q^{k_{r}-k_{s}}x_{r}/x_{s}}{1-x_{r}/x_{s}}\cdot\frac{(1-q^{2|\boldsymbol{k}|})(q)_{|\boldsymbol{k}|-1}}{\prod_{r,s=1}^{n}(qx_{r}/x_{s})_{k_{r}}}\\
 & \times(-1)^{n|\boldsymbol{k}|}q^{\sum_{r=1}^{n}(r-1)k_{r}+n\sum_{i=1}^{n}{k_{i} \choose 2}}\prod_{r=1}^{n}x_{r}^{nk_{r}-|\boldsymbol{k}|}\\
 & \;\times\sum_{{0\leq m_{r}\leq k_{r}\atop r=1,2,\cdots,n}}\prod_{1\leq r<s\leq n}\frac{1-q^{m_{r}-m_{s}}x_{r}/x_{s}}{1-x_{r}/x_{s}}
\end{aligned}
\label{eq:21-17}
\end{equation}
\[
\times\prod_{r,s=1}^{n}\frac{(q^{-k_{s}}x_{r}/x_{s})_{m_{r}}}{(qx_{r}/x_{s})_{m_{r}}}\cdot\frac{(q^{|\boldsymbol{k}|})_{|\boldsymbol{m}|}}{(q)_{|\boldsymbol{m}|}^{m+1}}q^{\sum_{r=1}^{n}rm_{r}}.
\]
\end{thm}
\noindent{\it Proof.} Putting $c_{1}=\cdots=c_{m}=0$ in \eqref{eq:21-3}
we arrive at

\begin{equation}
\begin{aligned} & \frac{(\alpha q)_{\infty}}{(\alpha b)_{\infty}\prod_{j=1}^{m}(\alpha b_{j})_{\infty}}\\
 & =\sum_{{k_{r}\geq0\atop r=1,2,\cdots,n}}\prod_{1\leq r<s\leq n}\frac{1-q^{k_{r}-k_{s}}x_{r}/x_{s}}{1-x_{r}/x_{s}}\cdot\frac{(1-\alpha q^{2|\boldsymbol{k}|})(\alpha)_{|\boldsymbol{k}|}}{(1-\alpha)\prod_{r,s=1}^{n}(qx_{r}/x_{s})_{k_{r}}}\\
 & \times(-1)^{n|\boldsymbol{k}|}q^{\sum_{r=1}^{n}(r-1)k_{r}+n\sum_{i=1}^{n}{k_{i} \choose 2}}\prod_{r=1}^{n}x_{r}^{nk_{r}-|\boldsymbol{k}|}\\
 & \;\times\sum_{{0\leq m_{r}\leq k_{r}\atop r=1,2,\cdots,n}}\prod_{1\leq r<s\leq n}\frac{1-q^{m_{r}-m_{s}}x_{r}/x_{s}}{1-x_{r}/x_{s}}\prod_{r,s=1}^{n}\frac{(q^{-k_{s}}x_{r}/x_{s})_{m_{r}}}{(qx_{r}/x_{s})_{m_{r}}}\\
 & \quad\times\frac{(\alpha q^{|\boldsymbol{k}|})_{|\boldsymbol{m}|}}{(\alpha b)_{|\boldsymbol{m}|}\prod_{j=1}^{m}(\alpha b_{j})_{|\boldsymbol{m}|}}q^{\sum_{r=1}^{n}rm_{r}}.
\end{aligned}
\label{eq:tt1-7-1}
\end{equation}
We take $\alpha=q,b=b_{1}=\cdots=b_{m}=1$ and $\alpha\rightarrow1,b=b_{1}=\cdots=b_{m}=q$
in \eqref{eq:tt1-7-1} to obtain \eqref{eq:21-18} and \eqref{eq:21-17}
respectively. \qed

The identity \eqref{eq:21-18} is an $A_{n}$ extension of the identity
in \cite[Proposition 3.3]{L1} while \eqref{eq:21-17} can be considered
as an $A_{n}$ extension of the identity in \cite[Proposition 3.4]{L1}.
The identity \eqref{eq:1-23} is the special $m=2$ case of  \eqref{eq:21-17}.

\subsection{Multiple expansion formulas for $\pi_{q}$ and $1/\pi_{q}$}

In \cite{GR}, W. Gosper introduced a $q$-analogue of the irrational
number $\pi:$
\[
\pi_{q}:=(1-q^{2})q^{1/4}\dfrac{(q^{2};q^{2})_{\infty}^{2}}{(q;q^{2})_{\infty}^{2}}.
\]
It satisfies the following limiting relation:
\[
\lim_{q\rightarrow1}\pi_{q}=\pi.
\]
In this subsection, we give some multiple expansion formulas for $\pi_{q}$
and $1/\pi_{q}.$
\begin{thm}
We have
\[
\begin{aligned}\pi_{q} & =(1-q^{2})\sum_{{k_{r}\geq0\atop r=1,2,\cdots,n}}\prod_{1\leq r<s\leq n}\frac{1-q^{2k_{r}-2k_{s}}x_{r}/x_{s}}{1-x_{r}/x_{s}}\prod_{r=1}^{n}x_{r}^{nk_{r}-|\boldsymbol{k}|}\\
 & \times\frac{(1-q^{4|\boldsymbol{k}|+2})(q^{2};q^{2})_{|\boldsymbol{k}|}}{\prod_{r,s=1}^{n}(q^{2}x_{r}/x_{s};q^{2})_{k_{r}}}\cdot(-1)^{n|\boldsymbol{k}|}q^{2\sum_{r=1}^{n}(r-1)k_{r}+2n\sum_{r=1}^{n}{k_{r} \choose 2}+\frac{1}{4}}\\
 & \;\times\sum_{{0\leq m_{r}\leq k_{r}\atop r=1,2,\cdots,n}}\prod_{1\leq r<s\leq n}\frac{1-q^{2m_{r}-2m_{s}}x_{r}/x_{s}}{1-x_{r}/x_{s}}\\
 & \quad\times\prod_{r,s=1}^{n}\frac{(q^{-2k_{s}}x_{r}/x_{s};q^{2})_{m_{r}}}{(q^{2}x_{r}/x_{s};q^{2})_{m_{r}}}\cdot\frac{(q^{2|\boldsymbol{k}|+2},q^{2};q^{2})_{|\boldsymbol{m}|}}{(q,q;q^{2})_{|\boldsymbol{m}|}}q^{2\sum_{r=1}^{n}rm_{r}},
\end{aligned}
\]
and 
\[
\begin{aligned}\pi_{q} & =(1-q^{2})q^{1/4}\\
 & +(1-q^{2})\sum_{{k_{r}\geq0\atop |\boldsymbol{k}|\geq1}}\prod_{1\leq r<s\leq n}\frac{1-q^{2k_{r}-2k_{s}}x_{r}/x_{s}}{1-x_{r}/x_{s}}\cdot\frac{(1-q^{4|\boldsymbol{k}|})(q^{2};q^{2})_{|\boldsymbol{k}|-1}}{\prod_{r,s=1}^{n}(q^{2}x_{r}/x_{s};q^{2})_{k_{r}}}\\
 & \times(-1)^{n|\boldsymbol{k}|}q^{2\sum_{r=1}^{n}(r-1)k_{r}+2n\sum_{r=1}^{n}{k_{r} \choose 2}+\frac{1}{4}}\prod_{r=1}^{n}x_{r}^{nk_{r}-|\boldsymbol{k}|}\\
 & \;\times\sum_{{0\leq m_{r}\leq k_{r}\atop r=1,2,\cdots,n}}\prod_{1\leq r<s\leq n}\frac{1-q^{2m_{r}-2m_{s}}x_{r}/x_{s}}{1-x_{r}/x_{s}}\\
 & \quad\times\prod_{r,s=1}^{n}\frac{(q^{-2k_{s}}x_{r}/x_{s};q^{2})_{m_{r}}}{(q^{2}x_{r}/x_{s};q^{2})_{m_{r}}}\cdot\frac{(q^{2|\boldsymbol{k}|},q^{2};q^{2})_{|\boldsymbol{m}|}}{(q,q;q^{2})_{|\boldsymbol{m}|}}q^{2\sum_{r=1}^{n}rm_{r}}.
\end{aligned}
\]
\end{thm}
\noindent{\it Proof.} Replacing $q$ by $q^{2}$ in the identity
of Theorem \ref{t21-1} and then taking $a_{1}=\cdots=a_{n}=0$ and
$\gamma=d$ we have 
\begin{equation}
\begin{aligned} & \frac{(\alpha q^{2},\alpha\beta;q^{2})_{\infty}}{(\alpha b,\alpha c;q^{2})_{\infty}}\\
 & =\sum_{{k_{r}\geq0\atop r=1,2,\cdots,n}}\prod_{1\leq r<s\leq n}\frac{1-q^{2k_{r}-2k_{s}}x_{r}/x_{s}}{1-x_{r}/x_{s}}\cdot\frac{(1-\alpha q^{4|\boldsymbol{k}|})(\alpha;q^{2})_{|\boldsymbol{k}|}}{(1-\alpha)\prod_{r,s=1}^{n}(q^{2}x_{r}/x_{s};q^{2})_{k_{r}}}
\end{aligned}
\label{eq:10-11}
\end{equation}
\begin{align*}
 & \times(-1)^{n|\boldsymbol{k}|}q^{2\sum_{r=1}^{n}(r-1)k_{r}+2n\sum_{r=1}^{n}{k_{r} \choose 2}}\prod_{r=1}^{n}x_{r}^{nk_{r}-|\boldsymbol{k}|}\\
 & \;\times\sum_{{0\leq m_{r}\leq k_{r}\atop r=1,2,\cdots,n}}\prod_{1\leq r<s\leq n}\frac{1-q^{2m_{r}-2m_{s}}x_{r}/x_{s}}{1-x_{r}/x_{s}}\\
 & \quad\times\prod_{r,s=1}^{n}\frac{(q^{-2k_{s}}x_{r}/x_{s};q^{2})_{m_{r}}}{(q^{2}x_{r}/x_{s};q^{2})_{m_{r}}}\cdot\frac{(\alpha q^{2|\boldsymbol{k}|},\alpha\beta;q^{2})_{|\boldsymbol{m}|}}{(\alpha b,\alpha c;q^{2})_{|\boldsymbol{m}|}}q^{2\sum_{r=1}^{n}rm_{r}}.
\end{align*}
Then the first identity follows easily by setting $\alpha\mapsto q^{2},\beta\mapsto1,b\mapsto q^{-1},c\mapsto q^{-1}$
in \eqref{eq:10-11} and then multiplying both sides of the resulting
identity by $(1-q^{2})q^{1/4}.$

We put $\alpha\mapsto1,\beta\mapsto q^{2},b\mapsto q,c\mapsto q$
in \eqref{eq:10-11} and then multiply both sides of the resulting
identity by $(1-q^{2})q^{1/4}$ to get the second identity. \qed

\begin{thm}
We have 
\[
\begin{aligned}1/\pi_{q} & =\frac{1}{1-q^{2}}\sum_{{k_{r}\geq0\atop r=1,2,\cdots,n}}\prod_{1\leq r<s\leq n}\frac{1-q^{2k_{r}-2k_{s}}x_{r}/x_{s}}{1-x_{r}/x_{s}}\cdot\frac{(1-q^{4|\boldsymbol{k}|+1})(q;q^{2})_{|\boldsymbol{k}|}}{\prod_{r,s=1}^{n}(q^{2}x_{r}/x_{s};q^{2})_{k_{r}}}\\
 & \times(-1)^{n|\boldsymbol{k}|}q^{2\sum_{r=1}^{n}(r-1)k_{r}+2n\sum_{r=1}^{n}{k_{r} \choose 2}-\frac{1}{4}}\prod_{r=1}^{n}x_{r}^{nk_{r}-|\boldsymbol{k}|}\\
 & \;\times\sum_{{0\leq m_{r}\leq k_{r}\atop r=1,2,\cdots,n}}\prod_{1\leq r<s\leq n}\frac{1-q^{2m_{r}-2m_{s}}x_{r}/x_{s}}{1-x_{r}/x_{s}}\\
 & \quad\times\prod_{r,s=1}^{n}\frac{(q^{-2k_{s}}x_{r}/x_{s};q^{2})_{m_{r}}}{(q^{2}x_{r}/x_{s};q^{2})_{m_{r}}}\cdot\frac{(q^{2|\boldsymbol{k}|+1},q;q^{2})_{|\boldsymbol{m}|}}{(q^{2},q^{2};q^{2})_{|\boldsymbol{m}|}}q^{2\sum_{r=1}^{n}rm_{r}},
\end{aligned}
\]
\end{thm}
\begin{proof}
This identity follows readily by taking $\alpha\mapsto q,\beta\mapsto1,b\mapsto q,c\mapsto q$
in \eqref{eq:10-11} and then multiplying both sides of the resulting
identity by $(1-q^{2})^{-1}q^{-1/4}.$
\end{proof}

\section{\label{sec:7} Two $A_{n}$ extensions of the Rogers-Fine identity}

The Rogers-Fine identity \cite[eq.(14.1)]{Fi} may be one of the fundamental
formulas in basic hypergeometric series. It can be stated in the following
form:
\[
\sum_{n=0}^{\infty}\frac{(a)_{n}}{(b)_{n}}z^{n}=\sum_{n=0}^{\infty}\frac{(1-azq^{2n})(a,azq/b)_{n}}{(b)_{n}(z)_{n+1}}(bz)^{n}q^{n^{2}-n},
\]
where $|z|<1.$ This identity was first proved by L.J. Rogers \cite{R}
and also discovered by N.J. Fine \cite{Fi} independently. In \cite{Fi},
Fine used it to deduce many important results in partitions, modular
equations and mock theta functions. In \cite{L02}, Liu gave an alternate
proof of this result by using an important expansion formula \cite[Theorem 2]{L02}
for $q$-series.

In this section, we will establish two $A_{n}$ extensions of the
Rogers-Fine identity by applying the identity \eqref{eq:t1-1} in
Theorem \ref{t1}.
\begin{thm}
\textup{\label{t8-1} (First $A_{n}$ extension of the Rogers-Fine
identity)} Let $a,z$ be such that $|z/a^{n-1}|<1.$ Suppose that
none of the denominators in the following identity vanishes, then
\begin{align*}
 & \sum_{{j_{r}\geq0\atop r=1,2,\cdots,n}}\prod_{1\leq r<s\leq n}\frac{1-q^{j_{r}-j_{s}}x_{r}/x_{s}}{1-x_{r}/x_{s}}\prod_{r,s=1}^{n}\frac{(ax_{r}/x_{s})_{j_{r}}}{(qx_{r}/x_{s})_{j_{r}}}\\
 & \times\frac{(q)_{|\boldsymbol{j}|}}{(b)_{|\boldsymbol{j}|}}(z/a^{n-1})^{|\boldsymbol{j}|}q^{\sum_{r=1}^{n}(r-1)j_{r}}\\
 & =\sum_{{k_{r}\geq0\atop r=1,2,\cdots,n}}\prod_{1\leq r<s\leq n}\frac{1-q^{k_{r}-k_{s}}x_{r}/x_{s}}{1-x_{r}/x_{s}}\prod_{r,s=1}^{n}\frac{(ax_{r}/x_{s})_{k_{r}}}{(qx_{r}/x_{s})_{k_{r}}}\\
 & \times\frac{(1-azq^{2|\boldsymbol{k}|})(q,azq/b)_{|\boldsymbol{k}|}}{(b)_{|\boldsymbol{k}|}(z/a^{n-1})_{|\boldsymbol{k}|+1}}(bz/a^{n-1})^{|\boldsymbol{k}|}q^{\sum_{r=1}^{n}(r-1)k_{r}+2{|\boldsymbol{k}| \choose 2}}.
\end{align*}
\end{thm}
\noindent{\it Proof.} We recall a $U(n+1)$ $q$-Gauss summation
in \cite[Theorem 7.6]{M}:
\begin{align*}
 & \sum_{{y_{k}\geq0\atop k=1,2,\cdots,n}}\prod_{1\leq r<s\leq n}\frac{1-\tfrac{x_{r}}{x_{s}}q^{y_{r}-y_{s}}}{1-\tfrac{x_{r}}{x_{s}}}\prod_{r,s=1}^{n}\frac{\left(\tfrac{x_{r}}{x_{s}}a_{s}\right)_{y_{r}}}{\left(q\tfrac{x_{r}}{x_{s}}\right)_{y_{r}}}\\
 & \quad\times\frac{(b)_{|\boldsymbol{y}|}}{(c)_{|\boldsymbol{y}|}}\left(\frac{c}{a_{1}\cdots a_{n}b}\right)^{|\boldsymbol{y}|}q^{y_{2}+2y_{3}+\cdots+(n-1)y_{n}}\\
 & =\frac{(c/b,c/a_{1}\cdots a_{n})_{\infty}}{(c,c/a_{1}\cdots a_{n}b)_{\infty}},
\end{align*}
where $|c|<|a_{1}\cdots a_{n}b|.$ If $a_{i}=q^{-N_{i}}$ for $i=1,2,\cdots,n,$
then the above identity becomes
\begin{align*}
 & \sum_{{0\leq y_{k}\leq N_{k}\atop k=1,2,\cdots,n}}\prod_{1\leq r<s\leq n}\frac{1-\tfrac{x_{r}}{x_{s}}q^{y_{r}-y_{s}}}{1-\tfrac{x_{r}}{x_{s}}}\prod_{r,s=1}^{n}\frac{\left(\tfrac{x_{r}}{x_{s}}q^{-N_{s}}\right)_{y_{r}}}{\left(q\tfrac{x_{r}}{x_{s}}\right)_{y_{r}}}\\
 & \quad\times\frac{(b)_{|\boldsymbol{y}|}}{(c)_{|\boldsymbol{y}|}}\left(\frac{cq^{|\boldsymbol{N}|}}{b}\right)^{|\boldsymbol{y}|}q^{y_{2}+2y_{3}+\cdots+(n-1)y_{n}}\\
 & =\frac{(c/b)_{|\boldsymbol{N}|}}{(c)_{|\boldsymbol{N}|}}.
\end{align*}
Applying 
\[
(A;q)_{n}=(A^{-1};q^{-1})_{n}(-A)^{n}q^{{n \choose 2}}
\]
in this formula, replacing each parameter, including $q$ and $x_{1},\cdots,x_{n},$
by its reciprocal and then simplifying we get
\begin{equation}
\begin{aligned} & \sum_{{0\leq y_{r}\leq N_{r}\atop r=1,2,\cdots,n}}\prod_{1\leq r<s\leq n}\frac{1-\tfrac{x_{r}}{x_{s}}q^{y_{r}-y_{s}}}{1-\tfrac{x_{r}}{x_{s}}}\\
 & \quad\times\prod_{r,s=1}^{n}\frac{\left(\tfrac{x_{r}}{x_{s}}q^{-N_{s}}\right)_{y_{r}}}{\left(q\tfrac{x_{r}}{x_{s}}\right)_{y_{r}}}\frac{(b)_{|\boldsymbol{y}|}}{(c)_{|\boldsymbol{y}|}}q^{y_{1}+2y_{2}+\cdots+ny_{n}}\\
 & =\frac{(c/b)_{|\boldsymbol{N}|}}{(c)_{|\boldsymbol{N}|}}b^{|\boldsymbol{N}|}.
\end{aligned}
\label{eq:8-1}
\end{equation}
Taking $N_{r}\mapsto k_{r}-j_{r},x_{r}\mapsto q^{j_{r}}x_{r},b\mapsto azq^{|\boldsymbol{k}|+|\boldsymbol{j}|},c\mapsto azq^{|\boldsymbol{j}|}$
in \eqref{eq:8-1} we have
\begin{equation}
\begin{aligned} & \sum_{{0\leq m_{r}\leq k_{r}-j_{r}\atop r=1,2,\cdots,n}}\prod_{1\leq r<s\leq n}\frac{1-q^{m_{r}-m_{s}+j_{r}-j_{s}}x_{r}/x_{s}}{1-q^{j_{r}-j_{s}}x_{r}/x_{s}}q^{\sum_{r=1}^{n}rm_{r}}\\
 & \times\prod_{r,s=1}^{n}\frac{(q^{-k_{s}+j_{r}}x_{r}/x_{s})_{m_{r}}}{(q^{1+j_{r}-j_{s}}x_{r}/x_{s})_{m_{r}}}\frac{(azq^{|\boldsymbol{k}|+|\boldsymbol{j}|})_{|\boldsymbol{m}|}}{(azq^{|\boldsymbol{j}|})_{|\boldsymbol{m}|}}\\
 & =\frac{(q^{-|\boldsymbol{k}|})_{|\boldsymbol{k}|-|\boldsymbol{j}|}}{(azq^{|\boldsymbol{j}|})_{|\boldsymbol{k}|-|\boldsymbol{j}|}}(azq^{|\boldsymbol{k}|+|\boldsymbol{j}|})^{|\boldsymbol{k}|-|\boldsymbol{j}|}\\
 & =\frac{(q)_{|\boldsymbol{k}|}(-az)^{|\boldsymbol{k}|-|\boldsymbol{j}|}}{(azq^{|\boldsymbol{j}|})_{|\boldsymbol{k}|-|\boldsymbol{j}|}(q)_{|\boldsymbol{j}|}}q^{{|\boldsymbol{k}| \choose 2}-{|\boldsymbol{j}| \choose 2}},
\end{aligned}
\label{eq:8-2}
\end{equation}
where, in the last equality we have used the identity \cite[eq.(I.14)]{GR}.

Setting $N_{r}\mapsto k_{r},b\mapsto azq^{|\boldsymbol{k}|},c\mapsto b$
in \eqref{eq:8-1} we see that
\begin{equation}
\begin{aligned} & \sum_{{0\leq j_{r}\leq k_{r}\atop r=1,2,\cdots,n}}\prod_{1\leq r<s\leq n}\frac{1-q^{j_{r}-j_{s}}x_{r}/x_{s}}{1-x_{r}/x_{s}}\\
 & \times\prod_{r,s=1}^{n}\frac{(q^{-k_{s}}x_{r}/x_{s})_{j_{r}}}{(qx_{r}/x_{s})_{j_{r}}}\frac{(azq^{|\boldsymbol{k}|})_{|\boldsymbol{j}|}}{(b)_{|\boldsymbol{j}|}}q^{\sum_{r=1}^{n}rj_{r}}\\
 & =\frac{(bq^{-|\boldsymbol{k}|}/az)_{|\boldsymbol{k}|}}{(b)_{|\boldsymbol{k}|}}(azq^{|\boldsymbol{k}|})^{|\boldsymbol{k}|}\\
 & =\frac{(azq/b)_{|\boldsymbol{k}|}}{(b)_{|\boldsymbol{k}|}}\left(-b\right)^{|\boldsymbol{k}|}q^{{|\boldsymbol{k}| \choose 2}},
\end{aligned}
\label{eq:8-3}
\end{equation}
where, in the last equality we applied the formula \cite[eq.(I.8)]{GR}.

Replacing $a$ by $az$ in the $b=0$ case of \eqref{eq:t1-1} in
Theorem \ref{t1} and then taking
\[
K(y)=\frac{1-azy_{1}\cdots y_{n}}{1-az}
\]
 and
\[
\beta(\boldsymbol{j})=\frac{(q)_{|\boldsymbol{j}|}(az)^{|\boldsymbol{j}|}}{(b)_{|\boldsymbol{j}|}\prod_{r,s=1}^{n}(qx_{r}/x_{s})_{j_{r}}}
\]
we obtain
\[
\begin{aligned} & \frac{1-azy_{1}\cdots y_{n}}{1-az}\sum_{{j_{r}\geq0\atop r=1,2,\cdots,n}}\prod_{1\leq r<s\leq n}\frac{1-q^{j_{r}-j_{s}}x_{r}/x_{s}}{1-x_{r}/x_{s}}\\
 & \quad\times\prod_{r,s=1}^{n}\frac{(x_{r}/x_{s}y_{s})_{j_{r}}}{(qx_{r}/x_{s})_{j_{r}}}\frac{(q)_{|\boldsymbol{j}|}}{(b)_{|\boldsymbol{j}|}}(azy_{1}y_{2}\cdots y_{n})^{|\boldsymbol{j}|}q^{\sum_{r=1}^{n}(r-1)j_{r}}\\
 & =\sum_{{k_{r}\geq0\atop r=1,2,\cdots,n}}\prod_{1\leq r<s\leq n}\frac{1-q^{k_{r}-k_{s}}x_{r}/x_{s}}{1-x_{r}/x_{s}}\prod_{r,s=1}^{n}\frac{(x_{r}/x_{s}y_{s})_{k_{r}}}{(qx_{r}/x_{s})_{k_{r}}}
\end{aligned}
\]
\[
\begin{aligned} & \times\frac{(1-azq^{2|\boldsymbol{k}|})(az)_{|\boldsymbol{k}|}}{(1-az)(azqy_{1}y_{2}\cdots y_{n})_{|\boldsymbol{k}|}}(y_{1}y_{2}\cdots y_{n})^{|\boldsymbol{k}|}q^{\sum_{r=1}^{n}(r-1)k_{r}}\\
 & \times\sum_{{0\leq j_{r}\leq k_{r}\atop r=1,2,\cdots,n}}\prod_{1\leq r<s\leq n}\frac{1-q^{j_{r}-j_{s}}x_{r}/x_{s}}{1-x_{r}/x_{s}}\prod_{r,s=1}^{n}\frac{(q^{-k_{s}}x_{r}/x_{s})_{j_{r}}}{(qx_{r}/x_{s})_{j_{r}}}\\
 & \;\times\frac{(azq^{|\boldsymbol{k}|},q)_{|\boldsymbol{j}|}}{(b,az)_{|\boldsymbol{j}|}}(-az)^{|\boldsymbol{j}|}q^{\sum_{r=1}^{n}rj_{r}+{|\boldsymbol{j}| \choose 2}}\\
 & \;\times\sum_{{0\leq m_{r}\leq k_{r}-j_{r}\atop r=1,2,\cdots,n}}\prod_{1\leq r<s\leq n}\frac{1-q^{m_{r}-m_{s}+j_{r}-j_{s}}x_{r}/x_{s}}{1-q^{j_{r}-j_{s}}x_{r}/x_{s}}q^{\sum_{r=1}^{n}rm_{r}}\\
 & \quad\times\prod_{r,s=1}^{n}\frac{(q^{-k_{s}+j_{r}}x_{r}/x_{s})_{m_{r}}}{(q^{1+j_{r}-j_{s}}x_{r}/x_{s})_{m_{r}}}\frac{(azq^{|\boldsymbol{k}|+|\boldsymbol{j}|})_{|\boldsymbol{m}|}}{(azq^{|\boldsymbol{j}|})_{|\boldsymbol{m}|}}.
\end{aligned}
\]

Substituting \eqref{eq:8-2} into the inner-most sum $\sum_{{0\leq m_{r}\leq k_{r}-j_{r}\atop r=1,2,\cdots,n}}$
on the right side of the above identity and then applying \eqref{eq:8-3}
in the second sum $\sum_{{0\leq j_{r}\leq k_{r}\atop r=1,2,\cdots,n}}$
on the right side of the resulting identity yields
\begin{align*}
 & \frac{1-azy_{1}\cdots y_{n}}{1-az}\sum_{{j_{r}\geq0\atop r=1,2,\cdots,n}}\prod_{1\leq r<s\leq n}\frac{1-q^{j_{r}-j_{s}}x_{r}/x_{s}}{1-x_{r}/x_{s}}\prod_{r,s=1}^{n}\frac{(x_{r}/x_{s}y_{s})_{j_{r}}}{(qx_{r}/x_{s})_{j_{r}}}\\
 & \quad\times\frac{(q)_{|\boldsymbol{j}|}}{(b)_{|\boldsymbol{j}|}}(azy_{1}y_{2}\cdots y_{n})^{|\boldsymbol{j}|}q^{\sum_{r=1}^{n}(r-1)j_{r}}\\
 & =\sum_{{k_{r}\geq0\atop r=1,2,\cdots,n}}\prod_{1\leq r<s\leq n}\frac{1-q^{k_{r}-k_{s}}x_{r}/x_{s}}{1-x_{r}/x_{s}}\prod_{r,s=1}^{n}\frac{(x_{r}/x_{s}y_{s})_{k_{r}}}{(qx_{r}/x_{s})_{k_{r}}}\\
 & \times\frac{(1-azq^{2|\boldsymbol{k}|})(q,azq/b)_{|\boldsymbol{k}|}}{(1-az)(b,azqy_{1}y_{2}\cdots y_{n})_{|\boldsymbol{k}|}}(abzy_{1}y_{2}\cdots y_{n})^{|\boldsymbol{k}|}q^{\sum_{r=1}^{n}(r-1)k_{r}+2{|\boldsymbol{k}| \choose 2}}.
\end{align*}
Then the result follows by multiplying both sides of this identity
by $\frac{1-az}{1-azy_{1}\cdots y_{n}}$ and then setting $y_{1}=y_{2}=\cdots=y_{n}=1/a.$
\qed

A second $A_{n}$ extension of the Rogers-Fine identity is as follows.
\begin{thm}
\textup{\label{t8-2} (Second $A_{n}$ extension of the Rogers-Fine
identity)} Let $a,z$ be such that $|z/a^{n-1}|<1.$ Suppose that
none of the denominators in the following identity vanishes, then
\[
\begin{aligned} & \sum_{{j_{r}\geq0\atop r=1,2,\cdots,n}}\prod_{1\leq r<s\leq n}\frac{1-q^{j_{r}-j_{s}}x_{r}/x_{s}}{1-x_{r}/x_{s}}\prod_{r,s=1}^{n}\frac{(ax_{r}/x_{s})_{j_{r}}}{(qx_{r}/x_{s})_{j_{r}}}\\
 & \quad\times\frac{(q)_{|\boldsymbol{j}|}}{\prod_{i=1}^{n}(bx_{i})_{j_{i}}}(z/a^{n-1})^{|\boldsymbol{j}|}q^{\sum_{r=1}^{n}(r-1)j_{r}}\\
 & =\sum_{{k_{r}\geq0\atop r=1,2,\cdots,n}}\prod_{1\leq r<s\leq n}\frac{1-q^{k_{r}-k_{s}}x_{r}/x_{s}}{1-x_{r}/x_{s}}\prod_{r,s=1}^{n}\frac{(ax_{r}/x_{s})_{k_{r}}}{(qx_{r}/x_{s})_{k_{r}}}
\end{aligned}
\]
\[
\begin{aligned} & \times\frac{(1-azq^{2|\boldsymbol{k}|})(q)_{|\boldsymbol{k}|}}{(z/a^{n-1})_{|\boldsymbol{k}|+1}}\prod_{i=1}^{n}\frac{(azq^{1+|\boldsymbol{k}|-k_{i}}/bx_{i})_{k_{i}}}{(bx_{i})_{k_{i}}}\\
 & \times(bz/a^{n-1})^{|\boldsymbol{k}|}q^{\sum_{r=1}^{n}(r-1)k_{r}+{|\boldsymbol{k}| \choose 2}+\sum_{i=1}^{n}{k_{i} \choose 2}}\prod_{i=1}^{n}x_{i}^{k_{i}}.
\end{aligned}
\]
\end{thm}
\begin{proof}
We first take $a\mapsto az,b\mapsto0,K(y)\mapsto\frac{1-azy_{1}\cdots y_{n}}{1-az}$
and 
\[
\beta(\boldsymbol{j})\mapsto\frac{(q)_{|\boldsymbol{j}|}(az)^{|\boldsymbol{j}|}}{\prod_{i=1}^{n}(bx_{i})_{j_{i}}\prod_{r,s=1}^{n}(qx_{r}/x_{s})_{j_{r}}}
\]
in \eqref{eq:t1-1} of Theorem \ref{t1}, then apply \eqref{eq:8-2}
and the identity in \cite[Theorem 5.10]{M} and finally set $y_{i}\mapsto1/a$
for $i=1,2,\cdots,n$ to obtain this $A_{n}$ extension of the Rogers-Fine
identity.
\end{proof}

\section{\label{sec:8} An $A_{n}$ extension of Liu's extension of Rogers'
non-terminating $_{6}\phi_{5}$ summation}

Rogers' non-terminating $_{6}\phi_{5}$ summation \cite[eq.(II.20)]{GR}
may be one of the most important formulas in $q$-series. It can be
stated in the following equivalent form.
\begin{thm}
For $|\alpha abc/q^{2}|<1,$ we have
\begin{align*}
 & \text{}_{6}\phi_{5}\left(\begin{matrix}\alpha,q\sqrt{\alpha},-q\sqrt{\alpha},q/a,q/b,q/c\\
\sqrt{\alpha},-\sqrt{\alpha},\alpha a,\alpha b,\alpha c
\end{matrix};q,\frac{\alpha abc}{q^{2}}\right)\\
 & =\frac{(\alpha q,\alpha ab/q,\alpha ac/q,\alpha bc/q)_{\infty}}{(\alpha a,\alpha b,\alpha c,\alpha abc/q^{2})_{\infty}}.
\end{align*}
\end{thm}
Using the operator method, Liu proved the following extension \cite[Theorem 3]{L11}
of the Rogers non-terminating $_{6}\phi_{5}$ summation formula.
\begin{thm}
\label{t7-1} For $\max\{|\alpha\beta abc/q^{2}|,|\alpha\gamma abc/q^{2}|\}<1,$
we have
\begin{align*}
 & \sum_{n=0}^{\infty}\frac{(1-\alpha q^{2n})(\alpha,q/a,q/b,q/c)_{n}}{(q,\alpha a,\alpha b,\alpha c)_{n}}\bigg(\frac{\alpha abc}{q^{2}}\bigg)^{n}\text{}_{4}\phi_{3}\left(\begin{matrix}q^{-n},\alpha q^{n},\beta,\gamma\\
q/a,q/b,\alpha\beta\gamma abc/q
\end{matrix};q,q\right)\\
 & =\frac{(\alpha,\alpha ac/q,\alpha bc/q,\alpha\beta ab/q,\alpha\gamma ab/q,\alpha\beta\gamma abc/q^{2})_{\infty}}{(\alpha a,\alpha b,\alpha c,\alpha\beta abc/q^{2},\alpha\gamma abc/q^{2},\alpha\beta\gamma ab/q)_{\infty}}.
\end{align*}
\end{thm}
Applying an expansion theorem \cite[Proposition 1.6]{L3} for the
analytic functions, Liu provided a simple proof of the above summation
formula.

In this section, we will establish an $A_{n}$ extension of Liu's
extension  by using the identity \eqref{eq:12-1} in Theorem \ref{t11-2}.
\begin{thm}
\label{t7-2} Let $\alpha,\beta,\gamma,a,b,c_{1},\cdots,c_{n}$ be
such that $\max\{|\alpha\beta abc_{1}\cdots c_{n}/q^{n+1}|,$\\
$|\alpha\gamma abc_{1}\cdots c_{n}/q^{n+1}|\}<1$; or, $c_{r}=q^{N_{r}},$
for $r=1,2,\cdots,n,$ and so the series terminate. Suppose that none
of the denominators in the following identity vanishes, then
\begin{align*}
 & \frac{(\alpha\beta ab/q,\alpha\gamma ab/q,\alpha ac_{1}\cdots c_{n}/q^{n},\alpha bc_{1}\cdots c_{n}/q^{n})_{\infty}}{(\alpha a,\alpha b,\alpha\beta abc_{1}\cdots c_{n}/q^{n+1},\alpha\gamma abc_{1}\cdots c_{n}/q^{n+1})_{\infty}}\\
 & \times\prod_{i=1}^{n}\frac{(\alpha x_{i},\alpha\beta\gamma abc_{i}x_{i}/q^{2})_{\infty}}{(\alpha\beta\gamma abx_{i}/q,\alpha c_{i}x_{i})_{\infty}}\\
 & =\sum_{{k_{r}\geq0\atop r=1,\cdots,n}}\prod_{1\leq r<s\leq n}\frac{1-q^{k_{r}-k_{s}}x_{r}/x_{s}}{1-x_{r}/x_{s}}\prod_{r,s=1}^{n}\frac{\left(qx_{r}/x_{s}c_{s}\right)_{k_{r}}}{\left(qx_{r}/x_{s}\right)_{k_{r}}}\\
 & \times\prod_{r=1}^{n}\frac{(1-\alpha x_{r}q^{k_{r}+|\boldsymbol{k}|})\left(\alpha x_{r}\right)_{|\boldsymbol{k}|}}{\left(\alpha c_{r}x_{r}\right)_{|\boldsymbol{k}|}}q^{\sum_{r=1}^{n}(r-1)k_{r}}\\
 & \times(\alpha a,\alpha b)_{|\boldsymbol{k}|}^{-1}\prod_{i=1}^{n}(q/ax_{i},qx_{i}/b)_{k_{i}}\\
 & \times\left(\frac{\alpha abc_{1}\cdots c_{n}}{q^{n+1}}\right)^{|\boldsymbol{k}|}q^{e_{2}(k_{1},k_{2}\cdots,k_{n})}\prod_{i=1}^{n}x_{i}^{k_{i}}\\
 & \times\sum_{{0\leq m_{i}\leq k_{i}\atop r=1,\cdots,n}}\prod_{1\leq r<s\leq n}\frac{1-q^{m_{r}-m_{s}}x_{r}/x_{s}}{1-x_{r}/x_{s}}q^{\sum_{r=1}^{n}rm_{r}}\\
 & \;\times\prod_{r,s=1}^{n}\frac{\left(q^{-k_{s}}x_{r}/x_{s}\right)_{m_{r}}}{\left(qx_{r}/x_{s}\right)_{m_{r}}}\prod_{i=1}^{n}\frac{(\alpha q^{|\boldsymbol{k}|}x_{i},\beta x_{i},\gamma x_{i})_{m_{i}}}{(\alpha\beta\gamma abx_{i}/q,qx_{i}/a,qx_{i}/b)_{m_{i}}},
\end{align*}
where $e_{2}(k_{1},\cdots,k_{n})$ is the second elementary symmetric
function of $\{k_{1},\cdots,k_{n}\}.$
\end{thm}
\begin{proof}
Replacing $a_{r}$ by $qy_{r}$ for $r=1,2,\cdots,n$ in the identity
\eqref{eq:12-1} of Theorem \ref{t11-2} we have
\begin{equation}
\begin{aligned} & f(\alpha qy_{1},\cdots,\alpha qy_{n})\prod_{i=1}^{n}\frac{(\alpha qx_{i},\alpha bx_{i}y_{i})_{\infty}}{(\alpha bx_{i},\alpha qx_{i}y_{i})_{\infty}}\\
 & =\sum_{{k_{r}\geq0\atop r=1,\cdots,n}}\prod_{1\leq r<s\leq n}\frac{1-q^{k_{r}-k_{s}}x_{r}/x_{s}}{1-x_{r}/x_{s}}\prod_{r,s=1}^{n}\frac{\left(x_{r}/x_{s}y_{s}\right)_{k_{r}}}{\left(qx_{r}/x_{s}\right)_{k_{r}}}\\
 & \times\prod_{r=1}^{n}\frac{1-\alpha x_{r}q^{k_{r}+|\boldsymbol{k}|}}{1-\alpha x_{r}}\prod_{r=1}^{n}\frac{\left(\alpha x_{r}\right)_{|\boldsymbol{k}|}}{\left(\alpha qx_{r}y_{r}\right)_{|\boldsymbol{k}|}}\left(y_{1}\cdots y_{n}\right)^{|\boldsymbol{k}|}q^{\sum_{r=1}^{n}(r-1)k_{r}}\\
 & \times\sum_{{0\leq m_{r}\leq k_{r}\atop r=1,\cdots,n}}\prod_{1\leq r<s\leq n}\frac{1-q^{m_{r}-m_{s}}x_{r}/x_{s}}{1-x_{r}/x_{s}}\prod_{r,s=1}^{n}\frac{\left(q^{-k_{s}}x_{r}/x_{s}\right)_{m_{r}}}{\left(qx_{r}/x_{s}\right)_{m_{r}}}\\
 & \;\times\prod_{r=1}^{n}\frac{\left(\alpha x_{r}q^{|\boldsymbol{k}|}\right)_{m_{r}}}{(\alpha bx_{r})_{m_{r}}}q^{\sum_{r=1}^{n}rm_{r}}f(\alpha q^{1+m_{1}},\cdots,\alpha q^{1+m_{n}}).
\end{aligned}
\label{eq:7-1}
\end{equation}
Now we choose 
\begin{align*}
f(y_{1},\cdots,y_{n}) & =\frac{(A_{2},A_{3},B_{2}y_{1}\cdots y_{n}/\alpha^{n}q^{n},B_{3}y_{1}\cdots y_{n}/\alpha^{n}q^{n})_{\infty}}{(B_{2},B_{3},A_{2}y_{1}\cdots y_{n}/\alpha^{n}q^{n},A_{3}y_{1}\cdots y_{n}/\alpha^{n}q^{n})_{\infty}}\\
 & \;\cdot\prod_{i=1}^{n}\frac{(A_{1}x_{i},B_{1}x_{i}y_{i}/\alpha q,\alpha bx_{i},x_{i}y_{i})_{\infty}}{(B_{1}x_{i},A_{1}x_{i}y_{i}/\alpha q,\alpha qx_{i},bx_{i}y_{i}/q)_{\infty}}.
\end{align*}
Then
\begin{equation}
\begin{aligned}f(\alpha qy_{1},\cdots,\alpha qy_{n}) & =\frac{(A_{2},A_{3},B_{2}y_{1}\cdots y_{n},B_{3}y_{1}\cdots y_{n})_{\infty}}{(B_{2},B_{3},A_{2}y_{1}\cdots y_{n},A_{3}y_{1}\cdots y_{n})_{\infty}}\\
 & \;\cdot\prod_{i=1}^{n}\frac{(A_{1}x_{i},B_{1}x_{i}y_{i},\alpha bx_{i},\alpha qx_{i}y_{i})_{\infty}}{(B_{1}x_{i},A_{1}x_{i}y_{i},\alpha qx_{i},\alpha bx_{i}y_{i})_{\infty}}
\end{aligned}
\label{eq:7-2}
\end{equation}
and
\begin{equation}
f(\alpha q^{1+m_{1}},\cdots,\alpha q^{1+m_{n}})=\frac{(A_{2},A_{3})_{|\boldsymbol{m}|}}{(B_{2},B_{3})_{|\boldsymbol{m}|}}\prod_{i=1}^{n}\frac{(A_{1}x_{i},\alpha bx_{i})_{m_{i}}}{(B_{1}x_{i},\alpha qx_{i})_{m_{i}}}.\label{eq:7-3}
\end{equation}
Substituting \eqref{eq:7-2} and \eqref{eq:7-3} into \eqref{eq:7-1}
and then simplifying we get
\[
\begin{aligned} & \frac{(A_{2},A_{3},B_{2}y_{1}\cdots y_{n},B_{3}y_{1}\cdots y_{n})_{\infty}}{(B_{2},B_{3},A_{2}y_{1}\cdots y_{n},A_{3}y_{1}\cdots y_{n})_{\infty}}\prod_{i=1}^{n}\frac{(A_{1}x_{i},B_{1}x_{i}y_{i})_{\infty}}{(B_{1}x_{i},A_{1}x_{i}y_{i})_{\infty}}\\
 & =\sum_{{k_{r}\geq0\atop r=1,\cdots,n}}\prod_{1\leq r<s\leq n}\frac{1-q^{k_{r}-k_{s}}x_{r}/x_{s}}{1-x_{r}/x_{s}}\prod_{r,s=1}^{n}\frac{\left(x_{r}/x_{s}y_{s}\right)_{k_{r}}}{\left(qx_{r}/x_{s}\right)_{k_{r}}}\\
 & \times\prod_{r=1}^{n}\frac{1-\alpha x_{r}q^{k_{r}+|\boldsymbol{k}|}}{1-\alpha x_{r}}\prod_{r=1}^{n}\frac{\left(\alpha x_{r}\right)_{|\boldsymbol{k}|}}{\left(\alpha qx_{r}y_{r}\right)_{|\boldsymbol{k}|}}\left(y_{1}\cdots y_{n}\right)^{|\boldsymbol{k}|}q^{\sum_{r=1}^{n}(r-1)k_{r}}\\
 & \times\sum_{{0\leq m_{r}\leq k_{r}\atop r=1,\cdots,n}}\prod_{1\leq r<s\leq n}\frac{1-q^{m_{r}-m_{s}}x_{r}/x_{s}}{1-x_{r}/x_{s}}\prod_{r,s=1}^{n}\frac{\left(q^{-k_{s}}x_{r}/x_{s}\right)_{m_{r}}}{\left(qx_{r}/x_{s}\right)_{m_{r}}}q^{\sum_{r=1}^{n}rm_{r}}\\
 & \;\times\frac{(A_{2},A_{3})_{|\boldsymbol{m}|}}{(B_{2},B_{3})_{|\boldsymbol{m}|}}\prod_{r=1}^{n}\frac{\left(\alpha x_{r}q^{|\boldsymbol{k}|},A_{1}x_{r}\right)_{m_{r}}}{(B_{1}x_{r},\alpha qx_{r})_{m_{r}}}.
\end{aligned}
\]
Taking $A_{1}\mapsto\alpha q,A_{2}\mapsto\alpha\beta ab/q,A_{3}\mapsto\alpha\gamma ab/q,B_{1}\mapsto\alpha\beta\gamma ab/q,B_{2}\mapsto\alpha a,B_{3}\mapsto\alpha b$
in the above identity and then multiplying both sides of the resulting
identity by $\prod_{i=1}^{n}(1-\alpha x_{i})$ we arrive at
\begin{equation}
\begin{aligned} & \frac{(\alpha\beta ab/q,\alpha\gamma ab/q,\alpha ay_{1}\cdots y_{n},\alpha by_{1}\cdots y_{n})_{\infty}}{(\alpha a,\alpha b,\alpha\beta aby_{1}\cdots y_{n}/q,\alpha\gamma aby_{1}\cdots y_{n}/q)_{\infty}}\prod_{i=1}^{n}\frac{(\alpha x_{i},\alpha\beta\gamma abx_{i}y_{i}/q)_{\infty}}{(\alpha\beta\gamma abx_{i}/q,\alpha qx_{i}y_{i})_{\infty}}\\
 & =\sum_{{k_{r}\geq0\atop r=1,\cdots,n}}\prod_{1\leq r<s\leq n}\frac{1-q^{k_{r}-k_{s}}x_{r}/x_{s}}{1-x_{r}/x_{s}}\prod_{r,s=1}^{n}\frac{\left(x_{r}/x_{s}y_{s}\right)_{k_{r}}}{\left(qx_{r}/x_{s}\right)_{k_{r}}}\\
 & \times\prod_{r=1}^{n}\frac{(1-\alpha x_{r}q^{k_{r}+|\boldsymbol{k}|})\left(\alpha x_{r}\right)_{|\boldsymbol{k}|}}{\left(\alpha qx_{r}y_{r}\right)_{|\boldsymbol{k}|}}\left(y_{1}\cdots y_{n}\right)^{|\boldsymbol{k}|}q^{\sum_{r=1}^{n}(r-1)k_{r}}\\
 & \;\times\sum_{{0\leq m_{r}\leq k_{r}\atop r=1,\cdots,n}}\prod_{1\leq r<s\leq n}\frac{1-q^{m_{r}-m_{s}}x_{r}/x_{s}}{1-x_{r}/x_{s}}\prod_{r,s=1}^{n}\frac{\left(q^{-k_{s}}x_{r}/x_{s}\right)_{m_{r}}}{\left(qx_{r}/x_{s}\right)_{m_{r}}}q^{\sum_{r=1}^{n}rm_{r}}\\
 & \quad\times\frac{(\alpha\beta ab/q,\alpha\gamma ab/q)_{|\boldsymbol{m}|}}{(\alpha a,\alpha b)_{|\boldsymbol{m}|}}\prod_{r=1}^{n}\frac{\left(\alpha x_{r}q^{|\boldsymbol{k}|}\right)_{m_{r}}}{\left(\alpha\beta\gamma abx_{r}/q\right)_{m_{r}}}.
\end{aligned}
\label{eq:c1-3}
\end{equation}
It follows from \cite[eq.(I.8)]{GR} that 
\[
(ax_{i}q^{-k_{i}},bq^{-k_{i}}/x_{i})_{k_{i}}=(q/ax_{i},qx_{i}/b)_{k_{i}}(ab)^{k_{i}}q^{-k_{i}-k_{i}^{2}}.
\]
Then
\begin{equation}
\prod_{i=1}^{n}(ax_{i}q^{-k_{i}},bq^{-k_{i}}/x_{i})_{k_{i}}=(ab)^{|\boldsymbol{k}|}q^{-|k|-\sum_{i=1}^{n}k_{i}^{2}}\prod_{i=1}^{n}(q/ax_{i},qx_{i}/b)_{k_{i}}.\label{eq:c1-1}
\end{equation}

Recall a $U(n+1)$ Sears $_{4}\phi_{3}$ transformation \cite[Theorem 5.2]{MN}:
\begin{align*}
 & \sum_{{0\leq y_{r}\leq N_{r}\atop r=1,\cdots,n}}\prod_{1\leq r<s\leq n}\frac{1-q^{y_{r}-y_{s}}x_{r}/x_{s}}{1-x_{r}/x_{s}}q^{\sum_{r=1}^{n}rj_{r}}\\
 & \times\prod_{r,s=1}^{n}\frac{\left(q^{-N_{s}}x_{r}/x_{s}\right)_{y_{r}}}{\left(qx_{r}/x_{s}\right)_{y_{r}}}\prod_{r=1}^{n}\frac{\left(ax_{r}\right)_{y_{r}}}{(dx_{r})_{y_{r}}}\cdot\frac{(b,c)_{|\boldsymbol{y}|}}{(e,f)_{|\boldsymbol{y}|}}\\
 & =(e,f)_{|\boldsymbol{N}|}^{-1}\prod_{i=1}^{n}(ex_{i}q^{|\boldsymbol{N}|-N_{i}}/a,fq^{|\boldsymbol{N}|-N_{i}}/ax_{i})_{N_{i}}\\
 & \times a^{|\boldsymbol{N}|}q^{-e_{2}(N_{1},N_{2}\cdots,N_{n})}\prod_{i=1}^{n}x_{i}^{N_{i}}\\
 & \times\sum_{{0\leq y_{i}\leq N_{i}\atop r=1,\cdots,n}}\prod_{1\leq r<s\leq n}\frac{1-q^{y_{r}-y_{s}}x_{r}/x_{s}}{1-x_{r}/x_{s}}q^{\sum_{r=1}^{n}ry_{r}}\\
 & \;\times\prod_{r,s=1}^{n}\frac{\left(q^{-N_{s}}x_{r}/x_{s}\right)_{y_{r}}}{\left(qx_{r}/x_{s}\right)_{y_{r}}}\prod_{i=1}^{n}\frac{(ax_{i},\:dx_{i}/b,dx_{i}/c)_{y_{i}}}{(dx_{i},ax_{i}q^{1-|\boldsymbol{N}|}/e,ax_{i}q^{1-|\boldsymbol{N}|}/f)_{y_{i}}},
\end{align*}
where $abc=defq^{|\boldsymbol{N}|-1}$ and $e_{2}(N_{1},N_{2},\cdots,N_{n})$
is the second elementary symmetric function of $\{N_{1},N_{2},\cdots,N_{n}\}.$
Applying the $U(n+1)$ Sears $_{4}\phi_{3}$ transformation with $N_{r}\mapsto k_{r},a\mapsto\alpha q^{|\boldsymbol{k}|},b\mapsto\alpha\beta ab/q,c\mapsto\alpha\gamma ab/q,d\mapsto\alpha\beta\gamma ab/q,e\mapsto\alpha a,f\mapsto\alpha b$
and then substituting \eqref{eq:c1-1} into the resulting identity
we get
\begin{equation}
\begin{aligned} & \sum_{{0\leq m_{r}\leq k_{r}\atop r=1,\cdots,n}}\prod_{1\leq r<s\leq n}\frac{1-q^{m_{r}-m_{s}}x_{r}/x_{s}}{1-x_{r}/x_{s}}\prod_{r,s=1}^{n}\frac{\left(q^{-k_{s}}x_{r}/x_{s}\right)_{m_{r}}}{\left(qx_{r}/x_{s}\right)_{m_{r}}}q^{\sum_{r=1}^{n}rm_{r}}\\
 & \quad\times\frac{(\alpha\beta ab/q,\alpha\gamma ab/q)_{|\boldsymbol{m}|}}{(\alpha a,\alpha b)_{|\boldsymbol{m}|}}\prod_{r=1}^{n}\frac{\left(\alpha x_{r}q^{|\boldsymbol{k}|}\right)_{m_{r}}}{\left(\alpha\beta\gamma abx_{r}/q\right)_{m_{r}}}\\
 & =(\alpha a,\alpha b)_{|\boldsymbol{k}|}^{-1}\prod_{i=1}^{n}(q/ax_{i},qx_{i}/b)_{k_{i}}\\
 & \times(\alpha ab/q)^{|\boldsymbol{k}|}q^{e_{2}(k_{1},k_{2}\cdots,k_{n})}\prod_{i=1}^{n}x_{i}^{k_{i}}\\
 & \times\sum_{{0\leq m_{i}\leq k_{i}\atop r=1,\cdots,n}}\prod_{1\leq r<s\leq n}\frac{1-q^{m_{r}-m_{s}}x_{r}/x_{s}}{1-x_{r}/x_{s}}q^{\sum_{r=1}^{n}rm_{r}}\\
 & \;\times\prod_{r,s=1}^{n}\frac{\left(q^{-k_{s}}x_{r}/x_{s}\right)_{m_{r}}}{\left(qx_{r}/x_{s}\right)_{m_{r}}}\prod_{i=1}^{n}\frac{(\alpha q^{|\boldsymbol{k}|}x_{i},\beta x_{i},\gamma x_{i})_{m_{i}}}{(\alpha\beta\gamma abx_{i}/q,qx_{i}/a,qx_{i}/b)_{m_{i}}}.
\end{aligned}
\label{eq:c1-2}
\end{equation}
Then the result follows by substituting \eqref{eq:c1-2} into \eqref{eq:c1-3}
and then replacing $y_{i}$ by $c_{i}/q$ for $i=1,2,\cdots,n$ in
the resulting identity. 
\end{proof}
The summation formula in Theorem \ref{t7-2} is an $A_{n}$ extension
of Liu's extension of the Rogers non-terminating $_{6}\phi_{5}$ summation
formula; replace $x_{i}$ by $x_{i}/x_{n}$ in the $n=1$ case of
the identity in Theorem \ref{t7-2}.

If $\gamma=0,$ then the identity in Theorem \ref{t7-2} becomes the
following summation formula.
\begin{thm}
\label{t7-3} Let $\alpha,\beta,a,b,c_{1},\cdots,c_{n}$ be such that
$|\alpha\beta abc_{1}\cdots c_{n}/q^{n+1}|$$<1$; or, $c_{r}=q^{N_{r}},$
for $r=1,2,\cdots,n,$ and so the series terminate. Suppose that none
of the denominators in the following identity vanishes, then
\begin{align*}
 & \frac{(\alpha\beta ab/q,\alpha ac_{1}\cdots c_{n}/q^{n},\alpha bc_{1}\cdots c_{n}/q^{n})_{\infty}}{(\alpha a,\alpha b,\alpha\beta abc_{1}\cdots c_{n}/q^{n+1})_{\infty}}\prod_{i=1}^{n}\frac{(\alpha x_{i})_{\infty}}{(\alpha c_{i}x_{i})_{\infty}}\\
 & =\sum_{{k_{r}\geq0\atop r=1,\cdots,n}}\prod_{1\leq r<s\leq n}\frac{1-q^{k_{r}-k_{s}}x_{r}/x_{s}}{1-x_{r}/x_{s}}\prod_{r,s=1}^{n}\frac{\left(qx_{r}/x_{s}c_{s}\right)_{k_{r}}}{\left(qx_{r}/x_{s}\right)_{k_{r}}}\\
 & \times\prod_{r=1}^{n}\frac{(1-\alpha x_{r}q^{k_{r}+|\boldsymbol{k}|})\left(\alpha x_{r}\right)_{|\boldsymbol{k}|}}{\left(\alpha c_{r}x_{r}\right)_{|\boldsymbol{k}|}}q^{\sum_{r=1}^{n}(r-1)k_{r}}\\
 & \times(\alpha a,\alpha b)_{|\boldsymbol{k}|}^{-1}\prod_{i=1}^{n}(q/ax_{i},qx_{i}/b)_{k_{i}}\\
 & \times\left(\frac{\alpha abc_{1}\cdots c_{n}}{q^{n+1}}\right)^{|\boldsymbol{k}|}q^{e_{2}(k_{1},k_{2}\cdots,k_{n})}\prod_{i=1}^{n}x_{i}^{k_{i}}\\
 & \times\sum_{{0\leq m_{i}\leq k_{i}\atop r=1,\cdots,n}}\prod_{1\leq r<s\leq n}\frac{1-q^{m_{r}-m_{s}}x_{r}/x_{s}}{1-x_{r}/x_{s}}q^{\sum_{r=1}^{n}rm_{r}}\\
 & \;\times\prod_{r,s=1}^{n}\frac{\left(q^{-k_{s}}x_{r}/x_{s}\right)_{m_{r}}}{\left(qx_{r}/x_{s}\right)_{m_{r}}}\prod_{i=1}^{n}\frac{(\alpha q^{|\boldsymbol{k}|}x_{i},\beta x_{i})_{m_{i}}}{(qx_{i}/a,qx_{i}/b)_{m_{i}}},
\end{align*}
where $e_{2}(k_{1},\cdots,k_{n})$ is the second elementary symmetric
function of $\{k_{1},\cdots,k_{n}\}.$
\end{thm}
The identity in Theorem \ref{t7-3} is an $A_{n}$ extension of the
following summation formula:

\begin{align*}
 & \sum_{n=0}^{\infty}\frac{(1-\alpha q^{2n})(\alpha,q/a,q/b,q/c)_{n}}{(q,\alpha a,\alpha b,\alpha c)_{n}}\bigg(\frac{\alpha abc}{q^{2}}\bigg)^{n}\text{}_{3}\phi_{2}\left(\begin{matrix}q^{-n},\alpha q^{n},\beta\\
q/a,q/b
\end{matrix};q,q\right)\\
 & =\frac{(\alpha,\alpha ac/q,\alpha bc/q,\alpha\beta ab/q)_{\infty}}{(\alpha a,\alpha b,\alpha c,\alpha\beta abc/q^{2})_{\infty}},
\end{align*}
where $|\alpha\beta abc/q^{2}|<1.$ This formula is equivalent to
a summation theorem due to Ismail, Rahman and Suslov \cite[Theorem 5.1]{IRS}.

\section{\label{sec:9} An $A_{n}$ extension of a generalization of Fang's
identity}

Fang \cite{F07} applied a $q$-exponential operator constructed in
\cite{F07-2} to the Rogers-Fine identity to establish the following
interesting identity \cite[eq.(9)]{F07}:  

\begin{equation}
\begin{aligned}\sum_{n=0}^{M}\frac{(q^{-M},c)_{n}}{(\beta,bc)_{n}}\tau^{n} & =\sum_{n=0}^{M}\frac{(q^{-M},q^{1-M}\tau/\beta,c)_{n}(1-\tau q^{2n-M})}{(\beta,bc)_{n}(\tau)_{n+1}}(\beta\tau)^{n}q^{n^{2}-n}\\
 & \;\times\text{}_{3}\phi_{2}\left(\begin{matrix}q^{-n},b,q^{1-n}/\beta\\
q^{1-M}\tau/\beta,q^{1-n}/c
\end{matrix};q,q\right),
\end{aligned}
\label{eq:9-1}
\end{equation}
where $M$ is a nonnegative integer. Applying this $q$-exponential
operator repeatedly, Fang obtained a very general identity, from which
an extension of the terminating very-well-poised $_{6}\phi_{5}$ summation
formula was derived.  The formula \eqref{eq:9-1} is important in
Fang's paper \cite{F07}. It can be extended to the following non-terminating
identity:
\begin{equation}
\begin{aligned}\sum_{n=0}^{\infty}\frac{(1/y,c)_{n}}{(\beta,bc)_{n}}(y\tau)^{n} & =\sum_{n=0}^{\infty}\frac{(1/y,q\tau/\beta,c)_{n}(1-\tau q^{2n})}{(\beta,bc)_{n}(y\tau)_{n+1}}(y\beta\tau)^{n}q^{n^{2}-n}\\
 & \;\times\text{}_{3}\phi_{2}\left(\begin{matrix}q^{-n},b,q^{1-n}/\beta\\
q\tau/\beta,q^{1-n}/c
\end{matrix};q,q\right),
\end{aligned}
\label{eq:9-2}
\end{equation}
where $|y\tau|<1.$ Applying the terminating $_{3}\phi_{2}$ transformation
formula in \cite[eq.(III.13)]{GR} to the inner sum $_{3}\phi_{2}$
on the right side of \eqref{eq:9-2}, we obtain an equivalent identity:
\begin{equation}
\begin{aligned}\sum_{n=0}^{\infty}\frac{(1/y,c)_{n}}{(\beta,bc)_{n}}(y\tau)^{n} & =\sum_{n=0}^{\infty}\frac{(1/y,q\tau/\beta)_{n}(1-\tau q^{2n})}{(\beta)_{n}(y\tau)_{n+1}}(y\beta\tau)^{n}q^{n^{2}-n}\\
 & \;\times\text{}_{3}\phi_{2}\left(\begin{matrix}q^{-n},\tau q^{n},b\\
q\tau/\beta,bc
\end{matrix};q,cq/\beta\right),
\end{aligned}
\label{eq:9-5}
\end{equation}
where $|y\tau|<1.$ Let $y\mapsto q^{M},\tau\mapsto\tau/q^{M}.$ Then
the formula \eqref{eq:9-2} reduces to \eqref{eq:9-1}. Thus, \eqref{eq:9-2}
and \eqref{eq:9-5} are both generalizations of Fang's identity \eqref{eq:9-1}.

In this section, we deduce an $A_{n}$ extension of the identity \eqref{eq:9-5}
using the identity \eqref{eq:t1-1} in Theorem \ref{t1}.
\begin{thm}
\label{t9-1} Let $y_{1},y_{2},\cdots,y_{n},\tau$ be such that $|y_{1}y_{2}\cdots y_{n}\tau|<1.$
Suppose that none of the denominators in the following identity vanishes,
then
\[
\begin{aligned} & \sum_{{j_{r}\geq0\atop r=1,2,\cdots,n}}\prod_{1\leq r<s\leq n}\frac{1-q^{j_{r}-j_{s}}x_{r}/x_{s}}{1-x_{r}/x_{s}}\prod_{r,s=1}^{n}\frac{(x_{r}/x_{s}y_{s})_{j_{r}}}{(qx_{r}/x_{s})_{j_{r}}}\\
 & \quad\times(y_{1}y_{2}\cdots y_{n}\tau)^{|\boldsymbol{j}|}q^{\sum_{r=1}^{n}(r-1)j_{r}}\frac{(q,c)_{|\boldsymbol{j}|}}{(bc)_{|\boldsymbol{j}|}}\prod_{r=1}^{n}(\beta x_{r})_{j_{r}}^{-1}\\
 & =\sum_{{k_{r}\geq0\atop r=1,2,\cdots,n}}\prod_{1\leq r<s\leq n}\frac{1-q^{k_{r}-k_{s}}x_{r}/x_{s}}{1-x_{r}/x_{s}}\prod_{r,s=1}^{n}\frac{(x_{r}/x_{s}y_{s})_{k_{r}}}{(qx_{r}/x_{s})_{k_{r}}}\\
 & \times\frac{(1-\tau q^{2|\boldsymbol{k}|})(q)_{|\boldsymbol{k}|}}{(y_{1}y_{2}\cdots y_{n}\tau)_{|\boldsymbol{k}|+1}}(y_{1}y_{2}\cdots y_{n}\beta\tau)^{|\boldsymbol{k}|}q^{\sum_{r=1}^{n}(r-1)k_{r}+{|\boldsymbol{k}| \choose 2}+\sum_{r=1}^{n}{k_{r} \choose 2}}\\
 & \times\prod_{r=1}^{n}\frac{(q^{1+|\boldsymbol{k}|-k_{r}}\tau/\beta x_{r})_{k_{r}}}{(\beta x_{r})_{k_{r}}}\prod_{r=1}^{n}x_{r}^{k_{r}}\sum_{{0\leq j_{r}\leq k_{r}\atop r=1,\cdots,n}}\prod_{1\leq r<s\leq n}\frac{1-q^{k_{s}-k_{r}+j_{r}-j_{s}}x_{s}/x_{r}}{1-q^{k_{s}-k_{r}}x_{s}/x_{r}}
\end{aligned}
\]
\[
\begin{aligned} & \;\times q^{|\boldsymbol{k}|\boldsymbol{|j}|-\sum_{r=1}^{n}k_{r}j_{r}-e_{2}(j_{1},\cdots,j_{n})+\sum_{r=1}^{n}rj_{r}}\\
 & \;\times\prod_{r,s=1}^{n}\frac{\left(q^{-k_{r}}x_{s}/x_{r}\right)_{j_{r}}}{\left(q^{1+k_{s}-k_{r}}x_{s}/x_{r}\right)_{j_{r}}}\prod_{r=1}^{n}\big\{ x_{r}^{-j_{r}}\left(\tau q^{1-k_{r}+|\boldsymbol{k}|}/\beta x_{r}\right)_{j_{r}}^{-1}\big\}\\
 & \;\times\bigg(\frac{c}{\beta}\bigg)^{\boldsymbol{|j}|}\frac{(\tau q^{|\boldsymbol{k}|})_{|\boldsymbol{j}|}(b)_{|\boldsymbol{j}|}}{(bc)_{|\boldsymbol{j}|}},
\end{aligned}
\]
where $e_{2}(j_{1},\cdots,j_{n})$ is the second elementary symmetric
function of $\{j_{1},\cdots,j_{n}\}.$
\end{thm}
\noindent{\it Proof.} Recall a relation between (6.5a) and (6.5c)
in \cite[Theorem 6.5]{ML}:
\begin{align*}
 & \sum_{{0\leq j_{r}\leq k_{r}\atop r=1,\cdots,n}}\prod_{1\leq r<s\leq n}\frac{1-q^{j_{r}-j_{s}}x_{r}/x_{s}}{1-x_{r}/x_{s}}\prod_{r,s=1}^{n}\frac{\left(q^{-k_{s}}x_{r}/x_{s}\right)_{j_{r}}}{\left(qx_{r}/x_{s}\right)_{j_{r}}}\\
 & \times\prod_{r=1}^{n}\frac{\left(cx_{r}\right)_{j_{r}}}{(ex_{r})_{j_{r}}}\frac{(a)_{|\boldsymbol{j}|}(b)_{|\boldsymbol{j}|}}{(d)_{|\boldsymbol{j}|}(f)_{|\boldsymbol{j}|}}q^{\sum_{r=1}^{n}rj_{r}}\\
 & =a^{|\boldsymbol{k}|}\frac{(f/a)_{|\boldsymbol{k}|}}{(f)_{|\boldsymbol{k}|}}\prod_{r=1}^{n}\frac{(ex_{r}/a)_{k_{r}}}{(ex_{r})_{k_{r}}}\\
 & \times\sum_{{0\leq j_{r}\leq k_{r}\atop r=1,\cdots,n}}\prod_{1\leq r<s\leq n}\frac{1-q^{k_{s}-k_{r}+j_{r}-j_{s}}x_{s}/x_{r}}{1-q^{k_{s}-k_{r}}x_{s}/x_{r}}q^{\sum_{r=1}^{n}rj_{r}}\\
 & \;\times\prod_{r,s=1}^{n}\frac{\left(q^{-k_{r}}x_{s}/x_{r}\right)_{j_{r}}}{\left(q^{1+k_{s}-k_{r}}x_{s}/x_{r}\right)_{j_{r}}}\prod_{r=1}^{n}\frac{\left(dq^{|\boldsymbol{k}|-k_{r}}/cx_{r}\right)_{j_{r}}}{\left(aq^{1-k_{r}}/ex_{r}\right)_{j_{r}}}\\
 & \;\times\frac{(a)_{|\boldsymbol{j}|}(d/b)_{|\boldsymbol{j}|}}{(d)_{|\boldsymbol{j}|}(aq^{1-|\boldsymbol{k}|}/f)_{|\boldsymbol{j}|}},
\end{align*}
where $def=abcq^{1-|\boldsymbol{N}|}.$ This is actually an $A_{n}$
Sears transformation formula. Substituting $f=abcq^{1-|\boldsymbol{N}|}/de$
into the above identity and then taking $c\rightarrow0$ we obtain
an $A_{n}$ terminating $_{3}\phi_{2}$ transformation formula:
\[
\begin{aligned} & \sum_{{0\leq j_{r}\leq k_{r}\atop r=1,\cdots,n}}\prod_{1\leq r<s\leq n}\frac{1-q^{j_{r}-j_{s}}x_{r}/x_{s}}{1-x_{r}/x_{s}}\prod_{r,s=1}^{n}\frac{\left(q^{-k_{s}}x_{r}/x_{s}\right)_{j_{r}}}{\left(qx_{r}/x_{s}\right)_{j_{r}}}\\
 & \times\prod_{r=1}^{n}(ex_{r})_{j_{r}}^{-1}\frac{(a)_{|\boldsymbol{j}|}(b)_{|\boldsymbol{j}|}}{(d)_{|\boldsymbol{j}|}}q^{\sum_{r=1}^{n}rj_{r}}\\
 & =a^{|\boldsymbol{k}|}\prod_{r=1}^{n}\frac{(ex_{r}/a)_{k_{r}}}{(ex_{r})_{k_{r}}}\sum_{{0\leq j_{r}\leq k_{r}\atop r=1,\cdots,n}}\prod_{1\leq r<s\leq n}\frac{1-q^{k_{s}-k_{r}+j_{r}-j_{s}}x_{s}/x_{r}}{1-q^{k_{s}-k_{r}}x_{s}/x_{r}}\\
 & \;\times q^{|\boldsymbol{k}|\boldsymbol{|j}|-\sum_{r=1}^{n}k_{r}j_{r}-e_{2}(j_{1},\cdots,j_{n})+\sum_{r=1}^{n}rj_{r}}\\
 & \;\times\prod_{r,s=1}^{n}\frac{\left(q^{-k_{r}}x_{s}/x_{r}\right)_{j_{r}}}{\left(q^{1+k_{s}-k_{r}}x_{s}/x_{r}\right)_{j_{r}}}\prod_{r=1}^{n}\big\{ x_{r}^{-j_{r}}\left(aq^{1-k_{r}}/ex_{r}\right)_{j_{r}}^{-1}\big\}\\
 & \;\times\bigg(\frac{b}{e}\bigg)^{\boldsymbol{|j}|}\frac{(a)_{|\boldsymbol{j}|}(d/b)_{|\boldsymbol{j}|}}{(d)_{|\boldsymbol{j}|}}.
\end{aligned}
\]
The $a\mapsto\tau q^{|\boldsymbol{k}|},b\mapsto c,d\mapsto bc,e\mapsto\beta$
case of this identity is
\begin{equation}
\begin{aligned} & \sum_{{0\leq j_{r}\leq k_{r}\atop r=1,\cdots,n}}\prod_{1\leq r<s\leq n}\frac{1-q^{j_{r}-j_{s}}x_{r}/x_{s}}{1-x_{r}/x_{s}}\prod_{r,s=1}^{n}\frac{\left(q^{-k_{s}}x_{r}/x_{s}\right)_{j_{r}}}{\left(qx_{r}/x_{s}\right)_{j_{r}}}\\
 & \times\prod_{r=1}^{n}(\beta x_{r})_{j_{r}}^{-1}\frac{(\tau q^{|\boldsymbol{k}|})_{|\boldsymbol{j}|}(c)_{|\boldsymbol{j}|}}{(bc)_{|\boldsymbol{j}|}}q^{\sum_{r=1}^{n}rj_{r}}\\
 & =(\tau q^{|\boldsymbol{k}|})^{|\boldsymbol{k}|}\prod_{r=1}^{n}\frac{(\beta x_{r}q^{-|\boldsymbol{k}|}/\tau)_{k_{r}}}{(\beta x_{r})_{k_{r}}}\sum_{{0\leq j_{r}\leq k_{r}\atop r=1,\cdots,n}}\prod_{1\leq r<s\leq n}\frac{1-q^{k_{s}-k_{r}+j_{r}-j_{s}}x_{s}/x_{r}}{1-q^{k_{s}-k_{r}}x_{s}/x_{r}}\\
 & \;\times q^{|\boldsymbol{k}|\boldsymbol{|j}|-\sum_{r=1}^{n}k_{r}j_{r}-e_{2}(j_{1},\cdots,j_{n})+\sum_{r=1}^{n}rj_{r}}\\
 & \;\times\prod_{r,s=1}^{n}\frac{\left(q^{-k_{r}}x_{s}/x_{r}\right)_{j_{r}}}{\left(q^{1+k_{s}-k_{r}}x_{s}/x_{r}\right)_{j_{r}}}\prod_{r=1}^{n}\big\{ x_{r}^{-j_{r}}\left(\tau q^{1-k_{r}+|\boldsymbol{k}|}/\beta x_{r}\right)_{j_{r}}^{-1}\big\}\\
 & \;\times\bigg(\frac{c}{\beta}\bigg)^{\boldsymbol{|j}|}\frac{(\tau q^{|\boldsymbol{k}|})_{|\boldsymbol{j}|}(b)_{|\boldsymbol{j}|}}{(bc)_{|\boldsymbol{j}|}}.
\end{aligned}
\label{eq:9-4-1}
\end{equation}

It follows from \cite[eq.(I.8)]{GR} that
\[
(\beta x_{r}q^{-|\boldsymbol{k}|}/\tau)_{k_{r}}=\left(\frac{q^{1+|\boldsymbol{k}|-k_{r}}\tau}{\beta x_{r}}\right)_{k_{r}}\left(-\frac{\beta x_{r}q^{k_{r}-|\boldsymbol{k}|-1}}{\tau}\right)^{k_{r}}q^{-{k_{r} \choose 2}}.
\]
Then
\[
\prod_{r=1}^{n}(\beta x_{r}q^{-|\boldsymbol{k}|}/\tau)_{k_{r}}=\left(-\frac{\beta}{\tau}\right)^{|\boldsymbol{k}|}q^{\sum_{r=1}^{n}{k_{r} \choose 2}-|\boldsymbol{k}|^{2}}\prod_{r=1}^{n}\left(\frac{q^{1+|\boldsymbol{k}|-k_{r}}\tau}{\beta x_{r}}\right)_{k_{r}}\prod_{r=1}^{n}x_{r}^{k_{r}}.
\]
Substituting this identity into \eqref{eq:9-4-1} gives
\begin{equation}
\begin{aligned} & \sum_{{0\leq j_{r}\leq k_{r}\atop r=1,\cdots,n}}\prod_{1\leq r<s\leq n}\frac{1-q^{j_{r}-j_{s}}x_{r}/x_{s}}{1-x_{r}/x_{s}}\prod_{r,s=1}^{n}\frac{\left(q^{-k_{s}}x_{r}/x_{s}\right)_{j_{r}}}{\left(qx_{r}/x_{s}\right)_{j_{r}}}\\
 & \times\prod_{r=1}^{n}(\beta x_{r})_{j_{r}}^{-1}\frac{(\tau q^{|\boldsymbol{k}|})_{|\boldsymbol{j}|}(c)_{|\boldsymbol{j}|}}{(bc)_{|\boldsymbol{j}|}}q^{\sum_{r=1}^{n}rj_{r}}\\
 & =\prod_{r=1}^{n}\frac{(q^{1+|\boldsymbol{k}|-k_{r}}\tau/\beta x_{r})_{k_{r}}}{(\beta x_{r})_{k_{r}}}\left(-\beta\right)^{|\boldsymbol{k}|}q^{\sum_{r=1}^{n}{k_{r} \choose 2}}\prod_{r=1}^{n}x_{r}^{k_{r}}\\
 & \;\times\sum_{{0\leq j_{r}\leq k_{r}\atop r=1,\cdots,n}}\prod_{1\leq r<s\leq n}\frac{1-q^{k_{s}-k_{r}+j_{r}-j_{s}}x_{s}/x_{r}}{1-q^{k_{s}-k_{r}}x_{s}/x_{r}}\prod_{r,s=1}^{n}\frac{\left(q^{-k_{r}}x_{s}/x_{r}\right)_{j_{r}}}{\left(q^{1+k_{s}-k_{r}}x_{s}/x_{r}\right)_{j_{r}}}\\
 & \;\times\prod_{r=1}^{n}\big\{ x_{r}^{-j_{r}}\left(\tau q^{1-k_{r}+|\boldsymbol{k}|}/\beta x_{r}\right)_{j_{r}}^{-1}\big\}\cdot q^{|\boldsymbol{k}|\boldsymbol{|j}|-\sum_{r=1}^{n}k_{r}j_{r}-e_{2}(j_{1},\cdots,j_{n})+\sum_{r=1}^{n}rj_{r}}\\
 & \;\times\bigg(\frac{c}{\beta}\bigg)^{\boldsymbol{|j}|}\frac{(\tau q^{|\boldsymbol{k}|})_{|\boldsymbol{j}|}(b)_{|\boldsymbol{j}|}}{(bc)_{|\boldsymbol{j}|}}.
\end{aligned}
\label{eq:9-6}
\end{equation}

Taking $N_{r}\mapsto k_{r}-j_{r},x_{r}\mapsto q^{j_{r}}x_{r},b\mapsto\tau q^{|\boldsymbol{k}|+|\boldsymbol{j}|},c\mapsto\tau q^{|\boldsymbol{j}|}$
in \eqref{eq:8-1} we have
\begin{equation}
\begin{aligned} & \sum_{{0\leq m_{r}\leq k_{r}-j_{r}\atop r=1,2,\cdots,n}}\prod_{1\leq r<s\leq n}\frac{1-q^{m_{r}-m_{s}+j_{r}-j_{s}}x_{r}/x_{s}}{1-q^{j_{r}-j_{s}}x_{r}/x_{s}}q^{\sum_{r=1}^{n}rm_{r}}\\
 & \times\prod_{r,s=1}^{n}\frac{(q^{-k_{s}+j_{r}}x_{r}/x_{s})_{m_{r}}}{(q^{1+j_{r}-j_{s}}x_{r}/x_{s})_{m_{r}}}\frac{(\tau q^{|\boldsymbol{k}|+|\boldsymbol{j}|})_{|\boldsymbol{m}|}}{(\tau q^{|\boldsymbol{j}|})_{|\boldsymbol{m}|}}\\
 & =\frac{(q^{-|\boldsymbol{k}|})_{|\boldsymbol{k}|-|\boldsymbol{j}|}}{(\tau q^{|\boldsymbol{j}|})_{|\boldsymbol{k}|-|\boldsymbol{j}|}}(\tau q^{|\boldsymbol{k}|+|\boldsymbol{j}|})^{|\boldsymbol{k}|-|\boldsymbol{j}|}\\
 & =\frac{(q)_{|\boldsymbol{k}|}(-\tau)^{|\boldsymbol{k}|-|\boldsymbol{j}|}}{(\tau q^{|\boldsymbol{j}|})_{|\boldsymbol{k}|-|\boldsymbol{j}|}(q)_{|\boldsymbol{j}|}}q^{{|\boldsymbol{k}| \choose 2}-{|\boldsymbol{j}| \choose 2}},
\end{aligned}
\label{eq:9-3-1}
\end{equation}
where we have used the identity \cite[eq.(I.14)]{GR} in the last
equality.

We first set $b\mapsto0,a\mapsto\tau,$ and then take 
\[
K(y)\mapsto1-y_{1}y_{2}\cdots y_{n}\tau
\]
 and
\[
\beta(\boldsymbol{j})\mapsto\frac{(q,c)_{|\boldsymbol{j}|}}{(bc)_{|\boldsymbol{j}|}}\tau^{|\boldsymbol{j}|}\prod_{r=1}^{n}(\beta x_{r})_{j_{r}}^{-1}\prod_{r,s=1}^{n}(qx_{r}/x_{s})_{j_{r}}^{-1}
\]
 in the identity \eqref{eq:t1-1} of Theorem \ref{t1} to get

\[
\begin{aligned} & (1-y_{1}y_{2}\cdots y_{n}\tau)\sum_{{j_{r}\geq0\atop r=1,2,\cdots,n}}\prod_{1\leq r<s\leq n}\frac{1-q^{j_{r}-j_{s}}x_{r}/x_{s}}{1-x_{r}/x_{s}}\prod_{r,s=1}^{n}\frac{(x_{r}/x_{s}y_{s})_{j_{r}}}{(qx_{r}/x_{s})_{j_{r}}}\\
 & \quad\times(y_{1}y_{2}\cdots y_{n}\tau)^{|\boldsymbol{j}|}q^{\sum_{r=1}^{n}(r-1)j_{r}}\frac{(q,c)_{|\boldsymbol{j}|}}{(bc)_{|\boldsymbol{j}|}}\prod_{r=1}^{n}(\beta x_{r})_{j_{r}}^{-1}\\
 & =\sum_{{k_{r}\geq0\atop r=1,2,\cdots,n}}\prod_{1\leq r<s\leq n}\frac{1-q^{k_{r}-k_{s}}x_{r}/x_{s}}{1-x_{r}/x_{s}}\prod_{r,s=1}^{n}\frac{(x_{r}/x_{s}y_{s})_{k_{r}}}{(qx_{r}/x_{s})_{k_{r}}}\\
 & \times\frac{(1-\tau q^{2|\boldsymbol{k}|})(\tau)_{|\boldsymbol{k}|}}{(\tau qy_{1}y_{2}\cdots y_{n})_{|\boldsymbol{k}|}}(y_{1}y_{2}\cdots y_{n})^{|\boldsymbol{k}|}q^{\sum_{r=1}^{n}(r-1)k_{r}}\\
 & \times\sum_{{0\leq j_{r}\leq k_{r}\atop r=1,2,\cdots,n}}\prod_{1\leq r<s\leq n}\frac{1-q^{j_{r}-j_{s}}x_{r}/x_{s}}{1-x_{r}/x_{s}}\prod_{r,s=1}^{n}\frac{(q^{-k_{s}}x_{r}/x_{s})_{j_{r}}}{(qx_{r}/x_{s})_{j_{r}}}\\
 & \;\times\frac{(\tau q^{|\boldsymbol{k}|},q,c)_{|\boldsymbol{j}|}}{(\tau,bc)_{|\boldsymbol{j}|}}(-\tau)^{|\boldsymbol{j}|}q^{\sum_{r=1}^{n}rj_{r}+{|\boldsymbol{j}| \choose 2}}\prod_{r=1}^{n}(\beta x_{r})_{j_{r}}^{-1}\\
 & \;\times\sum_{{0\leq m_{r}\leq k_{r}-j_{r}\atop r=1,2,\cdots,n}}\prod_{1\leq r<s\leq n}\frac{1-q^{m_{r}-m_{s}+j_{r}-j_{s}}x_{r}/x_{s}}{1-q^{j_{r}-j_{s}}x_{r}/x_{s}}q^{\sum_{r=1}^{n}rm_{r}}\\
 & \quad\times\prod_{r,s=1}^{n}\frac{(q^{-k_{s}+j_{r}}x_{r}/x_{s})_{m_{r}}}{(q^{1+j_{r}-j_{s}}x_{r}/x_{s})_{m_{r}}}\frac{(\tau q^{|\boldsymbol{k}|+|\boldsymbol{j}|})_{|\boldsymbol{m}|}}{(\tau q^{|\boldsymbol{j}|})_{|\boldsymbol{m}|}}.
\end{aligned}
\]
Then the result follows by substituting \eqref{eq:9-3-1} into the
inner-most sum $\sum_{{0\leq m_{r}\leq k_{r}-j_{r}\atop r=1,2,\cdots,n}}$
on the right side of the above identity, then applying \eqref{eq:9-6}
to the second sum $\sum_{{0\leq j_{r}\leq k_{r}\atop r=1,2,\cdots,n}}$
and finally multiplying both sides by $\frac{1}{1-y_{1}y_{2}\cdots y_{n}\tau}.$
\qed

\section{An $A_{n}$ extension of Andrews' expansion formula}

Applying a Bailey pair in \cite[Theorem 3]{A86}, Andrews \cite[Theorem 5]{A86}
proved the following celebrated expansion formula\footnote{Note that the representations for $A_{n}$ in \cite[(4.2) and (4.5)]{A86}
are equal. Therefore, for brevity, we choose the former representation
for $A_{n}$ in \cite[Theorem 5]{A86}.} for an infinite product:
\begin{equation}
\frac{(q,bq/y,aq)_{\infty}}{(bq,q/y,aq/y)_{\infty}}=\sum_{n=0}^{\infty}\frac{(1-aq^{2n})(a,y)_{n}(a/y)^{n}q^{n^{2}}}{(1-a)(aq/y,bq)_{n}}\sum_{j=0}^{n}\frac{(b)_{j}}{(q)_{j}}(1/a)^{j}q^{j(1-n)}.\label{eq:81-1}
\end{equation}
The formula \eqref{eq:81-1} is very useful and was used by Andrews
to deduce Hecke modular form identities and generating functions for
the number of representations of an integer as a sum of three squares
and the number of representations of an integer as a sum of three
triangular numbers.

In this section, we apply the identity in Theorem \ref{t21-1} to
give an $A_{n}$ extension of the identity \eqref{eq:81-1}.
\begin{thm}
Suppose that none of the denominators in the following identity vanishes,
then
\begin{equation}
\begin{aligned} & \frac{(q,aq,bq/y_{1}\cdots y_{n})_{\infty}}{(aq/y_{1}\cdots y_{n},bq,q/y_{1}\cdots y_{n})_{\infty}}\\
 & =\sum_{{k_{r}\geq0\atop r=1,2,\cdots,n}}\prod_{1\leq r<s\leq n}\frac{1-q^{k_{r}-k_{s}}x_{r}/x_{s}}{1-x_{r}/x_{s}}\prod_{r,s=1}^{n}\frac{(y_{s}x_{r}/x_{s})_{k_{r}}}{(qx_{r}/x_{s})_{k_{r}}}\\
 & \times\frac{(1-aq^{2|\boldsymbol{k}|})(a,q)_{|\boldsymbol{k}|}(a/y_{1}\cdots y_{n})^{|\boldsymbol{k}|}}{(1-a)(aq/y_{1}\cdots y_{n},bq)_{|\boldsymbol{k}|}}q^{\sum_{r=1}^{n}(r-1)k_{r}+|\boldsymbol{k}|^{2}}\\
 & \times\sum_{{0\leq j_{r}\leq k_{r}\atop r=1,\cdots,n}}\prod_{1\leq r<s\leq n}\frac{1-q^{j_{r}-j_{s}}x_{r}/x_{s}}{1-x_{r}/x_{s}}\cdot q^{\sum_{r=1}^{n}rj_{r}-|\boldsymbol{k}||\boldsymbol{j}|}\\
 & \:\times\prod_{r,s=1}^{n}\frac{\left(q^{-k_{s}}x_{r}/x_{s}\right)_{j_{r}}}{\left(qx_{r}/x_{s}\right)_{j_{r}}}\frac{(b)_{|\boldsymbol{j}|}(1/a)^{|\boldsymbol{j}|}}{(q^{-|\boldsymbol{k}|})_{|\boldsymbol{j}|}}.
\end{aligned}
\label{eq:81-2}
\end{equation}
\end{thm}
\noindent{\it Proof.} We first prove the following identity: if $k_{1},k_{2},\cdots,k_{n}$
are non-negative integers, then
\begin{equation}
\begin{aligned} & \sum_{{0\leq j_{r}\leq k_{r}\atop r=1,\cdots,n}}\prod_{1\leq r<s\leq n}\frac{1-q^{j_{r}-j_{s}}x_{r}/x_{s}}{1-x_{r}/x_{s}}\\
 & \times\prod_{r,s=1}^{n}\frac{\left(q^{-k_{s}}x_{r}/x_{s}\right)_{j_{r}}}{\left(qx_{r}/x_{s}\right)_{j_{r}}}\frac{(aq^{|\boldsymbol{k}|},q)_{|\boldsymbol{j}|}}{(cq)_{|\boldsymbol{j}|}}q^{\sum_{r=1}^{n}rj_{r}}\\
 & =a^{|\boldsymbol{k}|}q^{|\boldsymbol{k}|^{2}}\frac{(q)_{|\boldsymbol{k}|}}{(cq)_{|\boldsymbol{k}|}}\sum_{{0\leq j_{r}\leq k_{r}\atop r=1,\cdots,n}}\prod_{1\leq r<s\leq n}\frac{1-q^{j_{r}-j_{s}}x_{r}/x_{s}}{1-x_{r}/x_{s}}
\end{aligned}
\label{eq:HB-1}
\end{equation}

\[
\times\prod_{r,s=1}^{n}\frac{\left(q^{-k_{s}}x_{r}/x_{s}\right)_{j_{r}}}{\left(qx_{r}/x_{s}\right)_{j_{r}}}\frac{(c)_{|\boldsymbol{j}|}(1/a)^{|\boldsymbol{j}|}}{(q^{-|\boldsymbol{k}|})_{|\boldsymbol{j}|}}q^{\sum_{r=1}^{n}rj_{r}-|\boldsymbol{k}||\boldsymbol{j}|}.
\]

Recall an $A_{n}$ Sears transformation formula between (6.5a) and
(6.5c) in \cite[Theorem 6.5]{ML}: Suppose that $def=abcq^{1-|\boldsymbol{k}|},$
then
\begin{equation}
\begin{aligned} & \sum_{{0\leq j_{r}\leq k_{r}\atop r=1,\cdots,n}}\prod_{1\leq r<s\leq n}\frac{1-q^{j_{r}-j_{s}}x_{r}/x_{s}}{1-x_{r}/x_{s}}\prod_{r,s=1}^{n}\frac{\left(q^{-k_{s}}x_{r}/x_{s}\right)_{j_{r}}}{\left(qx_{r}/x_{s}\right)_{j_{r}}}\\
 & \times\prod_{r=1}^{n}\frac{\left(cx_{r}\right)_{j_{r}}}{(ex_{r})_{j_{r}}}\frac{(a)_{|\boldsymbol{j}|}(b)_{|\boldsymbol{j}|}}{(d)_{|\boldsymbol{j}|}(f)_{|\boldsymbol{j}|}}\cdot q^{\sum_{r=1}^{n}rj_{r}}\\
 & =a^{|\boldsymbol{k}|}\frac{(f/a)_{|\boldsymbol{k}|}}{(f)_{|\boldsymbol{k}|}}\prod_{r=1}^{n}\frac{(ex_{r}/a)_{k_{r}}}{(ex_{r})_{k_{r}}}\\
 & \times\sum_{{0\leq j_{r}\leq k_{r}\atop r=1,\cdots,n}}\prod_{1\leq r<s\leq n}\frac{1-q^{k_{s}-k_{r}+j_{r}-j_{s}}x_{s}/x_{r}}{1-q^{k_{s}-k_{r}}x_{s}/x_{r}}q^{\sum_{r=1}^{n}rj_{r}}\\
 & \;\times\prod_{r,s=1}^{n}\frac{\left(q^{-k_{r}}x_{s}/x_{r}\right)_{j_{r}}}{\left(q^{1+k_{s}-k_{r}}x_{s}/x_{r}\right)_{j_{r}}}\prod_{r=1}^{n}\frac{\left(dq^{|\boldsymbol{k}|-k_{r}}/cx_{r}\right)_{j_{r}}}{\left(aq^{1-k_{r}}/ex_{r}\right)_{j_{r}}}\\
 & \;\times\frac{(a)_{|\boldsymbol{j}|}(d/b)_{|\boldsymbol{j}|}}{(d)_{|\boldsymbol{j}|}(aq^{1-|\boldsymbol{k}|}/f)_{|\boldsymbol{j}|}}.
\end{aligned}
\label{eq:ML-1}
\end{equation}
Substituting $e=abcq^{1-|\boldsymbol{k}|}/df$ into \eqref{eq:ML-1},
setting $c\rightarrow0,$ and then replacing $f$ by $e,$ and $a$
by $c,$ we obtain
\begin{equation}
\begin{aligned} & \sum_{{0\leq j_{r}\leq k_{r}\atop r=1,\cdots,n}}\prod_{1\leq r<s\leq n}\frac{1-q^{j_{r}-j_{s}}x_{r}/x_{s}}{1-x_{r}/x_{s}}\\
 & \times\prod_{r,s=1}^{n}\frac{\left(q^{-k_{s}}x_{r}/x_{s}\right)_{j_{r}}}{\left(qx_{r}/x_{s}\right)_{j_{r}}}\frac{(b)_{|\boldsymbol{j}|}(c)_{|\boldsymbol{j}|}}{(d)_{|\boldsymbol{j}|}(e)_{|\boldsymbol{j}|}}q^{\sum_{r=1}^{n}rj_{r}}\\
 & =c^{|\boldsymbol{k}|}\frac{(e/c)_{|\boldsymbol{k}|}}{(e)_{|\boldsymbol{k}|}}\sum_{{0\leq j_{r}\leq k_{r}\atop r=1,\cdots,n}}\prod_{1\leq r<s\leq n}\frac{1-q^{k_{s}-k_{r}+j_{r}-j_{s}}x_{s}/x_{r}}{1-q^{k_{s}-k_{r}}x_{s}/x_{r}}\\
 & \times\prod_{r,s=1}^{n}\frac{\left(q^{-k_{r}}x_{s}/x_{r}\right)_{j_{r}}}{\left(q^{1+k_{s}-k_{r}}x_{s}/x_{r}\right)_{j_{r}}}\cdot\left(\frac{bq}{e}\right)^{|\boldsymbol{j}|}q^{\sum_{r=1}^{n}(r-1)j_{r}}\\
 & \;\times\frac{(c)_{|\boldsymbol{j}|}(d/b)_{|\boldsymbol{j}|}}{(d)_{|\boldsymbol{j}|}(cq^{1-|\boldsymbol{k}|}/e)_{|\boldsymbol{j}|}}.
\end{aligned}
\label{eq:H-1}
\end{equation}
Replace $b$ by $d/b,$ $e$ by $cq^{1-|\boldsymbol{k}|}/e,x_{r}$
by $x_{r}q^{-k_{r}}$ in \eqref{eq:H-1} to arrive at 
\begin{equation}
\begin{aligned} & \sum_{{0\leq j_{r}\leq k_{r}\atop r=1,\cdots,n}}\prod_{1\leq r<s\leq n}\frac{1-q^{j_{r}-j_{s}}x_{s}/x_{r}}{1-x_{s}/x_{r}}\prod_{r,s=1}^{n}\frac{\left(q^{-k_{s}}x_{s}/x_{r}\right)_{j_{r}}}{\left(qx_{s}/x_{r}\right)_{j_{r}}}\\
 & \times\frac{(b)_{|\boldsymbol{j}|}(c)_{|\boldsymbol{j}|}}{(d)_{|\boldsymbol{j}|}(e)_{|\boldsymbol{j}|}}\cdot\left(\frac{deq^{|\boldsymbol{k}|}}{bc}\right)^{|\boldsymbol{j}|}q^{\sum_{r=1}^{n}(r-1)j_{r}}
\end{aligned}
\label{eq:H-2}
\end{equation}
\begin{align*}
 & =\frac{(e/c)_{|\boldsymbol{k}|}}{(e)_{|\boldsymbol{k}|}}\sum_{{0\leq j_{r}\leq k_{r}\atop r=1,\cdots,n}}\prod_{1\leq r<s\leq n}\frac{1-q^{j_{r}-j_{s}-k_{r}+k_{s}}x_{r}/x_{s}}{1-q^{-k_{r}+k_{s}}x_{r}/x_{s}}\\
 & \times\prod_{r,s=1}^{n}\frac{\left(q^{-k_{r}}x_{r}/x_{s}\right)_{j_{r}}}{\left(q^{1-k_{r}+k_{s}}x_{r}/x_{s}\right)_{j_{r}}}\frac{(c)_{|\boldsymbol{j}|}(d/b)_{|\boldsymbol{j}|}}{(d)_{|\boldsymbol{j}|}(cq^{1-|\boldsymbol{k}|}/e)_{|\boldsymbol{j}|}}q^{\sum_{r=1}^{n}rj_{r}}.
\end{align*}
Let $b\mapsto aq^{|\boldsymbol{k}|},d\mapsto1/t,e\mapsto cq$ in \eqref{eq:H-2}
we have
\[
\begin{aligned} & \sum_{{0\leq j_{r}\leq k_{r}\atop r=1,\cdots,n}}\prod_{1\leq r<s\leq n}\frac{1-q^{j_{r}-j_{s}}x_{s}/x_{r}}{1-x_{s}/x_{r}}\prod_{r,s=1}^{n}\frac{\left(q^{-k_{s}}x_{s}/x_{r}\right)_{j_{r}}}{\left(qx_{s}/x_{r}\right)_{j_{r}}}\\
 & \times\frac{(c)_{|\boldsymbol{j}|}(aq^{|\boldsymbol{k}|})_{|\boldsymbol{j}|}}{(1/t)_{|\boldsymbol{j}|}(cq)_{|\boldsymbol{j}|}}\cdot\left(\frac{q}{at}\right)^{|\boldsymbol{j}|}q^{\sum_{r=1}^{n}(r-1)j_{r}}\\
 & =\frac{(q)_{|\boldsymbol{k}|}}{(cq)_{|\boldsymbol{k}|}}\sum_{{0\leq j_{r}\leq k_{r}\atop r=1,\cdots,n}}\prod_{1\leq r<s\leq n}\frac{1-q^{j_{r}-j_{s}-k_{r}+k_{s}}x_{r}/x_{s}}{1-q^{-k_{r}+k_{s}}x_{r}/x_{s}}\\
 & \times\prod_{r,s=1}^{n}\frac{\left(q^{-k_{r}}x_{r}/x_{s}\right)_{j_{r}}}{\left(q^{1-k_{r}+k_{s}}x_{r}/x_{s}\right)_{j_{r}}}\frac{(c)_{|\boldsymbol{j}|}(1/atq^{|k|})_{|\boldsymbol{j}|}}{(1/t)_{|\boldsymbol{j}|}(q^{-|\boldsymbol{k}|})_{|\boldsymbol{j}|}}q^{\sum_{r=1}^{n}rj_{r}}
\end{aligned}
\]
Taking the limit $t\rightarrow0$ on both sides of this identity gives
\begin{equation}
\begin{aligned} & \sum_{{0\leq j_{r}\leq k_{r}\atop r=1,\cdots,n}}\prod_{1\leq r<s\leq n}\frac{1-q^{j_{r}-j_{s}}x_{s}/x_{r}}{1-x_{s}/x_{r}}\prod_{r,s=1}^{n}\frac{\left(q^{-k_{s}}x_{s}/x_{r}\right)_{j_{r}}}{\left(qx_{s}/x_{r}\right)_{j_{r}}}\\
 & \times\frac{(aq^{|\boldsymbol{k}|},c)_{|\boldsymbol{j}|}}{(cq)_{|\boldsymbol{j}|}}\left(-\frac{1}{a}\right)^{|\boldsymbol{j}|}q^{\sum_{r=1}^{n}rj_{r}-{|\boldsymbol{j}| \choose 2}}\\
 & =\frac{(q)_{|\boldsymbol{k}|}}{(cq)_{|\boldsymbol{k}|}}\sum_{{0\leq j_{r}\leq k_{r}\atop r=1,\cdots,n}}\prod_{1\leq r<s\leq n}\frac{1-q^{j_{r}-j_{s}-k_{r}+k_{s}}x_{r}/x_{s}}{1-q^{-k_{r}+k_{s}}x_{r}/x_{s}}\\
 & \times\prod_{r,s=1}^{n}\frac{\left(q^{-k_{r}}x_{r}/x_{s}\right)_{j_{r}}}{\left(q^{1-k_{r}+k_{s}}x_{r}/x_{s}\right)_{j_{r}}}\frac{(c)_{|\boldsymbol{j}|}(1/a)^{|\boldsymbol{j}|}}{(q^{-|\boldsymbol{k}|})_{|\boldsymbol{j}|}}q^{\sum_{r=1}^{n}rj_{r}-|\boldsymbol{k}||\boldsymbol{j}|}
\end{aligned}
\label{eq:He-1}
\end{equation}
We set $c\mapsto aq^{|\boldsymbol{k}|},b\mapsto c,d\mapsto cq,e\mapsto1/t$
in \eqref{eq:H-2} to get
\[
\begin{aligned} & \sum_{{0\leq j_{r}\leq k_{r}\atop r=1,\cdots,n}}\prod_{1\leq r<s\leq n}\frac{1-q^{j_{r}-j_{s}}x_{s}/x_{r}}{1-x_{s}/x_{r}}\prod_{r,s=1}^{n}\frac{\left(q^{-k_{s}}x_{s}/x_{r}\right)_{j_{r}}}{\left(qx_{s}/x_{r}\right)_{j_{r}}}\\
 & \times\frac{(aq^{|\boldsymbol{k}|})_{|\boldsymbol{j}|}(c)_{|\boldsymbol{j}|}}{(cq)_{|\boldsymbol{j}|}(1/t)_{|\boldsymbol{j}|}}\cdot\left(\frac{q}{at}\right)^{|\boldsymbol{j}|}q^{\sum_{r=1}^{n}(r-1)j_{r}}\\
 & =\frac{(1/atq^{|\boldsymbol{k}|})_{|\boldsymbol{k}|}}{(1/t)_{|\boldsymbol{k}|}}\sum_{{0\leq j_{r}\leq k_{r}\atop r=1,\cdots,n}}\prod_{1\leq r<s\leq n}\frac{1-q^{j_{r}-j_{s}-k_{r}+k_{s}}x_{r}/x_{s}}{1-q^{-k_{r}+k_{s}}x_{r}/x_{s}}\\
 & \times\prod_{r,s=1}^{n}\frac{\left(q^{-k_{r}}x_{r}/x_{s}\right)_{j_{r}}}{\left(q^{1-k_{r}+k_{s}}x_{r}/x_{s}\right)_{j_{r}}}\frac{(aq^{|\boldsymbol{k}|})_{|\boldsymbol{j}|}(q)_{|\boldsymbol{j}|}}{(cq)_{|\boldsymbol{j}|}(aqt)_{|\boldsymbol{j}|}}q^{\sum_{r=1}^{n}rj_{r}}.
\end{aligned}
\]
Taking the limit $t\rightarrow0$ on both sides of this identity yields
\begin{equation}
\begin{aligned} & \sum_{{0\leq j_{r}\leq k_{r}\atop r=1,\cdots,n}}\prod_{1\leq r<s\leq n}\frac{1-q^{j_{r}-j_{s}}x_{s}/x_{r}}{1-x_{s}/x_{r}}\prod_{r,s=1}^{n}\frac{\left(q^{-k_{s}}x_{s}/x_{r}\right)_{j_{r}}}{\left(qx_{s}/x_{r}\right)_{j_{r}}}\\
 & \times\frac{(aq^{|\boldsymbol{k}|},c)_{|\boldsymbol{j}|}}{(cq)_{|\boldsymbol{j}|}}\left(-\frac{1}{a}\right)^{|\boldsymbol{j}|}q^{\sum_{r=1}^{n}rj_{r}-{|\boldsymbol{j}| \choose 2}}\\
 & =(1/a)^{|\boldsymbol{k}|}q^{-|\boldsymbol{k}|^{2}}\sum_{{0\leq j_{r}\leq k_{r}\atop r=1,\cdots,n}}\prod_{1\leq r<s\leq n}\frac{1-q^{j_{r}-j_{s}-k_{r}+k_{s}}x_{r}/x_{s}}{1-q^{-k_{r}+k_{s}}x_{r}/x_{s}}\\
 & \times\prod_{r,s=1}^{n}\frac{\left(q^{-k_{r}}x_{r}/x_{s}\right)_{j_{r}}}{\left(q^{1-k_{r}+k_{s}}x_{r}/x_{s}\right)_{j_{r}}}\frac{(aq^{|\boldsymbol{k}|})_{|\boldsymbol{j}|}(q)_{|\boldsymbol{j}|}}{(cq)_{|\boldsymbol{j}|}}q^{\sum_{r=1}^{n}rj_{r}}
\end{aligned}
\label{eq:He-2}
\end{equation}
Then the identity \eqref{eq:HB-1} follows easily by combining \eqref{eq:He-1}
and \eqref{eq:He-2} and then replacing $x_{r}$ by $x_{r}q^{k_{r}}$
in the resulting identity.

We now show \eqref{eq:81-2}. Let $c=d=\gamma=0$ in the identity
of Theorem \ref{t21-1} and then replace $\beta$ by $q/\alpha$ in
the resulting identity. We get
\[
\begin{aligned} & \frac{(q,\alpha q,\alpha a_{1}\cdots a_{n}b/q^{n})_{\infty}}{(\alpha a_{1}\cdots a_{n}q^{1-n},\alpha b,a_{1}\cdots a_{n}q^{1-n})_{\infty}}\\
 & =\sum_{{k_{r}\geq0\atop r=1,2,\cdots,n}}\prod_{1\leq r<s\leq n}\frac{1-q^{k_{r}-k_{s}}x_{r}/x_{s}}{1-x_{r}/x_{s}}\prod_{r,s=1}^{n}\frac{(qx_{r}/x_{s}a_{s})_{k_{r}}}{(qx_{r}/x_{s})_{k_{r}}}\\
 & \times\frac{(1-\alpha q^{2|\boldsymbol{k}|})(\alpha)_{|\boldsymbol{k}|}}{(1-\alpha)(\alpha a_{1}\cdots a_{n}q^{1-n})_{|\boldsymbol{k}|}}(a_{1}\cdots a_{n})^{|\boldsymbol{k}|}q^{\sum_{r=1}^{n}(r-1)k_{r}-n|\boldsymbol{k}|}\\
 & \;\times\sum_{{0\leq m_{r}\leq k_{r}\atop r=1,2,\cdots,n}}\prod_{1\leq r<s\leq n}\frac{1-q^{m_{r}-m_{s}}x_{r}/x_{s}}{1-x_{r}/x_{s}}\\
 & \quad\times\prod_{r,s=1}^{n}\frac{(q^{-k_{s}}x_{r}/x_{s})_{m_{r}}}{(qx_{r}/x_{s})_{m_{r}}}\frac{(\alpha q^{|\boldsymbol{k}|},q)_{|\boldsymbol{m}|}}{(\alpha b)_{|\boldsymbol{m}|}}q^{\sum_{r=1}^{n}rm_{r}}.
\end{aligned}
\]
Replacing $a$ by $\alpha,$ $c$ by $\alpha b/q$ in the identity
\eqref{eq:HB-1} and then substituting the resulting identity into
the above identity we have 
\[
\begin{aligned} & \frac{(q,\alpha q,\alpha a_{1}\cdots a_{n}b/q^{n})_{\infty}}{(\alpha a_{1}\cdots a_{n}q^{1-n},\alpha b,a_{1}\cdots a_{n}q^{1-n})_{\infty}}\\
 & =\sum_{{k_{r}\geq0\atop r=1,2,\cdots,n}}\prod_{1\leq r<s\leq n}\frac{1-q^{k_{r}-k_{s}}x_{r}/x_{s}}{1-x_{r}/x_{s}}\prod_{r,s=1}^{n}\frac{(qx_{r}/x_{s}a_{s})_{k_{r}}}{(qx_{r}/x_{s})_{k_{r}}}\\
 & \times\frac{(1-\alpha q^{2|\boldsymbol{k}|})(\alpha,q)_{|\boldsymbol{k}|}(\alpha a_{1}\cdots a_{n})^{|\boldsymbol{k}|}}{(1-\alpha)(\alpha a_{1}\cdots a_{n}q^{1-n},\alpha b)_{|\boldsymbol{k}|}}q^{\sum_{r=1}^{n}(r-1)k_{r}-n|\boldsymbol{k}|+|\boldsymbol{k}|^{2}}\\
 & \times\sum_{{0\leq j_{r}\leq k_{r}\atop r=1,\cdots,n}}\prod_{1\leq r<s\leq n}\frac{1-q^{j_{r}-j_{s}}x_{r}/x_{s}}{1-x_{r}/x_{s}}\cdot q^{\sum_{r=1}^{n}rj_{r}-|\boldsymbol{k}||\boldsymbol{j}|}\\
 & \:\times\prod_{r,s=1}^{n}\frac{\left(q^{-k_{s}}x_{r}/x_{s}\right)_{j_{r}}}{\left(qx_{r}/x_{s}\right)_{j_{r}}}\frac{(\alpha b/q)_{|\boldsymbol{j}|}(1/\alpha)^{|\boldsymbol{j}|}}{(q^{-|\boldsymbol{k}|})_{|\boldsymbol{j}|}}.
\end{aligned}
\]
Then the identity \eqref{eq:81-2} follows readily by replacing $a_{i}$
by $q/y_{i}$ for $i=1,2,\cdots,n,$ $b$ by $bq/a,$ and $\alpha$
by $a$ in this identity. \qed

\end{document}